\documentclass[11pt,a4paper]{amsart}
\usepackage[left=2.45cm,right=2.45cm,top=2cm,bottom=2cm, marginparwidth=4cm, marginparsep=0.5cm]{geometry} 
\usepackage{amsmath,amssymb,mathtools,mathdots} 
\usepackage{amsthm}
\usepackage{graphicx}
\usepackage{mathrsfs}
\usepackage{bbm}
\usepackage{fancyhdr}

\usepackage{marginnote} 
\usepackage[utf8]{inputenc}         
\usepackage[T1]{fontenc}            %
\usepackage[final]{microtype} 
\usepackage{url}                    

\usepackage{tikz-cd}                   
\usetikzlibrary{arrows}                 
\usepackage{adjustbox}
\usepackage[]{color} 
\definecolor{azure}{rgb}{0.0, 0.5, 1.0}
\definecolor{awesome}{rgb}{1.0, 0.13, 0.32}
\usepackage{framed} 
\usepackage[hidelinks]{hyperref}
\hypersetup{
  colorlinks   = true, 
  urlcolor     = azure, 
  linkcolor    = azure, 
  citecolor    = awesome 
}
\usepackage[capitalise]{cleveref}

\tikzset{
commutative diagrams/.cd,
arrow style=tikz,
diagrams={>=latex}} 

\usepackage{enumitem} 
\setlist[itemize]{noitemsep, nolistsep}
\setlist[enumerate]{noitemsep, nolistsep}

\newtheorem{theorem}{Theorem}[section]
\newtheorem{proposition}[theorem]{Proposition}

\newtheorem{lm}[theorem]{Lemma}
\newtheorem{corollary}[theorem]{Corollary}

\theoremstyle{definition}
\newtheorem{definition}[theorem]{Definition}
\newtheorem{example}[theorem]{Example}

\theoremstyle{remark}

\newcommand{\m}{\mathfrak{m}}

\newcommand{\C}{\mathbb{C}}

\DeclareMathOperator\codim{codim}

\newcommand{\Or}{\mathcal{O}}

\newcommand{\Jac}{\mathrm{Jac}}
\newcommand{\cid}{\mathrm{cid}}

\newcommand{\ord}{\mathrm{ord}}
\title [The complete intersection discrepancy of a curve II: Families of curves]{The complete intersection discrepancy of a curve II: Families of curves}
\author{Andrei Bengu\textcommabelow{s}-Lasnier}
\author{Terence Gaffney}
\author{Antoni Rangachev}

\begin{document}

\pagestyle{fancy}
\fancyhf{}
\fancyhead[CE]{\textsc{Andrei Bengu\textcommabelow s-Lasnier,
Terence Gaffney and Antoni Rangachev}}
\fancyhead[CO]{\textsc{The complete intersection discrepancy of a curve II}}
\fancyhead[LE,RO]{\thepage}


\address{Andrei Bengu\textcommabelow s-Lasnier\\
Institute of Mathematics and Informatics\\
Bulgarian Academy of Sciences\\
Akad. G. Bonchev, Sofia 1113, Bulgaria \\
email:abengus@math.bas.bg}

\address{Terence Gaffney\\
Department of Mathematics\\
Northeastern University\\
Boston, MA 02115, United States \\
email: t.gaffney@northeastern.edu}

\address{Antoni Rangachev\\
Institut de Math\'ematiques de Jussieu-Paris Rive Gauche \\
Centre national de la recherche scientifique (CNRS) \\
Paris, France\\
Institute of Mathematics and Informatics\\
Bulgarian Academy of Sciences\\
email: rangachev@imj-prg.fr}

\begin{abstract}
We study equisingularity of families of reduced curves over smooth parameter spaces of arbitrary positive dimension, using the difference between two analytic invariants of a curve singularity: the multiplicity of its Jacobian ideal and its complete intersection discrepancy. This difference provides a fiberwise multiplicity criterion for Whitney equisingularity. We prove that Whitney equisingularity (equivalently, strong simultaneous resolution) is characterized by equidimensionality of the fibers of the exceptional locus of either the relative conormal space or the relative Nash blowup. We further show that this condition is equivalent to the emptiness of the relative polar variety of smallest dimension. In addition, we establish that the Milnor number of a reduced curve is Zariski upper semicontinuous. As an application, we show that the constancy of the Milnor number in a family of reduced curves is equivalent to its topological equisingularity, under a natural assumption on the fibers.
\end{abstract}


\maketitle

\tableofcontents

\section{Introduction}
Let $(X_0,0) \subset (\mathbb{C}^n,0)$ be a reduced complex analytic curve singularity. Let us first introduce the two analytic invariants of $(X_0,0)$ at play. Denote by $\mathrm{Jac}(X_0)$ the Jacobian ideal of $(X_0,0)$ in $\mathcal{O}_{X_0,0}$; this is the first Fitting ideal of $\Omega_{X_0,0}^1$. Denote by $e(\mathrm{Jac}(X_0))$ the {\it Hilbert--Samuel (HS) multiplicity} of $\mathrm{Jac}(X_0)$. The HS multiplicity of $\mathrm{Jac}(X_0)$ is the colength of the ideal generated by a general $\mathbb{C}$-linear combination of the $(n-1) \times (n-1)$ minors of the Jacobian matrix of $X_0$, obtained from any set of defining equations for $X_0$.
 
Let $(Z_0,0) \subset (\mathbb{C}^n,0)$ be a complete intersection curve defined by $n-1$ general $\mathbb{C}$-linear combinations of a set of defining equations for $(X_0,0)$. Then $(Z_0,0)$ contains $(X_0,0)$ as an analytic subset, and $(Z_0,0)$ is reduced (see \cite[Sct.\ 6]{BR24}). Set $W_0:=\overline{Z_0\setminus X_0}$. 
 
\begin{definition}
Define the {\it complete intersection discrepancy} of $X_0$ at $0$  as the intersection number 
$$\mathrm{cid}(X_0,0):=I_{0}(X_0,W_0)=\dim_{\mathbb{C}} \mathcal{O}_{\mathbb{C}^n,0}/(I_{X_0}+I_{W_0})$$ where $I_{X_0}$ and $I_{W_0}$ are the ideals of $X_0$ and $W_0$ in $\mathcal{O}_{\mathbb{C}^n,0}$, respectively. In \cite{BR24} it is shown that $\mathrm{cid}(X_0,0)$ is an intrinsic invariant of $(X_0,0)$.
\end{definition}
 
 
 
In Thm.\ \ref{constancy} we show that the complete intersection discrepancy behaves well under flat deformations of $(X_0,0)$ over a smooth base. When $(X_0,0)$ is smoothable - that is, when $(X_0,0)$ admits a flat deformation to a smooth curve - in Cor.\ \ref{computing cid} we show how to compute it without explicitly determining the ideal $I_{W_0}$. In \cite{BR24} we use the complete intersection discrepancy as a  correction term to obtain a generalization of a well-known formula in singularity theory due to L\^e, Greuel and Teissier for complete intersection curves, as well as a generalization of the classical genus-degree formula for projective plane curves. Further applications appear in \cite{PR25}, where $\cid$ is used to give a new proof of
the uniqueness of the Milnor fiber of a smoothable curve up to diffeomorphism, and in
\cite{RT24}, where it is used to generalize the Plücker formula for plane curves to
curves of arbitrary codimension.
 
One of the main results of this paper determines how the difference between the multiplicity of the Jacobian ideal and the complete intersection discrepancy changes as one moves from a given fiber to a nearby fiber in a flat family of curves.
 
\textbf{Family setup.} Suppose $h \colon (X,0) \rightarrow (Y,0)$ is a flat family of reduced curves such that 
$h$ factors through the composition of maps 
\begin{tikzcd}[column sep=small]
(X,0) \arrow[hookrightarrow]{r} & (\mathbb{C}^n,0) \times (Y,0) \ar[r,"pr_2"] & (Y,0)
\end{tikzcd}
with $(h^{-1}(0),0)=(X_0,0)$ (such a family is called a {{\it flat embedded deformation} of $X_0$). Assume  $Y$ is smooth. Often we will assume that $(X_0,0)$ is contained in a small open $n$-ball $\mathring{B_0}$, $(Y,0)$ is identified with an open ball in $\mathbb{C}^k$ centered at $0$, sufficiently small with respect to $B_0$, and we will view $X$ as an analytic subset of $B_0 \times Y$. 
 
Let $(Z,0) \rightarrow (Y,0)$ be a flat family of complete intersection curves in $(\mathbb{C}^n,0)$ defined by $n-1$ general $\mathbb{C}$-linear combinations of a set of defining equations for $(X,0)$. Define $W:=\overline{Z \setminus X}$. In Sct.\ \ref{main proof} it is shown that for such $Z$ the  intersection number $I_x(X_y,W_y)$ is equal to the intrinsic invariant $\mathrm{cid}(X_y,x)$ of $(X_y,x)$ for $x \in (X_y)_{\mathrm{sing}}$, for $y=0$ and for $y$ in a Zariski open subset of $Y$; in general $I_x(X_y,W_y) \geq \mathrm{cid}(X_y,x)$. Set 
 
\[e(\mathrm{Jac}(X_y)):=
\sum_{\mathclap{x \in (X_y)_{\mathrm{sing}}}}
e(\mathrm{Jac}(X_y,x))
\quad
\text{and}
\quad
\mathrm{cid}(X_y):=
\sum_{\mathclap{x \in (X_y)_{\mathrm{sing}}}}
\mathrm{cid}(X_y,x).\]
One readily sees that  $\mathrm{cid}(X_y)=0$ if and only if $X_y$ is a local complete intersection.
 
 
Assume $(Y,0)$ is of dimension $k$. Consider a linear functional $F$ on $\mathbb{C}^n$ whose kernel $H=F^{-1}(0)$ is a general point of $\check{\mathbb{P}}^{n-1}$ (in particular $H$ is not a limiting tangent hyperplane of $X_{0}$ at $0$). Define $\Gamma_{h}^{k}(X)$, the {\it relative polar variety of dimension $k$}, as the closure of the set of critical points of $F$ restricted to the smooth part of the fibers of $h$ (see Sct.\ \ref{polar} for an intersection-theoretic definition). When $Y$ is of dimension one, $\Gamma_{h}^{1}(X)$ is referred to as the {\it relative polar curve}.  Because $(X,0)$ is reduced, $\Gamma_{h}^{k}(X)$ is reduced and so for generic $y \neq 0$, $\Gamma_{h}^{k}(X)_y$ consists of finitely many reduced points.
Denote their number by $\mathrm{deg}_{Y}  \Gamma_{h}^{k}(X)$.

\begin{theorem}\label{main mpt} We have
\begin{equation}\label{mpt-cid}
e(\mathrm{Jac}(X_0)) - \mathrm{cid}(X_0)-\bigl( e(\mathrm{Jac}(X_y))-\mathrm{cid}(X_y)\bigl)=\mathrm{deg}_{Y}  \Gamma_{h}^{k}(X)
\end{equation} 
for $y \neq 0$ in a Zariski open subset of $Y$.
\end{theorem}
 
The proof of this theorem uses the Buchsbaum--Rim multiplicity and integral closures of modules, the constancy of $y  \mapsto \sum_{x \in X_y \cap W_y}I(X_y,W_y)$ and the Multiplicity--Polar theorem (Thm. \ref{MPT}).
 
 
Let $\nu:\overline{(X_0,0)}\to (X_0,0)$ be the normalization morphism. The number $r_0:=|\nu^{-1}(0)|$ is the \textit{number of branches of $X_0$ at} $0$. Define the \textit{delta invariant} of $X_0$ at $0$ as $\delta_{0}:=
\dim_\mathbb{C}\nu_*\mathcal{O}_{\overline{X_0,0}}/\mathcal{O}_{X_0,0}$. Denote by $\mu_0$ the Milnor number of $(X_0,0)$ as defined by Buchweitz and Greuel
in \cite{BG80}. By \cite[Prop.\ 1.2.1]{BG80}, $\mu_0=2\delta_0-r_0+1$, a formula discovered for plane curves by Milnor. Denote by $m_0$ the multiplicity of $X_0$ at $0$.  
Using Grothendieck duality theory the first and the third author derived in \cite[Thm.\ 3.9]{BR24}  the following  Teissier--L\^e--Greuel type formula
\begin{equation}\label{Milnor}
e(\mathrm{Jac}(X_0)) - \mathrm{cid}(X_0)= \mu_0+ m_0 -1.
\end{equation}
When $(X_0,0)$ is a complete intersection, then $\mathrm{cid}(X_0)=0$.
Thus (\ref{Milnor}) recovers a formula derived by Teissier for plane curves \cite[Prop.\ II.1.2]{T73} and by L\^e \cite{Le74} and Greuel \cite{Greuel73} for complete intersection curves. As a direct application of Thm.\ \ref{main mpt} in Cor.\ \ref{main}, we give a new proof of (\ref{Milnor}) for smoothable curves.  It is worth noting that when $(X_0,0)$ is Gorenstein, $e(\mathrm{Jac}(X_0)) - \mathrm{cid}(X_0)$ is the degree of the exceptional divisor of the Nash blowup of $(X_0,0)$ as demonstrated in \cite[Cor.\ 4.3]{BR24}.

By  \cite[Dfn.\ 3.1 and Cor.\ 3.2]{NT08} identity (\ref{Milnor}) gives an algebraic formula for the Milnor number $\mu(F|X_0)$ of a linear function $F: (\mathbb{C}^n,0) \rightarrow (\mathbb{C},0)$ on $X_0$. The Milnor numbers $\mu_0$ and $\mu(F|X_0)$ are calculated using the normalization of $(X_0,0)$, which in general may be hard to find, while (\ref{Milnor}) gives a way to compute $\mu_0$ directly from the defining equations of $(X_0,0) \subset (\mathbb{C}^n,0)$. Moreover, we use (\ref{mpt-cid}), (\ref{Milnor}) and a Riemann--Hurwitz formula for singular curves (Prop.\ \ref{prop:Euler_char}) to give in Thm.\ \ref{ZU of Milnor} a simple proof of the Zariski upper semicontinuity of the two Milnor numbers generalizing a result of Buchweitz and Greuel 
for one-parameter families ($\dim Y=1$).
 
 
The main application of Thm.\ \ref{main mpt} in this paper gives a multiplicity characterization of Whitney equisingularity of families of reduced curves over smooth parameter spaces of arbitrary positive dimension using conormal geometry.  
 
\textbf{Family setup with a section.}
Suppose $h \colon (X,0) \rightarrow (Y,0)$ is an embedded flat deformation of $(X_0,0) \subset (\mathbb{C}^n,0)$ with $(Y,0)$ smooth of dimension $k$. Suppose there exists a section $\sigma: (Y,0) \rightarrow (X,0)$. Identify $Y$ with the image of $\sigma$. Choose an embedding of $(X,0)$ in $\mathbb{C}^{n+k}=\mathbb{C}^{n} \times \mathbb{C}^{k}$, so that $(Y,0)$ is represented by $0 \times U$, where $0 \in \mathbb{C}^n$ and $U$ is an open neighborhood of $0$ in $\mathbb{C}^k$.

\textbf{Whitney equisingularity.}  In the above setup, let $\{x_i\}_{i \in \mathbb{N}}$ be a sequence of points from $X_{\mathrm{sm}}$ and $\{y_i\}_{i \in \mathbb{N}}$ be a sequence of points from $Y$ both converging to $0$. Consider the secant lines $\overline{x_{i}y_{i}}$, the tangent spaces  $T_{x_i}X$ and their limits $$l:=\lim_{i \rightarrow \infty}\overline{x_{i}y_{i}} \ \ \text{and} \ \ T:=\lim_{i\rightarrow \infty}T_{x_i}X$$ in the corresponding Grassmannians, which exist after taking subsequences.  We say that $(X,0) \rightarrow (Y,0)$ is {\it Whitney equisingular}, or that the pair $(X_{\mathrm{sm}},Y)$ satisfies Whitney condition B, if $l \subset T$. Equivalently, one can replace tangent spaces $T_{x_i}X$ by {\it tangent hyperplanes} $H_{x_i}$ (hyperplanes in $\mathbb{C}^{n+k}$ that contain $T_{x_i}X$) and require that $l$ is contained in their limits in $\check{\mathbb{P}}^{n+k-1}$ (see \cite{HM83}).

\textbf{Strong simultaneous resolution.}
Denote by $\nu:\overline{X}\to (X,0)$ the normalization morphism.  We say that $(X,0) \rightarrow (Y,0)$ admits {\it strong simultaneous resolution} if the fibers $\overline{X}_y$ are smooth and $\nu^{-1}(Y)=\overline{X} \times_{X} Y$, with its natural (in general non-reduced) analytic structure, is locally analytically trivial over $(Y,0)$, i.e.\ $(\nu^{-1}(Y),\bar{x}) \cong (\nu^{-1}(Y)_0,\bar{x}) \times (Y,0)$ over $(Y,0)$ for each $\bar{x} \in \nu^{-1}(0)$. 
 
In \cite[Prop.\ 4, p.\ 111]{T80} Teissier
proved that for families of hypersurfaces the existence of strong simultaneous resolution implies Whitney equisingularity. In \cite[Sct.\ 5, p.\ 115]{T80} he showed that the converse holds for families of plane curves. 
 
The definition of Whitney equisingularity can be turned into a geometric condition using integral closures of ideals and modules. Define the {\it relative conormal space} $C_{h}(X)$ of $X$ as the closure in $X \times \check{\mathbb{P}}^{n-1}$ of the set of pairs $(x,H)$ where $x$ is a smooth point of $X_{h(x)}$ and $H$ is a (tangent) hyperplane in $\mathbb{C}^n$ at $x$ containing $T_{x}X_{h(x)}$. Let $c_h: C_{h}(X) \rightarrow X$ be the structure morphism. Let $D$ be the exceptional divisor of $\mathrm{Bl}_{c_h^{-1}(Y)}C_{h}(X)$. Then Teissier and Gaffney (\cite{Teissier} and \cite{Gaf1}) showed that $(X,0) \rightarrow (Y,0)$ is  Whitney equisingular if the fibers $D \rightarrow Y$ are equidimensional. We show that for families of curves it suffices to control only the equidimensionality of the fibers of $c_h$. Concretely, 
$(X,0) \rightarrow (Y,0)$ is a Whitney equisingular family of reduced curves if and only if the fibers of $c_h^{-1}(Y) \rightarrow Y$ are equidimensional and $\bigcup_{y \in Y}(X_y)_{\mathrm{sing}}=Y$. 
 
Closely related to the relative conormal space is {\it the relative Nash blowup} $b_h:N_h(X) \rightarrow X$ defined as the closure in $X \times \mathrm{Gr}(1,n)$ of the pairs $(x,T_xX_{h(x)})$ where $x$ is a smooth point in $X_{h(x)}$ and $T_{x}X_{h(x)}$ is the tangent space to $X_{h(x)}$ at $x$. In Prop.\ \ref{vertical} we show that when $\dim Y=1$, the right-hand side of (\ref{mpt-cid}) is also equal to the degree 
of the part of the special fiber of  $N_h(X)$ that is not the Nash blowup of $X_0$.  As an application of Thm.\ \ref{main mpt}, we obtain the following result. 
\begin{theorem}\label{Whitney}
The following are equivalent.
\begin{enumerate}
    \item[\rm{(i)}] The function $y \mapsto e(\mathrm{Jac}(X_y,0)) - \mathrm{cid}(X_y,0)$ is constant across $(Y,0)$.
    
    \item[\rm{(ii)}] For each $y \in Y$ the only singular point of $X_y$ is $0$ and $\Gamma_{h}^k(X)$ is empty (the last is equivalent to asking that $c_{h}^{-1}(Y) \rightarrow Y$ has equidimensional fibers, or that $b_{h}^{-1}(Y) \rightarrow Y$ has finite fibers). 
    \item[\rm{(iii)}] For each $y \in Y$ the only singular point of $X_y$ is $0$ and the family $h: (X,0) \rightarrow (Y,0)$ is  Whitney equisingular.
 
    \item[\rm{(iv)}]For each $y \in Y$ the only singular point of $X_y$ is $0$ and the family $h: (X,0) \rightarrow (Y,0)$ admits strong simultaneous resolution.
    \end{enumerate}
\end{theorem}
 
In Prop.\ \ref{jacobian mult.} we show that 
the constancy of $y \mapsto e(\mathrm{Jac}(X_y,0))$ implies $\rm{(iii)}$, however, the converse is not true, as demonstrated by an example.
 
Thm.\ \ref{ZU of Milnor} and \cite[Thm.\ 3.9]{BR24}, together with the upper semicontinuity of the multiplicity, imply that \rm{(i)} is equivalent to the constancy of the Milnor number $\mu(X_y,0)$ and the multiplicity $m(X_y,0)$ of $X_y$ at $0$. For one-parameter families of reduced curves it is shown in \cite[Thm.\ III.3]{BGG} that $(X,0) \rightarrow (Y,0)$ is Whitney equisingular if and only if $\mu (X_y,0)$ and $m(X_y,0)$ are constant along $Y$. The proof of this result uses in an essential way results of Buchweitz and Greuel, such as \cite[Thm.\ 4.2.2]{BG80}, which have been known for one-parameter families only. Extending \cite[Thm.\ III.3]{BGG} and \cite[Thms.\ 4.2.2]{BG80} to the case $\dim Y > 1$ appears to have remained an open problem (see \cite[Sct.\ 10.2, p.\ 270]{Gre17}).
 
Our proof of Thm.\ \ref{Whitney} is algebro-geometric, which allows us to work with $Y$ of arbitrary positive dimension. First, we establish the equivalence of $\rm{(i)}$ and $\rm{(ii)}$ via Thm.\ \ref{main  mpt}. In Thm.\ \ref{equimult} we show that $\rm{(ii)}$ implies that $X$ is equimultiple along $Y$, i.e.\ $y \mapsto m_y$ is constant where $m_y$ is the multiplicity of $X_y$ at $0$. Then using the characterization of Whitney equisingularity in terms of integral closure of modules we showed that $\rm{(ii)}$ implies $\rm{(iii)}$. In Thm.\ \ref{equimult} we also show that $\rm{(ii)}$ implies that the delta invariant and the number of branches of $(X_y,0)$ are constant as $y$ varies and thus one obtains that  $\rm{(ii)}$ implies  $\rm{(iv)}$. Finally,  we show in a straightforward fashion that $\rm{(iv)}$ yields  the constancy of the delta invariant, the number of branches and the multiplicity of $(X_y,0)$ which by \cite[Thm.\ 3.9]{BR24} gives $\rm{(i)}$. 
 
The main application of Thm.\ \ref{ZU of Milnor} (Zariski upper semicontinuity) is a fiberwise numerical characterization of topological equisingularity. Recall the following definition.
 
\textbf{Topological equisingularity.} Suppose $h \colon (X,0) \rightarrow (Y,0)$ is a flat family of reduced curves such that 
$h$ factors through \begin{tikzcd}[column sep=small]
(X,0) \arrow[hookrightarrow]{r} & (\mathbb{C}^n,0) \times (Y,0) \ar[r,"pr_2"] & (Y,0).
\end{tikzcd} We say that $h: (X,0) \rightarrow (Y,0)$ is {\it topologically equisingular} if there exists a homeomorphism $\phi:(X,0) \rightarrow (X_0,0) \times (Y,0)$ such that $h=pr_2 \circ \phi$. 
 
 
\begin{theorem}\label{top. equis.}
Set $\mu_y:=\sum_{x \in (X_y)_{\mathrm{sing}}}\mu(X_y,x)$. Suppose $y \mapsto \mu_y$ is constant.
Assume that either the fibers $X_y$ are irreducible, or that all irreducible components of each $X_y$ intersect at a single point and have no other points of intersection. Then $h: (X,0) \rightarrow (Y,0)$ is topologically equisingular after possibly shrinking $Y$. Conversely, if  $h:(X,0) \rightarrow (Y,0)$ is topologically equisingular, then $y \mapsto \mu_y$ is constant.
\end{theorem}
 
Note that the assumptions on $X_y$ are necessary for the family to be topologically equisingular.
Our result generalizes a result of Buchweitz and Greuel (see \cite[Thm.\ 5.2.2]{BG80}) for one-parameter families of curves in two directions. In our treatment, we do not need to assume the existence of a section from $Y$ to $X$ that parametrizes the singular points of the fibers $X_y$ (see Sct.\ \ref{top. equising.}). We also do not impose a restriction on the dimension of the parameter space $Y$.

 \textbf{Acknowledgements.} The authors would like to thank Bernard Teissier, Pablo Portilla Cuadrado, Steven Kleiman, and Patrick Popescu-Pampu for helpful and stimulating discussions. We are grateful to the anonymous referee, whose comments on the notion of Zariski upper semicontinuity led us to prove the results of Sct.\ \ref{ZU section} in their present, stronger form. The first author was supported by the Bulgarian Ministry of Education and Science Scientific Programme
"Enhancing the Research Capacity in Mathematical Sciences (PIKOM)", No.
DO1-67/05.05.2022. The third author was supported by the European Union’s Horizon Europe research and innovation programme under Marie Sk\l{adowska}-Curie Actions project GTSP-$101111114$.   

\textbf{AI disclosure.} This paper was written in full by the authors before any AI tool was used; the pre-AI version is arXiv:2310.16558v4 (February 22, 2026). AI (Claude Fable 5.1) was subsequently used on that manuscript for proofreading and improving the exposition. It found a gap in the alternative proof of Thm.\ \ref{equimult} which is now fixed using (\ref{eq:RH_partial}) and (\ref{ball-polydisc}). AI also suggested the comparison (\ref{ball-polydisc}) in the proof of Thm.\ \ref{ZU of Milnor}, the reorganization of the proof of Thm.\ \ref{main mpt} through \cite[Thm.\ 3.9]{BR24}, and the expanded proof of Thm.\ \ref{MPT}. All arguments have been checked by the authors who take full  responsibility for the content of the final manuscript.

\section{Conormal spaces, polar varieties, and multiplicities}
\subsection{Conormal spaces}\label{Conormal spaces}
Let $h \colon (X,0) \rightarrow (Y,0)$ be a flat embedded deformation of a reduced complex analytic curve $(X_0,0) \subset \mathbb{C}^n$, with $(Y,0)$ smooth of dimension $k$, and $(X,0) \subset (\mathbb{C}^{n+k},0)$. As is conventional when dealing with germs, we will consider affine representatives. First, we describe the {\it relative conormal variety} $C_{h}(X)$ of $X$ in $\mathbb{C}^{n+k}$ using the relative Jacobian module of $X$. Suppose $X$ is reduced. Let $X$ be defined by the vanishing of analytic
functions $f_1, \ldots, f_p$ on a Euclidean neighborhood of $0$ in $\mathbb{C}^{n+k}=\mathbb{C}^n \times Y$. Consider the following conormal exact  sequence:
\begin{equation}\label{normal sequence}
\begin{tikzcd}
I/I^2\ar[r,"\delta"] &  \Omega_{\mathbb{C}^{n}\times Y / Y}^{1}|X \ar[r] &
\Omega_{X/Y}^{1} \ar[r] & 0
\end{tikzcd}
\end{equation}
where $I$ is the ideal of $X$ in $\mathcal{O}_{\mathbb{C}^{n}\times Y,0}$ and the map $\delta$ sends a function $f$ vanishing on $X$ to its differential $df$. Dualizing, we obtain the following  nested sequence of torsion-free sheaves:
\begin{displaymath}
\mathrm{Image}(\delta^{*}) \subset (\mathrm{Image} \ \delta)^{*} \subset (I/I^{2})^{*}.
\end{displaymath}
Observe that locally the sheaf $\mathrm{Image}(\delta^{*})$ can be viewed as the column space of the {\it relative Jacobian matrix} of $X$. If $({\bf x}, {\bf y}):=(x_1, \ldots, x_n, y_1, \ldots, y_k)$ are coordinates on $(\mathbb{C}^{n+k},0)$, where $(y_1, \ldots, y_k)$ are coordinates on $(Y,0)$, then the relative Jacobian matrix is simply the $p\times n$ matrix $(\partial f_{i}({\bf x},{\bf y}) / \partial x_{j})$.

Denote $\mathrm{Image}(\delta^{*})$ by $J_{h}(X)$ and call it {\it the relative Jacobian module} of $X$. Because $\mathcal{O}_{X,0}^p \twoheadrightarrow I/I^2$, we have $(I/I^{2})^{*} \subset \mathcal{O}_{X,0}^p$. So $J_{h}(X)$ is contained in $\mathcal{O}_{X,0}^{p}$. Define the {\it Rees algebra} $\mathcal{R}(J_{h}(X))$ to be the subalgebra of $\mathrm{Sym}(\mathcal{O}_{X,0}^{p})$ generated by $J_{h}(X)$ placed in degree $1$. Define the {\it relative conormal space} $C_{h}(X)$ of $X$ to be the closure in $X \times \check{\mathbb{P}}^{n-1}$ of the set of pairs $(x,H)$ where $x$ is a smooth point of the fiber $X_{h(x)}$ and $H$ is a (tangent) hyperplane in $\mathbb{C}^n$ at $x$ containing $T_{x}X_{h(x)}$. We have 
\begin{displaymath}
 C_{h}(X) = \mathrm{Projan}(\mathcal{R}(J_{h}(X))).
\end{displaymath}
Indeed, observe that both sides are equal over the smooth part of the fibers $(X_y)_\mathrm{sm}$ to the set of pairs $(x,H)$ where $H$ is a tangent hyperplane  in $\mathbb{C}^n$ to the simple point $x \in (X_y)_\mathrm{sm}$. The left-hand side is the closure of this set, and so is the right-hand side, simply because the Rees algebra is by construction a subalgebra of the symmetric algebra $\mathrm{Sym}(\mathcal{O}_{X,0}^{p})$.

\subsection{Polar varieties and the Milnor number}\label{polar}
We preserve the setup from the previous section. Consider the following diagram:
\[\begin{tikzcd}
C_{h}(X)\ar[r,hook] \ar[dr,"\lambda"'] \ar[d,"c_h"]
    & X\times\check{\mathbb{P}}^{n-1} \ar[d,"pr_2"] \\
X & \check{\mathbb{P}}^{n-1}.
\end{tikzcd}\]
 
The dimension of the generic fiber of $c_h$ is $n-2$, so $\dim C_{h}(X)=\dim X+n-2=n+k-1$.  By Kleiman's transversality theorem \cite{Kl74}
(taking the transitive action $PGL(n,\mathbb{C})$ on $\check{\mathbb{P}}^{n-1}$, considered as
the set of hyperplanes of $\mathbb{C}^n$), for a generic point
$H\in\check{\mathbb{P}}^{n-1}$, the pullback
$\lambda^{-1}(H)$ is either empty or of pure dimension equal to $\dim C_{h}(X)-\text{codim}(H)=
n+k-1-(n-1)=k$. Define the {\it relative polar variety of 
$h \colon (X,0) \rightarrow (Y,0)$ of  dimension $k$} as
$\Gamma_{h}^{k}(X):=c_h(\lambda^{-1}(H))$. 
When $\dim Y=1$, $\Gamma_{h}^{1}(X)$ is called the relative polar curve. Again by Thm.\ 2 \rm{(ii)} and Rmk.\ 7 in \cite{Kl74}, choosing $H$ generic enough guarantees that $\Gamma_{h}^{k}(X)$ is reduced, because by assumption $X$ is reduced and so is $C_h(X)$. An equivalent way to obtain $\Gamma_{h}^{k}(X)$ is to consider a linear functional $F$ on $\mathbb{C}^n$ such that $F^{-1}(0)=H$ is general as above. Then $\Gamma_{h}^{k}(X)$ is the closure of the set of critical points of $F$ restricted to the smooth part of the fibers of $h$ (\cite[Prop.\ 3.4]{Flores}).
Denote by $\mathrm{deg}_{Y}  \Gamma_{h}^{k}(X)$ the number of points in the fiber $\Gamma_h^{k}(X)_y$ for $y \neq 0$ generic. Denote by $m_0$ the multiplicity of $(X_0,0)$.

In \cite{BG80} Buchweitz and Greuel define the Milnor number of $(X_0,0)$ as follows. Let $n \colon \overline{(X_0,0)} \rightarrow (X_0,0)$ be the normalization morphism. Denote by $\omega_{X_0,0}$ the dualizing module of Grothendieck. Denote by $d$ the composition of $\Omega_{X_0,0}^{1}\rightarrow n_{*}\Omega_{\overline{X_0,0}}^1 \cong n_{*}\omega_{\overline{X_0,0}} \rightarrow \omega_{X_0,0}$ and the exterior derivation $\mathcal{O}_{X_0,0} \rightarrow \Omega_{X_0,0}^{1}$. Then the Milnor number $\mu$ of $(X_0,0)$ is defined as 
$$\mu_0 = \mu(X_0,0):=\dim_{\mathbb{C}}(\omega_{X_0,0}/d\mathcal{O}_{X_0,0}).$$ When $(X_0,0)$ is smoothable, Bassein \cite{Bas77} proved that $\mu_0$ equals the first Betti number of the smooth nearby fiber; in particular, for complete intersection curves $\mu_0$
coincides with the usual Milnor number.
 
In the family setup, we choose representatives for $(X,0)$ and $(Y,0)$ as follows. Assume $(X_0,0)$ is contained in an open ball $\mathring{B_0}=\mathring{B}(0,\epsilon)\subset(\C^n,0)$. Identify $Y$ with an open $k$-ball. View $(X,0)$ as analytic subspace of $\mathring{B_0} \times (Y,0)$.  Write $\chi(X_y)$ for the Euler characteristic of the Euclidean closure of $X_y$ in $B_0 \times \{y\}$.
 
\begin{proposition}\label{Euler} Suppose $h \colon (X,0) \rightarrow (Y,0)$ is an embedded flat deformation of a reduced curve $(X_0,0) \subset (\mathbb{C}^n,0)$ such that $X_y$ is smooth for all $y$ in a Zariski open subset of $Y$. Let $F \colon (\mathbb{C}^n,0) \rightarrow (\mathbb{C},0)$ be a general linear functional. For $\mathring{B_0}$ and $(Y,0)$ sufficiently small we have the following  Euler characteristic--degree relation
\begin{equation}\label{Le}
\mathrm{deg}_{Y}  \Gamma_{h}^{k}(X) = \chi(X_y \cap F^{-1}(c))-\chi(X_y) = \mu_0 + m_0-1
\end{equation}
where $y$ is in a Zariski open subset $U$ of $Y$ and $c \in (\mathbb{C},0)$ is generic. Moreover, $1-\chi(X_y) \geq 0$ for $y \in U$ with equality if and only if $X_0$ is smooth, which is equivalent to $\Gamma_{h}^{k}(X)=\emptyset$. 
 
\end{proposition}
\begin{proof}
Let $(\mathbf{x},\mathbf{y}):=(x_1,\ldots,x_n,y_1, \ldots, y_k)$ be coordinates for $\mathring{B_0}\times (Y,0)$. Set $F(\mathbf{x}):= \sum_{i=1}^{n} \alpha_i x_i$ where the $\alpha_i$ are complex numbers subject to finitely many genericity conditions specified below. Define $\tilde{F}: \mathring{B_0} \times (Y,0) \rightarrow (\mathbb{C},0) \times (Y,0)$ as $\tilde{F}(\mathbf{x},\mathbf{y})=(F(\mathbf{x}),\mathbf{y})$. Choose sufficiently small representatives for $\mathring{B_0}$,  $(Y,0)$ and $(\mathbb{C},0)$ so that the restriction of $\tilde{F}$ to $X$ is a finite map (see \cite[Prop.\ 1), p. 63]{GrR84}). Thus $X$ is replaced by $X \cap (\mathring{B_0} \times Y) \cap \tilde{F}^{-1}(\mathring{\mathbb{D}} \times Y)$, where $\mathbb{D}=\mathbb{D}(0,\delta)$ is a small closed disk; for $y$ close enough to $0$ the Euler characteristic of its fiber equals $\chi(X_y)$, the Euler characteristic of the closure of $X_y \cap \mathring{B_0}$, because the two pieces differ by $r_0$ annuli (see (\ref{ball-polydisc})). Denote by  $\Sigma_{X_y}(F)$ the critical locus of $F$ on $X_y$. Denote by $\Sigma_{X/Y}(\tilde{F})$ the union of  $\Sigma_{X_y}(F)$ as $y$ varies through $Y$. This set has an analytic structure given by the vanishing of the $n \times n$ minors of the $(p+1) \times n$ matrix obtained by augmenting the relative Jacobian matrix of $(X,0) \rightarrow (Y,0)$ with the gradient of $\tilde{F}$. The locus $\Sigma_{X/Y}$ where $h: (X,0) \rightarrow (Y,0)$ is not smooth is given by the vanishing of the $(n-1) \times (n-1)$ minors of the relative Jacobian matrix. By \cite[Prop.\ 3.4]{Flores} $\Gamma_{h}^{k}(X)$ is $\overline{\Sigma_{X/Y}(\tilde{F}) \setminus \Sigma_{X/Y}}$. Thus $\mathrm{deg}_{Y} \Gamma_{h}^{k}(X)= \#(\Sigma_{X_y}(F))$ for generic $y \in U$.
 
Choose $\epsilon$ small enough so that 
$\partial B_0$ intersects $X_0$ transversally; by openness of transversality after possibly shrinking $Y$, we have that $\partial B_0 \times \{y\}$ intersects $X_y$ transversally. Also, for $\epsilon$ small enough we have $(\mathring{B_0}\times \{0\}) \cap \Gamma_{h}^{k}(X) \subset \{0\}$.
Thus, for $y$ close enough to $0$, the functional $F_{|X_y}$ does not have critical points on $S_{\epsilon}^{2n-1}$. After possibly shrinking $U$ we can assume that this is the case for all $y \in U$. By Sard's theorem and Ehresmann's fibration theorem, we can assume that $X \times_{Y} U \rightarrow U$, where $X$ is identified with its Euclidean closure in $B_0 \times (Y,0)$, is a trivial fibration after further shrinking $U$ if necessary. In particular, $y \mapsto \chi(X_y)$ is constant for $y \in U$. Bertini's theorem and Ehresmann's fibration theorem imply similarly that $y \mapsto \chi(F^{-1}(c) \cap X_y)$ is constant for $y \in U$ after possibly shrinking $U$. Fix $y \in U$. By  \cite[Thm.\ A.5]{NOT13} (cf.\ \cite[Thm.\ 1.4]{Kav04})
\begin{equation}\label{Kaveh}
\#(\Sigma_{X_y}(F))=\chi(F^{-1}(c) \cap X_y)-\chi(X_y) \ \ c \in \mathbb{C} \ \text{generic}
\end{equation}
where $\alpha_i$ are generic so that $F|_{X_y}$ has Morse critical points. By \cite{Bas77} (cf.\ \cite[Cor.\ 4.2.3]{BG80}) $\mu_0 = 1-\chi(X_y)$. Therefore, it remains to show that $$\chi(F^{-1}(c) \cap X_y)=m_0.$$
 
Set $f:=\tilde{F}|X$ and $f_y:=F|_{X_{y}}$. Because $(X_0,0)$ is a Cohen--Macaulay curve, for generic $\alpha_i$ we have $m_0=\dim_{\mathbb{C}} \mathcal{O}_{X_0,0}/f_0$. Because $h$ is flat and $X_0$ is reduced, then $(y_1, \ldots, y_k,f)$ is a regular sequence in $\mathcal{O}_{X,0}$, and so is $(f,y_1, \ldots, y_k)$, because $\mathcal{O}_{X,0}$ is local. Our choices for $\mathring{B_0}$ and $(Y,0)$ ensure that $\mathcal{O}_{X,0}/f$ is a finite $\mathcal{O}_{Y,0}$-module.
By \cite[Thm.\ 18.16 (b) (Miracle Flatness)]{Eis95} we have $\dim_{\mathbb{C}} \mathcal{O}_{X_0,0}/f_0=\dim_{\mathbb{C}} \mathcal{O}_{X_y}/f_y.$
 
Fix $y \in Y$. Consider the trivial family $X_y \times (\mathbb{C},0) \rightarrow (\mathbb{C},0)$, where for $(\mathbb{C},0)$ we choose the representative fixed in the definition of $\tilde{F}$. We have $X_y \times (\mathbb{C},0) \subset \mathring{B_0} \times (\mathbb{C},0)$. Let $(x_1, \ldots,x_n,t)$ be coordinates on $\mathring{B_0} \times (\mathbb{C},0)$. View $g_t: = \sum_{i=1}^n \alpha_{i}x_i - t$ as an element in $\mathcal{O}_{X_y \times (\mathbb{C},0)}$. Set  $\mathbb{V}(g_t):=\mathrm{Specan}(\mathcal{O}_{X_y \times (\mathbb{C},0)}/(g_t))$. Because $X_y \times (\mathbb{C},0)$ is Cohen--Macaulay, so is $\mathbb{V}(g_t)$. Because $f_y: X_y \rightarrow (\mathbb{C},0)$ is finite we have that $\mathcal{O}_{X_y \times (\mathbb{C},0)}/(g_t)=\mathcal{O}_{X_y}[t]/(g_t)=\mathcal{O}_{X_y}$ is a finite $\mathcal{O}_{\mathbb{C},0}$-module. By Bertini's theorem $\dim_{\mathbb{C}}\mathcal{O}_{X_y}/g_c=\#(f^{-1}(c))=\#(F^{-1}(c) \cap X_y)$ for $c \in (\mathbb{C},0)$ generic. Because  $\mathbb{V}(g_t)\rightarrow (\mathbb{C},0)$ is finite and flat, we get that 
 
$$\dim_{\mathbb{C}} \mathcal{O}_{X_y}/f_y=\dim_{\mathbb{C}}\mathcal{O}_{X_y}/g_0=\dim_{\mathbb{C}}\mathcal{O}_{X_y}/g_c=\# (F^{-1}(c) \cap X_y)$$
for $c$ close enough to $0$. But $\# (F^{-1}(c) \cap X_y)=\chi(F^{-1}(c) \cap X_y)$ for $y \in U$. Therefore, $\chi(F^{-1}(c) \cap X_y)=m_0$. This proves (\ref{Le}). 
Identify $X_y$ with its Euclidean closure in the closed ball $B_0$. Denote by $g(X_y)$ the genus of $X_y$ for $y \in U$. By \cite[Thm.\ 8.2]{Gre17} and because $X \times_U U\rightarrow U$ is a topological fibration, we have that $X_y$ is a connected manifold. By Ehresmann's fibration theorem applied to the boundary $(\partial B_0 \times (Y,0)) \cap X \rightarrow (Y,0)$ we get that $X_y$ has $r_0$ boundary components, where $r_0$ is the number of branches of $X_0$. Thus $1-\chi(X_y)=2g(X_y)+r_0-1 \geq 0$ because $r_0 \geq 1$. We have $1-\chi(X_y)=0$ if and only if $g(X_y)=0$, $r_0=1$ ($X_y$ is a disk). Thus $1-\chi(X_y)=0$ is equivalent to $\delta_0=0$, where $\delta_0$ is the delta invariant of $(X_0,0)$, which is equivalent to  $X_0$ smooth. But $X_0$ smooth is equivalent to $m_0=1$. So by (\ref{Le}) $X_0$ smooth is equivalent to $\Gamma_{h}^{k}(X)=\emptyset$.
\end{proof}
Note that the first part of the argument above, with (\ref{Kaveh}) replaced by its general form \cite[Thm.\ 1.4]{Kav04}, gives $\mathrm{deg}_{Y}  \Gamma_{h}^{k}(X) = (-1)^{d+1}\chi(F^{-1}(c)\cap X_y)+(-1)^{d}\chi(X_y)$ for a smoothable $(X_0,0)$  of dimension $d$. In the case of curves considered in this paper, instead of (\ref{Kaveh}) one can apply
Prop.\ \ref{prop:Euler_char} to $X_y \cap \mathring{B_0} \cap F^{-1}(\mathring{\mathbb{D}})$, on which $F$ is finite of degree $m_0$ by the computation above, and use (\ref{ball-polydisc}) to get directly $\#(\Sigma_{X_y}(F))=m_0-\chi(X_y)$ for $y \in U$ close to $0$. Also, instead of applying Ehresmann's theorem twice to obtain the constancy of $\chi(X_y)$ and $\chi(F^{-1}(c) \cap X_y)$, one can use Prop.\ \ref{relative Morse} to ensure that $F_{|X_{y}}$ has Morse critical points for all $ y\in U$, after possibly shrinking $U$, in which case (\ref{Kaveh}) will hold for all $y \in U$.
 
Denote by $\mu(F|X_0)$ the Milnor number of $F|X_0$ (see \cite[Dfn.\ 2.1]{MvS01}). Suppose $k=1$. The equality between $\mathrm{deg}_{Y}  \Gamma_{h}^{1}(X)$ and $\mu_0 + m_0-1$ for curves  can be obtained by observing that $\mathrm{deg}_{Y}  \Gamma_{h}^{1}(X)=\mu(F|X_0)$ by \cite[Prop.\ 2.2]{MvS01} and $\mu(F|X_0)=\mu_0 + m_0-1$ by  \cite[Dfn.\ 3.1 and Cor.\ 3.2]{NT08}.
 
 
\subsection{Integral closure and multiplicities}
Let $(V,0)$ be a reduced complex analytic variety of pure dimension $d$. Let $M$ be an $\mathcal{O}_V$-module such that its rank at the generic point of each irreducible component of $V$ is $e$. Assume $M$ is contained in a free module $\mathcal{O}_{V}^p$ for some $p$. Denote by $[M] \in \mathrm{Mat}(p \times c, \mathcal{O}_V)$ a presentation matrix of $\mathcal{O}_{V}^p/M$. Let $A\in\text{Mat}(e\times p,\mathbb{C})$ be a generic matrix such that $[M_e]:=A[M]$ has rank $e$ at the generic point of each irreducible component of $(V,0)$. Let $M_e$ be the $\mathcal{O}_V$-module generated by the columns of $[M_e]$. Then $M_e$ is contained in a free module $\mathcal{O}_{V}^e$. Denote by $\mathcal{R}(M)$ and $\mathcal{R}(M_e)$ the Rees algebras of $M$ and $M_e$, respectively.
\begin{proposition} \label{conormal}
We have
$$\mathrm{Projan}(\mathcal{R}(M))=\mathrm{Projan}(\mathcal{R}(M_e)).$$
\end{proposition}
\begin{proof}
By construction the row spaces of  $[M]$ and of $[M_e]$ agree on the Zariski open dense subset of $V$ where the rank of these matrices is $e$. Over this Zariski open set the spaces $\mathrm{Projan}(\mathcal{R}(M))$ and $\mathrm{Projan}(\mathcal{R}(M_e))$
are just the projectivization of these row spaces, hence they are equal. Since the spaces agree on a Zariski open dense set, they agree everywhere.
\end{proof}
 
A submodule $M'$ of $M$ is called a {\it reduction} of $M$ if $\mathcal{R}(M)$ is an integral extension of $\mathcal{R}(M')$. By \cite[Cor.\ 16.4.7]{Huneke} there exists a reduction of $M$ generated by $d+e-1$ generic linear combinations of the columns of $[M]$.
 
To simplify notation set $\mathcal{I}_{0}(M_e):=\mathrm{Fitt}_{0}(\mathcal{O}_{V}^e/M_e)$ and $\mathcal{I}_{p-e}(M):= \mathrm{Fitt}_{p-e}(\mathcal{O}_{V}^p/M)$. These ideals are generated by the $e \times e$ minors of $[M_e]$ and $[M]$. By the Cauchy--Binet formula we have $\mathcal{I}_{0}(M_e) \subset \mathcal{I}_{p-e}(M)$. Denote by $\overline{\mathcal{I}_{0}(M_e)}$ and $\overline{\mathcal{I}_{p-e}(M)}$ the integral closures of  $\mathcal{I}_{0}(M_e)$ and $\mathcal{I}_{p-e}(M)$ in $\mathcal{O}_{V,0}$, respectively.
 
\begin{proposition}\label{int.clos.} Let $S_1, \ldots, S_q$ be irreducible analytic subsets in $(V,0)$ of codimension $1$. For a generic $A$ we have $S_i \not \in \mathrm{Supp}_{\mathcal{O}_{V,0}}(\overline{\mathcal{I}_{p-e}(M)}/\overline{\mathcal{I}_{0}(M_e)})$ for each $i=1, \ldots, q$.
\end{proposition}
\begin{proof} Denote by $M^*$ the $\mathcal{O}_{V,0}$-submodule of $\mathcal{O}_{V}^c$ generated by the columns of $[M]^{\mathrm{tr}}$. We have $\mathcal{I}_{p-e}(M)=\mathrm{Fitt}_{c-e}(\mathcal{O}_{V}^c/M^{*})$. Denote by $M_{e}^*$ the $\mathcal{O}_{V,0}$-submodule of $\mathcal{O}_{V}^c$ generated by the columns of $[M_e]^{\mathrm{tr}}$. Note that $M_{e}^{*}$ is a submodule of $M^{*}$ generated by $e$ elements. Moreover, $\mathcal{I}_{0}(M_e)=\mathrm{Fitt}_{c-e}(\mathcal{O}_{V}^c/M_{e}^{*})$. 
 
Fix a $S_i$ and consider the stalk $\mathcal{O}_{V,S_i}$ of $\mathcal{O}_{V,0}$ at the generic point of $S_i$.  Because the formation of integral closure and Fitting ideals commute with localization, it is enough to prove that $\mathcal{I}_{p-e}(M)$ and $\mathcal{I}_{0}(M_e)$ have the same integral closure after replacing  $\mathcal{O}_{V,0}$ by $\mathcal{O}_{V,S_i}$ and $M^{*}$ by $M^{*} \otimes_{\mathcal{O}_{V,0}} \mathcal{O}_{V,S_i}$.
By assumption $\dim \mathcal{O}_{V,S_i}=1$. So, $e$ generic $\mathbb{C}$-linear combinations of the generators of $M^{*} \otimes_{\mathcal{O}_{V,0}} \mathcal{O}_{V,S_i}$ generate a reduction of the latter module. By \cite[Thm.\ 16.3.1]{Huneke}, for a generic $A$, the ideals $\mathcal{I}_{p-e}(M)$ and $\mathcal{I}_{0}(M_e)$ have the same integral closure in $\mathcal{O}_{V,S_i}$, because $M_{e}^{*} \otimes_{\mathcal{O}_{V,0}} \mathcal{O}_{V,S_i}$ is a reduction of $M^{*} \otimes_{\mathcal{O}_{V,0}} \mathcal{O}_{V,S_i}$. Because there are finitely many $S_i$, there are finitely many genericity conditions on $A$. The proof is complete.
\end{proof}
 
Suppose now that $(V,0)$ is a reduced curve. Set $F:=\mathcal{O}_{V}^e$. Then $\mathrm{Supp}_{V}(F/M_e)=\{0\}$. The inclusion $M_e \subset F$ induces an inclusion of the Rees algebra $\mathcal{R}(M_e)$ in the symmetric algebra $\mathrm{Sym}(F)$. Denote by $M_{e}^l$ and $F^l$ the $l$th graded components of 
$\mathcal{R}(M_e)$ and $\mathrm{Sym}(F)$, respectively. Then by results of \cite{Buch} for $l$ large enough 
 
$$\dim_{\mathbb{C}} F^l/M_{e}^l = e(M_e,F)l^e/e! + O(l^{e-1}).$$
The coefficient $e(M_e,F)$ is called the {\it Buchsbaum--Rim (BR) multiplicity} of the pair $(M_e,F)$. The BR-multiplicity generalizes the HS-multiplicity of ideals to modules.  By \cite[Cor.\ 16.4.7]{Huneke} there exists a reduction $M_e'$ of $M_e$ generated by $e$ generic linear combinations of the columns of $[M_e]$, or in other words for a generic matrix $B\in\text{Mat}(c\times e,\mathbb{C})$ such that the module generated by the columns of $[M_e]B$ is a reduction of $M_e$. Set $\mathcal{I}_{0}(M_{e}'):=\mathrm{Fitt}_{0}(F/M_{e}')$. Note that $\mathcal{I}_{0}(M_{e}')$ is the determinant of the $e \times e$ matrix $[M_{e}']$.
 
By \cite[Cor.\ 16.5.7]{Huneke} $e(M_e,F)=e(M_{e}',F)$. Because $(V,0)$ is Cohen--Macaulay by \cite[4.5, p.\ 223]{Buch} $e(M_{e}',F)=\dim_{\mathbb{C}}(\mathcal{O}_{V,0}/\mathcal{I}_{0}(M_{e}'))$. The Fitting ideal $\mathcal{I}_{p-e}(M)$ is generated by the $e \times e$ minors of $[M]$. It is an ideal in $\mathcal{O}_{V,0}$ primary to the maximal ideal; thus it has a well-defined HS-multiplicity $e(\mathcal{I}_{p-e}(M))$.
 
\begin{corollary}\label{jacobian ideal}
For a generic choice of $A$ we have $e(M_e,F)=e(\mathcal{I}_{p-e}(M)).$
\end{corollary} 
\begin{proof}
Applying Prop.\ \ref{int.clos.} with $S_1=\{0\}$, we get that $\overline{\mathcal{I}_{p-e}(M)}=\overline{\mathcal{I}_{0}(M_e)}$. Therefore, by the Rees criterion (\cite[Theorem 11.3.1]{Huneke}) we have an equality of Hilbert--Samuel multiplicities $e(\mathcal{I}_{p-e}(M))=e(\mathcal{I}_{0}(M_e))$. Because $M_{e}'$ is a reduction of $M_e$ the ideals $\mathcal{I}_{0}(M_e)$ and $\mathcal{I}_{0}(M_e')$ have the same integral closures. Thus $e(\mathcal{I}_{0}(M_e))=e(\mathcal{I}_{0}(M_e'))$. But $\mathcal{I}_{0}(M_e')$ is principal. So $e(\mathcal{I}_{0}(M_e'))=\dim_{\mathbb{C}}\mathcal{O}_{V,0}/\mathcal{I}_{0}(M_e').$ But $e(M_e,F)=\dim_{\mathbb{C}}\mathcal{O}_{V,0}/\mathcal{I}_{0}(M_e')$ which proves the result.
\end{proof}
To obtain the equality in Cor.\ \ref{jacobian ideal}, selecting $A$ so that Prop.\ \ref{conormal} holds is not enough as indicated in Ex.\ \ref{3,4,5}. Finally, we remark that $e(M_e,F)$ was considered in a more general setting in \cite[Sct.\ 7]{KT-Al}. To define  $e(M_e,F)$ from the point of view of Kleiman and Thorup, one intersects $\mathrm{Projan}(\mathrm{Sym}(\mathcal{O}_{V}^p))=V \times \mathbb{P}^{p-1}$ with $p-e$ general hyperplanes from $\mathbb{P}^{p-1}$, blowing up this intersection with the image of the ideal generated by $M$ in $\mathrm{Sym}(\mathcal{O}_{V}^p)$, and then deriving $e(M_e,F)$ as a sum of certain intersection numbers of the exceptional divisor.

\section{A flatness result}

Let $(X_0,0) \subset (\mathbb{C}^n,0)$ be a complex analytic curve such that the origin is not an embedded point of $X_0$. Let $(X,0) \rightarrow (Y,0)$ be a flat deformation of $(X_0,0)$ with $(Y,0)$ smooth. Let $(Z,0) \rightarrow (Y,0)$ be a flat family of analytic curves such that $(Z,0)$ contains  $(X,0)$ as a closed analytic subspace. Assume that $(Z,0)$ is Gorenstein (e.g.\ a complete intersection). Further, assume that $X_y$ is a union of irreducible components of $Z_y$ for each $y \in Y$, in the following sense: $\mathcal{O}_{Z_y,\eta}=\mathcal{O}_{X_y,\eta}$ for every generic point $\eta$ of an irreducible component of $X_y$. Since $Z_y$ has no embedded components (see the beginning of the proof of Thm.\ \ref{constancy}), this means that $I_{X_y}=\bigcap_{i \in E}\mathfrak{q}_i$, where $(0)=\bigcap_{i}\mathfrak{q}_i$ is the primary decomposition of $(0)$ in $\mathcal{O}_{Z_y}$ and $E$ is the set of components of $X_y$; then $\overline{Z_y \setminus X_y}$ denotes the subspace of $Z_y$ with ideal $\bigcap_{i \notin E}\mathfrak{q}_i=(0:_{\mathcal{O}_{Z_y}}I_{X_y})$. Note that $\mathcal{O}_{X,0}$ is Cohen--Macaulay, because $\mathrm{depth}(\mathcal{O}_{X_0,0})=1$ and $(X,0) \rightarrow (Y,0)$ is flat. Set $W:=\overline{Z\setminus X}$, the subspace of $Z$ defined by the intersection of the primary components of $(0)$ in $\mathcal{O}_{Z,0}$ whose supports are not contained in $X$, and $S:=X \times_{Z} W$; for $y \in Y$ we have $S_y=X_y \times_{Z_y} W_y$, where $W_y$ is the fiber of $W$ over $y$. Define $I(X_y,W_y):= \sum_{s_y \in S_y} I_{s_y}(X_y,W_y)$, where $I_{s_y}(X_y,W_y):=\dim_{\mathbb{C}}\mathcal{O}_{S_y,s_y}=\dim_{\mathbb{C}}\mathcal{O}_{Z_y,s_y}/(I_{X_y}+I_{W_y})$ is the intersection multiplicity of $X_y$ and $W_y$ at $s_y$; it is $0$ if $X_y$ and $W_y$ do not intersect.

\begin{theorem}\label{constancy} In the setup above, $S$ is flat over $Y$ and $W_y=\overline{Z_y\setminus X_y}$
for all $y$. In particular, the function $y\mapsto I(X_y,W_y)$ is constant for all $y$ sufficiently close to $0$. 
\end{theorem}
\begin{proof}
Denote by $I_X$ and $I_W$ the ideals of $X$ and $W$ in $\mathcal{O}_{Z,0}$. We first record the primary decomposition that will be used throughout. Because $(Z,0)$ is Gorenstein, the primary decomposition $(0)=\bigcap_{i}Q_i$ in $\mathcal{O}_{Z,0}$ has no embedded components; the same holds in $\mathcal{O}_{X,0}$, which is Cohen--Macaulay. Let $x$ be a point of $X_y$ at which $Z_y=X_y$, and let $\mathfrak{n}_y$ be the maximal ideal of $\mathcal{O}_{Y,y}$. Since $\mathcal{O}_{X}$ is flat over $\mathcal{O}_Y$, we have $I_X \cap \mathfrak{n}_y\mathcal{O}_{Z} = \mathfrak{n}_y I_X$ at $x$, so $I_X\mathcal{O}_{Z_y,x}=0$ gives $(I_X)_x=\mathfrak{n}_y(I_X)_x$ and $(I_X)_x=0$ by Nakayama's lemma, i.e.\ $Z=X$ at $x$. As $X_y$ is a union of irreducible components of $Z_y$ for each $y$, the analytic set of points of $X$ at which $Z \neq X$ meets each fiber $X_y$ in a finite set, so it has dimension at most $\dim Y<\dim X$. Hence $Z=X$ at the generic points of $X$, and therefore $I_X=\bigcap_{i \in E}Q_i$ for a subset $E$ of the primary components, $I_W=\bigcap_{i \notin E}Q_i=(0:_{\mathcal{O}_{Z,0}}I_X)$, $I_X \cap I_W=(0)$ and $I_XI_W=0$. In particular, $X$ and $W$ have no irreducible component in common. Denote by $\mathfrak{n}$ the maximal ideal of $\mathcal{O}_{Y,0}$. Identify $\mathfrak{n}$ with its image in $\mathcal{O}_{Z,0}$. We have $\mathcal{O}_{Z_0,0}=\mathcal{O}_{Z,0}/\mathfrak{n}\mathcal{O}_{Z,0}$. Denote by $I_{X_0}$ and $I_{W_0}$ the images of $I_X$ and $I_W$ in $\mathcal{O}_{Z_0,0}$, respectively. Set $W_{0}':=\overline{Z_0\setminus X_0}$. Since $(Z,0) \rightarrow (Y,0)$ is flat, and $(Z,0)$ and $(Y,0)$ are Cohen--Macaulay, by \cite[Thm.\ 23.3]{Matsumura} we have $\mathrm{depth}(\mathcal{O}_{Z_0,0})=1$. Hence the associated primes of $\mathcal{O}_{Z_0,0}$ are all minimal, and, by the hypothesis on $X_0$, the primary decompositions of $I_{X_0}$ and $I_{W_0'}$ provide the primary decomposition of $(0)$ in $\mathcal{O}_{Z_0,0}$ via the identity $(0)= I_{X_0} \cap I_{W_0'}$. Thus, by prime avoidance, there exists $\overline{a} \in I_{X_0}$ such that $\overline{a}$ avoids the minimal primes of $I_{W_0'}$. So if $\bar{a}\bar{b}=0$ in $\mathcal{O}_{Z_0,0}$, then $\bar{b} \in I_{W_0'}$. But $I_XI_W=0$ (see above), so $I_{X_0}I_{W_0}=\bar{0}$ in $\mathcal{O}_{Z_0,0}$. Therefore, $I_{W_0} \subset I_{W_0'}$. We will prove that $S$ is flat over $Y$ and that the formation of $W$ specializes with passage to the fibers of $(X,0) \rightarrow (Y,0)$ by induction on $\dim Y$. 
 
Assume $\dim Y=1$. Set $\mathfrak{n}=(t)$. If $S$ is empty, there is nothing to prove, so suppose otherwise. Observe that the ideal of $(S,0)$ in $\mathcal{O}_{Z,0}$ is $I_X+I_W$. We will show that $t$ is a nonzerodivisor in $\mathcal{O}_{S,0}=\mathcal{O}_{Z,0}/(I_X+I_W)$. 
Suppose there exists $z \in \mathcal{O}_{Z,0}$ such that $tz=x+w$, with $x \in I_X$ and $w \in I_W$. Denote by $\bar{x}$ and $\bar{w}$ the images of $x$ and $w$ in $\mathcal{O}_{Z_0,0}$. We have $\bar{x}+\bar{w}=\bar{0}$. Note that $\bar{w} \in I_{W_0} \subset I_{W_0'}$ and $\bar{x}+\bar{w} \in I_{W_0'}$. So $\bar{x} \in I_{W_0'}$. But $\bar{x} \in I_{X_0}$, so $\bar{x}=\bar{0}$. Thus we can write $x=tx_1$ for some $x_1 \in \mathcal{O}_{Z,0}$. Because $X$ is flat over $Y$, we have that $t$ is a nonzerodivisor in $\mathcal{O}_{X,0}$. Hence $x_1 \in I_X$. Now $tz=x+w$ gives $t(z-x_1)=w$. Because $(Z,0) \rightarrow (Y,0)$ is flat, $t$ avoids the associated primes of $\mathcal{O}_{Z,0}$. Because $I_W$ is obtained from the primary decomposition of $(0)$ by removing $I_X$, we get that $t$ avoids the associated primes of $I_W$. Thus, there exists $w_1 \in I_W$ such that $z-x_1=w_1$, i.e.\ $z=x_1+w_1 \in I_X+I_W$. So, $t$ is a nonzerodivisor in $\mathcal{O}_{S,0}$ and hence $(S,0)$ is flat over $(Y,0)$.

Next we prove that the formation of $W$ specializes with passage to the fibers of $(X,0) \rightarrow (Y,0)$.
We will prove it below for $y=0$, i.e.\  we will show that $W_0=W_{0}'$. The same argument applies to every $y$ close enough to $0$ and every $s \in S_y$: by \cite{Frisch} the sets of points at which $(X,0) \rightarrow (Y,0)$ and $(Z,0) \rightarrow (Y,0)$ are not flat are closed analytic subsets not containing $0$, so after shrinking the representatives both morphisms are flat at every point; $X$ is Cohen--Macaulay near $0$, so every fiber $X_y$ is Cohen--Macaulay and has depth one at $s$; $Z_y$ is Gorenstein; and $(X_y,s)$ is a union of irreducible components of $(Z_y,s)$ by hypothesis. Thus we can apply what follows to the germs $(X,s) \rightarrow (Y,y)$ and $(Z,s) \rightarrow (Y,y)$. We have already established above the inclusion $I_{W_0} \subset I_{W_0'}$. Thus to prove that $W_0=W_{0}'$ it is enough to show  $I_{W_0'} \subset I_{W_0}$.
 
The ring $\mathcal{O}_{X,0}$ is Cohen--Macaulay because $\mathrm{depth}(\mathcal{O}_{X_0,0})=1$ and $(X,0) \rightarrow (Y,0)$ is flat. Since $I_W=(0:_{\mathcal{O}_{Z,0}}I_X)$, by \cite[Thm.\ 21.23 (b) (Linkage)]{Eis95} $\mathcal{O}_{W,0}$ is Cohen--Macaulay because $\mathcal{O}_{Z,0}$ is Gorenstein. So $\mathrm{depth}(\mathcal{O}_{W,0}) =2$. Because $t$ avoids the associated primes of $I_W$ (see above), $t$ is a nonzerodivisor in $\mathcal{O}_{W,0}$, so $\mathcal{O}_{W,0}$ is a flat $\mathcal{O}_{Y,0}$-module and $\mathrm{depth}(\mathcal{O}_{W_0,0})=1$. In particular, the maximal ideal in $\mathcal{O}_{Z_0,0}$ is not an associated prime of $I_{W_0}$ because $\mathcal{O}_{W_0,0}=\mathcal{O}_{Z_0,0}/I_{W_0}$.  Observe that set-theoretically $Z_0 = X_0 \cup W_0$. Since $S$ is flat over $Y$ and $\dim S \leq \dim Y$ (because $X$ and $W$ have no irreducible component in common), the fiber $S_0=X_0 \cap W_0$ is finite, so $W_0$ contains no irreducible component of $X_0$. This shows that $W_0$ and $W_{0}'$ are set-theoretically equal. By prime avoidance we can select $r \in I_{X}$ such that $r$ avoids the minimal primes of $I_W$ and such that its image $\bar{r}$ in $\mathcal{O}_{Z_0,0}$ avoids the associated primes of $I_{W_0}$. Thus $\bar{r}$ is a nonzerodivisor in $\mathcal{O}_{W_0,0}$. Therefore, $(t,r)$ is a regular sequence in $\mathcal{O}_{W,0}$ and so is $(r,t)$ because $\mathcal{O}_{W,0}$ is local. 
 
We claim that $t$ is a nonzerodivisor in $\mathcal{O}_{Z,0}/r\mathcal{O}_{Z,0}$. Suppose there exists $z \in \mathcal{O}_{Z,0}$ such that $tz \in r\mathcal{O}_{Z,0}$. Because $r \in I_X$ and $t$ is a nonzerodivisor in $\mathcal{O}_{X,0}$ we have $z \in I_X$. Also, we have $tz=0$ in $\mathcal{O}_{Z,0}/(r,I_{W})=\mathcal{O}_{W,0}/r\mathcal{O}_{W,0}$. Because $(r,t)$ is a regular sequence in $\mathcal{O}_{W,0}$, we get that $z=rz_1+w$ where $z_1 \in \mathcal{O}_{Z,0}$ and $w \in I_W$. But $z$ and $r$ belong to $I_X$. Thus $w \in I_X$. But $I_X \cap I_W = (0)$. Therefore, $w=0$ and so $z=rz_1$ which establishes our claim. 
 
Let $v \in \mathcal{O}_{Z,0}$ such that its image $\bar{v}$ in $\mathcal{O}_{Z_0,0}$ is in $I_{W_0'}$. Since $\bar{r} \in I_{X_0}$ and $I_{X_0} \cap I_{W_0'}=(0)$, we have $\bar{r}\bar{v}=0$ in $\mathcal{O}_{Z_0,0}$. Thus $rv=tz_2$ for some $z_2 \in  \mathcal{O}_{Z,0}$. By what we have just proved above, there exists $z_3 \in  \mathcal{O}_{Z,0}$ such that $z_2=rz_3$. Hence $r(v-tz_3)=0$. By our choice of $r$ we must have $v-tz_3 \in I_W$. Thus $\bar{v} \in I_{W_0}$ and so $I_{W_0'} \subset I_{W_0}$, which finishes the proof. 
 
Assume $\dim Y>1$. Because $(X,0) \rightarrow (Y,0)$ and $(Z,0) \rightarrow (Y,0)$ are flat, by \cite[Thm.\ 18.16 (a)]{Eis95} we have $\mathrm{depth}(\mathfrak{n},\mathcal{O}_{X,0})= \mathrm{depth}(\mathfrak{n},\mathcal{O}_{Z,0})=\dim Y$. Again, because  $\mathrm{depth}(\mathcal{O}_{X_0,0})=1$ and $(X,0) \rightarrow (Y,0)$ is flat, we have that $\mathcal{O}_{X,0}$ is Cohen--Macaulay  and so is $\mathcal{O}_{W,0}$  because $\mathcal{O}_{Z,0}$ is Gorenstein (\cite[Thm.\ 21.23 (b)]{Eis95}).  Since $\dim (W,0) = \dim (Z,0)=\dim Y+1$, by Krull's height theorem $\dim (W_0) \geq 1$. But $W_0 \subset Z_0$ and $\dim Z_0=1$. So $\dim W_0=1$. Then \cite[Thm.\ 18.16 (b) (Miracle Flatness)]{Eis95}  implies that $(W,0) \rightarrow (Y,0)$ is flat. In particular, $\mathrm{depth}(\mathfrak{n},\mathcal{O}_{W,0})=\dim Y$. Therefore, the local cohomology modules $\mathrm{H}_{\mathfrak{n}}^{i}(\mathcal{O}_{X,0})$, $\mathrm{H}_{\mathfrak{n}}^{i}(\mathcal{O}_{W,0})$  and  $\mathrm{H}_{\mathfrak{n}}^{i}(\mathcal{O}_{Z,0})$ vanish for $i<\dim Y$. Consider the short exact sequence of finitely generated $\mathcal{O}_{Z,0}$-modules
$$0 \rightarrow \mathcal{O}_{Z,0} \rightarrow \mathcal{O}_{X,0} \oplus \mathcal{O}_{W,0} \rightarrow \mathcal{O}_{S,0} \rightarrow 0.$$
It yields an exact sequence on local cohomology
$$\mathrm{H}_{\mathfrak{n}}^{i}(\mathcal{O}_{X,0} \oplus \mathcal{O}_{W,0}) \rightarrow \mathrm{H}_{\mathfrak{n}}^{i}(\mathcal{O}_{S,0})\rightarrow \mathrm{H}_{\mathfrak{n}}^{i+1}(\mathcal{O}_{Z,0}).$$
Thus $\mathrm{H}_{\mathfrak{n}}^{i}(\mathcal{O}_{S,0})=0$ for all $i \leq \dim Y-2$. In particular, $\mathrm{depth}(\mathfrak{n},\mathcal{O}_{S,0})\geq 1$ because $\dim Y \geq 2$.
For $y_1 \in \mathfrak{n}$ denote by $Y_1$ the subspace $\{y_1=0\}$ of $(Y,0)$. By prime avoidance, we can select $y_1$ with $y_{1} \not \in \mathfrak{n}^2$ so that $Y_1$ is smooth, $X_1:=X \times_{Y} Y_1$, $W_1:=W \times_{Y} Y_1$ are  Cohen--Macaulay,  $Z_1:=Z \times_{Y} Y_1$ is Gorenstein, and $y_1 \not \in \mathrm{z.div}(\mathcal{O}_{S,0})$. Thus $S_1=S \times_{Y} Y_1$ is of dimension strictly less than $\dim X_1=\dim W_1$, and so $X_1$ and $W_1$ do not share an irreducible component. We check that $(X_1,0) \rightarrow (Y_1,0)$ and $(Z_1,0) \rightarrow (Y_1,0)$ satisfy the hypotheses of the theorem: the fibers over the points of $Y_1$ are the $X_y$ and $Z_y$; and $X_1$ is a union of irreducible components of $Z_1$ in the sense of the setup, because the analytic set of points of $X$ at which $Z \neq X$ is finite over $Y$ of dimension at most $\dim Y$ (see the beginning of the proof), so its intersection with $X_1$ has dimension at most $\dim Y-1<\dim X_1$, while $X_1$ and $Z_1$ have no embedded components. Set $W_1':=\overline{Z_1 \setminus X_1}$. As $X_1$,  $W_1$ and $Z_1$ are Cohen--Macaulay, the argument given in the case $\dim Y=1$ for the inclusion $I_{W_0'} \subset I_{W_0}$, with $t$, $Z_0$, $X_0$, $W_0$, $W_0'$ replaced by $y_1$, $Z_1$, $X_1$, $W_1$, $W_1'$ (its ingredients are available: $W \rightarrow Y$ is flat, $W_1$ is Cohen--Macaulay and hence unmixed, and $W_1=W_1'$ set-theoretically because $\dim S_1<\dim X_1$), shows that we can construct a regular sequence $(r,y_1)$ in $\mathcal{O}_{W,0}$ with $y_1$ a nonzerodivisor on $\mathcal{O}_{Z,0}/r\mathcal{O}_{Z,0}$, which yields that ${W_1}=W_1'$ in $\mathcal{O}_{Z_1,0}$. Consequently $X_1 \times_{Z_1} W_1'=S \times_{Y} Y_1=S_1$, and by induction $S_1$ is flat over $Y_1$. Because $y_1 \not \in \mathrm{z.div}(\mathcal{O}_{S,0})$, by \cite[Cor.\ 6.9]{Eis95} $S$ is flat over $Y$. By induction, $(W_1')_{0}=\overline{(Z_1)_{0}\setminus (X_1)_{0}}$, and $(W_1')_0=(W_1)_0$. But $X_{0}=(X_1)_{0}$, $W_{0}=(W_1)_{0}$ and $Z_{0}=(Z_1)_{0}$, so $W_{0}=\overline{Z_{0}\setminus X_{0}}$.
 
Because $X$ and $W$ have no irreducible component in common, $\dim S < \dim Z=\dim Y+1$. Thus $\dim S \leq \dim Y$. Because $(S,0) \rightarrow (Y,0)$ is flat, by \cite[\href{https://stacks.math.columbia.edu/tag/00R5}{Tag 00R5}]{Stacks} $\dim (S,0)=\dim(Y,0)$ and the fiber $S_0$ is zero-dimensional, so $(S,0)$ is Cohen--Macaulay by \cite[Thm.\ 23.3]{Matsumura}. So $(S,0) \rightarrow (Y,0)$ is quasi-finite. Therefore, by taking sufficiently small representatives of the germs $(X,0)$ and $(Y,0)$ we can ensure that $(S,0) \rightarrow (Y,0)$ is finite (see \cite[Prop.\ 1), p. 63]{GrR84}). So $\mathcal{O}_{S,0}$ is a finite and a flat $\mathcal{O}_{Y,0}$-module, and $I(X_y,W_y)=\dim_{\mathbb{C}}\mathcal{O}_{S_y}$ is its rank. Therefore, $y \mapsto  I(X_y,W_y)$ is constant for all $y$ close to $0$.
\end{proof}
 
In our applications $(X_0,0)$ is a reduced curve defined in $(\mathbb{C}^n,0)$ by the vanishing of analytic functions $\tilde{f_1}, \ldots, \tilde{f_p}$.  Let $A \in \mathrm{Mat}((n-1)\times p,\mathbb{C})$ and let $(Z_0,0)\subset (\mathbb{C}^n,0)$ be the analytic set defined by the vanishing of $A(\tilde{f_1}, \ldots, \tilde{f_p})^{T}$. Obviously, $(Z_0,0)$ contains $(X_0,0)$ as a closed subspace. We say that $(Z_0,0)$ is {\it reduced along} $(X_0,0)$ if $Z_0$ is reduced at the generic point of each irreducible component of $X_0$, i.e.\ if $X_0$ is a union of irreducible components of $Z_0$ in the sense of the setup above. For a generic $A$ we show below in Prop.\ \ref{reducedness} that $(Z_0,0)$ is a complete intersection which is reduced along $(X_0,0)$.
 
Fix such an $A$. Observe that any $(X,0) \rightarrow (Y,0)$ flat embedded deformation of $(X_0,0)$ with $(Y,0)$ smooth and $(X,0) \subset (\mathbb{C}^n\times Y,0)$, together with a choice of equations $f_1, \ldots, f_p$ of $(X,0)$ in $(\mathbb{C}^n\times Y,0)$ restricting to $\tilde{f_1}, \ldots, \tilde{f_p}$, induces an embedded deformation on the complete intersection $(Z_0,0)$ which we denote by $(Z,0) \rightarrow (Y,0)$: $(Z,0)$ is cut out from  $(\mathbb{C}^n\times Y,0)$ by $A(f_1, \ldots, f_p)^{T}$. It is a complete intersection in $(\mathbb{C}^n\times Y,0)$ with one-dimensional fiber $Z_0$, hence Cohen--Macaulay of dimension $\dim Y+1$ and flat over $(Y,0)$.
After shrinking the representatives we may assume that the fibers $X_y$ are reduced (the set of points of $X$ at which the fiber of the flat map $(X,0) \rightarrow (Y,0)$ is reduced is Zariski open and contains the germ $X_0$). Then $Z_y$ is reduced along $X_y$ for every $y$ close enough to $0$: the set $B$ of points $x \in X$ at which $Z_{h(x)}$ is singular is analytic, and $B_0$ is finite because $(Z_0,0)$ is reduced along $(X_0,0)$, so $B$ is finite over $Y$ near $0$; hence $Z_y$ is smooth, and therefore equal to the reduced curve $X_y$, at the generic points of $X_y$. In particular, the hypotheses of Thm.\ \ref{constancy} are satisfied.
\begin{corollary}\label{computing cid} Suppose $(X,0) \rightarrow (Y,0)$ is a smoothing, i.e.\ $X_y$ is smooth for generic $y \neq 0$. Identify the Jacobian ideal $\mathrm{Jac}(Z_y)$ of $Z_y$ with its image in $\mathcal{O}_{X_y}$. Then $I_0(X_0,W_0)=\dim_{\mathbb{C}}\mathcal{O}_{X_y}/\mathrm{Jac}(Z_y)$ for $y \neq 0$ close enough to $0$ such that $X_y$ is smooth.
\end{corollary}
\begin{proof} By Thm.\ \ref{constancy} $I_0(X_0,W_0)=I(X_y,W_y)$ and $W_y=\overline{Z_y \setminus X_y}$. Because $X_y$ is smooth and $Z_y$ is Cohen--Macaulay and reduced along $X_y$, the points of $X_y \cap W_y$ are exactly the points of $X_y$ at which $Z_y$ is singular, i.e.\ the points of $\mathbb{V}(\mathrm{Jac}(Z_y)) \cap X_y$. Thus by \cite[Cor.\ 3.10]{BR24} we have $I(X_y,W_y)=\sum_{x \in X_y \cap W_y}\dim_{\mathbb{C}}\mathcal{O}_{X_y,x}/\mathrm{Jac}(Z_y,x)=\dim_{\mathbb{C}}\mathcal{O}_{X_y}/\mathrm{Jac}(Z_y)$.
\end{proof}

 
\begin{proposition}\label{reducedness}
Suppose $(X_0,0)$ is reduced and suppose $A$ is generic. Then $(Z_0,0)$ is a complete intersection which is reduced along $(X_0,0)$.
\end{proposition}
\begin{proof}
Successive applications of prime avoidance imply that $(Z_0,0)$ is a complete intersection for generic $A$. Next, observe that the following relation holds between the Jacobian matrices of $(X_0,0)$ and $(Z_0,0)$: $J(Z_0)_{|X_0}=AJ(X_0)$. Let $(X_0,0)=\bigcup_{i=1}^s X_i$ be the decomposition of $(X_0,0)$ into irreducible
components. For each $i=1,\ldots,s$ pick a smooth point $x_i\in X_i$. Because $(X_0,0)$ is reduced, $J(X_0)$ evaluated at $x_i$ has rank $n-1$, and so by the Cauchy--Binet formula for generic $A$ we have that $J(Z_0)$ evaluated at $x_i$ is of maximal rank $n-1$. Thus $Z_0$ is smooth, hence reduced, at $x_i$; since the non-reduced locus of $Z_0$ is closed, $Z_0$ is reduced at the generic point of each $X_i$, i.e.\ $(Z_0,0)$ is reduced along $(X_0,0)$.
\end{proof}
\textbf{Remarks.} The Bertini argument of \cite[Sct.\ 6.3]{BR24} yields a stronger result: for generic $A$ the complete intersection $Z_0$ is reduced.
In \cite[Cor.\ 5.4]{BR24} we prove that $y \mapsto I(X_y,W_y)$ is constant for projective curves using a formula for the arithmetic genus of a projective curve.

\section{Proof of Theorem \ref{main mpt}. The relative Nash blowup}\label{main proof}
 
\subsection{Proof of Theorem \ref{main mpt}}\label{proof mpt} Let $h \colon (X,0) \rightarrow (Y,0)$ be an embedded flat deformation of a reduced curve $(X_0,0) \subset (\mathbb{C}^n,0)$ with $Y$ smooth of dimension $k$. Suppose $(X,0)$ is cut out from $(\mathbb{C}^n \times Y,0)$ by the vanishing of the analytic functions $f_1, \ldots, f_p$. Let $Z$ be the complete intersection in $(\mathbb{C}^n \times Y,0)$ with equations given by
$A(f_1, \ldots, f_p)^{T}$ for some general $A \in \mathrm{Mat}((n-1)\times p,\mathbb{C})$.  Set $W:=\overline{Z\setminus X}$ and set $S:=X \times_{Z} W$. By Prop.\ \ref{reducedness}, $Z_0$ is reduced along $X_0$, and so is $Z_y$  along $X_y$ for each $y \in Y$ after $Y$ is replaced by a smaller neighborhood of $0$. We also shrink the representatives so that the fibers $X_y$ are reduced (the set of points of $X$ at which the fiber of the flat map $h$ is reduced is Zariski open and contains the germ $X_0$); in particular $X_y$ is a union of irreducible components of $Z_y$ for each $y$, as required in Thm.\ \ref{constancy}. Note that $Z \rightarrow Y$ is flat, because $Z$ is a complete intersection in $(\mathbb{C}^n \times Y,0)$ whose fiber $Z_0$ is a curve, and that $X$ is Cohen--Macaulay, being flat over $Y$ with Cohen--Macaulay fibers. Finally, the relative singular locus $T$ of $Z$ along $X$, i.e.\ the set of points $x \in X$ at which $Z_{h(x)}$ is singular, satisfies $T_0=(X_0)_{\mathrm{sing}} \cup (X_0 \cap W_0)=\{0\}$ as a germ, so we may assume that $T$ is finite over $Y$.
 
Denote by $J_{h}(Z)$ the $\mathcal{O}_X$-module generated by the columns of the relative Jacobian matrix of $Z$ restricted to $X$. By Prop.\ \ref{conormal} $C_{h}(X)=\mathrm{Projan}(\mathcal{R}(J_{h}(Z)))$ where $\mathcal{R}(J_{h}(Z))$ is the Rees algebra of $J_{h}(Z)$. 
 
In terms of presentation matrices, we have $[J_{h}(Z)]= A[J_{h}(X)]$.
For each $y$ denote by $J(Z_y)$ the image of $J_{h}(Z)$ in $\mathcal{O}_{X_y}^{n-1}$. It is the $\mathcal{O}_{X_y}$-module generated by the columns of the Jacobian matrix of $Z_y$ restricted to $X_y$. Denote by $e(J(Z_y),\mathcal{O}_{X_y}^{n-1})$ the Buchsbaum--Rim multiplicity of the pair $J(Z_y) \subset \mathcal{O}_{X_y}^{n-1}$. Applying Thm.\ \ref{MPT} to $N:=J_h(Z) \subset F:=\mathcal{O}_X^{n-1}$, for which $\mathrm{Supp}_X(F/N)=T$ and $\Gamma^{k}(N)=\Gamma_{h}^{k}(X)$, we get
 
\begin{equation}\label{MPT-CID}
e(J(Z_0),\mathcal{O}_{X_0}^{n-1})-\sum_{x \in (Z_y)_{\mathrm{sing}} \cap X_y}e(J(Z_y,x),\mathcal{O}_{X_y,x}^{n-1})= \mathrm{deg}_{Y}  \Gamma_{h}^{k}(X)
\end{equation}
where $J(Z_y,x)$ is the stalk of $J(Z_y)$ at $x$, and where $(Z_y)_{\mathrm{sing}} \cap X_y=(X_y)_{\mathrm{sing}} \sqcup ((X_y)_{\mathrm{sm}} \cap W_y)$ because $Z_y$ is Cohen--Macaulay and generically reduced along $X_y$.
For $x \in (X_y)_{\mathrm{sing}}$ the Buchsbaum--Rim multiplicity $e(J(Z_y,x),\mathcal{O}_{X_y,x}^{n-1})$ equals the Hilbert--Samuel multiplicity $e(\mathrm{Jac}(Z_y,x))$ of the ideal of maximal minors of $[J(Z_y,x)]$, because $(X_y,x)$ is a reduced curve (see the proof of Cor.\ \ref{jacobian ideal}); for $x \in (X_y)_{\mathrm{sm}}$ it equals $\dim_{\mathbb{C}}\mathcal{O}_{X_y,x}/\mathrm{Jac}(Z_y,x)$. Since $Z_y$ is generically reduced along $X_y$ and $W_y=\overline{Z_y \setminus X_y}$ by Thm.\ \ref{constancy}, the identities (10) and (11) of \cite[Thm.\ 3.9]{BR24} give
$$e(\mathrm{Jac}(Z_y,x))-I_x(X_y,W_y)=e(\mathrm{Jac}(X_y,x))-\mathrm{cid}(X_y,x) \ \ \text{for} \ x \in (X_y)_{\mathrm{sing}},$$
where $I_x(X_y,W_y):=0$ if $x \notin W_y$, and \cite[Cor.\ 3.10]{BR24} gives $\dim_{\mathbb{C}}\mathcal{O}_{X_y,x}/\mathrm{Jac}(Z_y,x)=I_x(X_y,W_y)$ for $x \in (X_y)_{\mathrm{sm}} \cap W_y$. Summing over $(Z_y)_{\mathrm{sing}} \cap X_y=(X_y)_{\mathrm{sing}} \sqcup ((X_y)_{\mathrm{sm}} \cap W_y)$ we obtain, for every $y$ close enough to $0$,
\begin{equation}\label{MPT2}
\sum_{x \in (Z_y)_{\mathrm{sing}} \cap X_y}e(J(Z_y,x),\mathcal{O}_{X_y,x}^{n-1})=e(\mathrm{Jac}(X_y))-\mathrm{cid}(X_y)+I(X_y,W_y),
\end{equation}
where $e(\mathrm{Jac}(X_y)):=\sum_{x \in (X_y)_{\mathrm{sing}}}e(\mathrm{Jac}(X_y,x))$ and $I(X_y,W_y)=\sum_{x \in X_y \cap W_y}I_x(X_y,W_y)$. In particular, for $y=0$,
\begin{equation}\label{flatness-cid}
e(J(Z_0),\mathcal{O}_{X_0}^{n-1})=e(\mathrm{Jac}(X_0))-\mathrm{cid}(X_0)+I_0(X_0,W_0).
\end{equation}
By Thm.\ \ref{constancy},
\begin{equation}\label{flatness}
I_{0}(X_0,W_0)=I(X_y,W_y)
\end{equation}
for all $y$ close enough to $0$. Substituting (\ref{MPT2}), (\ref{flatness-cid}) and (\ref{flatness}) in (\ref{MPT-CID}) we obtain
$$
e(\mathrm{Jac}(X_0)) - \mathrm{cid}(X_0)-\bigl( e(\mathrm{Jac}(X_y))-\mathrm{cid}(X_y)\bigl)=\mathrm{deg}_{Y}  \Gamma_{h}^{k}(X),$$
which proves (\ref{mpt-cid}). We remark that for general $A$ the intersection numbers at the singular points are the complete intersection discrepancies themselves: $I_0(X_0,W_0)=\mathrm{cid}(X_0)$ by Cor.\ \ref{jacobian ideal} and \cite[Thm.\ 3.9]{BR24}, and $I_x(X_y,W_y)=\mathrm{cid}(X_y,x)$ for all $x \in (X_y)_{\mathrm{sing}}$ and $y$ in a Zariski open subset of $Y$, by Prop.\ \ref{int.clos.} applied to $V=X$, $M=J_h(X)$, $M_e=J_h(Z)$ and the $k$-dimensional irreducible components of the relative singular locus of $h$ (the ideals $\mathrm{Jac}_h(Z)$ and $\mathrm{Jac}_h(X)$ then have the same integral closure off a proper analytic subset of that locus, hence so do their restrictions $\mathrm{Jac}(Z_y,x) \subset \mathrm{Jac}(X_y,x)$ to the fibers, and the Rees criterion gives $e(\mathrm{Jac}(Z_y,x))=e(\mathrm{Jac}(X_y,x))$). In general one only has $I_x(X_y,W_y) \geq \mathrm{cid}(X_y,x)$.
\begin{corollary}\label{main} Suppose $(X_0,0) \subset (\mathbb{C}^n,0)$
is a reduced smoothable curve. Then 
\begin{equation}
e(\mathrm{Jac}(X_0)) - \mathrm{cid}(X_0)= \mu_0+ m_0 -1.
\end{equation}
\end{corollary}
\begin{proof} Let $h:(X,0) \rightarrow (Y,0)$ be a smoothing of $(X_0,0)$ over a smooth curve germ $(Y,0)$, i.e.\ an embedded flat deformation with $X_y$ smooth for $y \neq 0$ (restrict a smoothing of $(X_0,0)$ to a smooth curve through $0$ in its base whose general points have smooth fibers). Then $e(\mathrm{Jac}(X_y))=\mathrm{cid}(X_y)=0$ for $y \neq 0$, and Thm.\ \ref{main mpt} gives $e(\mathrm{Jac}(X_0))-\mathrm{cid}(X_0)= \mathrm{deg}_{Y}  \Gamma_{h}^{1}(X)$. The proof now follows from Prop.\ \ref{Euler}.
\end{proof}
 
 
\begin{example}\label{3,4,5}\rm{Consider the space curve $X_0$ given by the following equations $f_1=x^2y-z^2,f_2=x^3-yz,f_3=y^2-xz$, which come from the $2\times2$ minors of the matrix
$\begin{pmatrix} x & y & z \\ y & z & x^2 \end{pmatrix}.$ 
The parametrization for $X_0$ is given by $\{(t^3,t^4,t^5)|t\in \mathbb{C}\}\subset \mathbb{C}^3$. It is a well-known
example of a non-planar curve, which is Cohen--Macaulay of codimension $2$, but not Gorenstein. Its Jacobian matrix is 
\[\begin{pmatrix}
    2xy & x^2 & -2z \\ 3x^2 & -z & -y \\ -z & 2y & -x
\end{pmatrix}.\]
Let $f_0=2y^2+xz+3yz+z^2-4x^3-7x^2y-2xy^2-7x^2z-2xyz-3x^4$. It is a generic combination of the $2\times2$ minors of the Jacobian matrix. Then (via a \textsc{Singular} computation \cite{DGPS}) $e(\mathrm{Jac}(X_0))=\dim_{\mathbb{C}} \mathcal{O}_{X_0,0}/(f_0) = 8.$
 
To compute $\mathrm{cid}(X_0)$ we use the complete intersection
$Z_0=\mathbb{V}(f_1+f_2,f_1+f_3)$, which gives
a generic complete intersection, such that $e(J(Z_0),\mathcal{O}_{X_0}^2)=e(\mathrm{Jac}(X_0))$, where $\mathcal{O}_{X_0}^2/J(Z_0)$ has presentation matrix $[J(Z_0)]:=\begin{pmatrix} 1 & 1 & 0 \\ 1 & 0 & 1 \end{pmatrix}[J(X_0)]$. By a quotient ideal computation we find $W_0=\mathbb{V}(x+y+z,y+z+x^2)
\ \text{and} \ 
I_{0}(X_0, W_0)=\dim_{\mathbb{C}}\mathcal{O}_{\mathbb{C}^3,0}/(I(X_0)+I(W_0))=2;$
as a germ at $0$, $W_0$ is the line $\mathbb{V}(x,y+z)$. Because $e(J(Z_0),\mathcal{O}_{X_0}^2)=e(\mathrm{Jac}(X_0))$, Cor.\ \ref{jacobian ideal} and \cite[Thm.\ 3.9]{BR24} give $\mathrm{cid}(X_0)=I_{0}(X_0, W_0)$.
Thus 
$$\mu (X_0) = e(\mathrm{Jac}(X_0)) - I_{0}(X_0,W_0) - m +1=8-2-3+1 = 4.$$
Now choose $Z'_0=\mathbb{V}(f_1,f_3)$. Denote by $J(Z_0)'$ the submodule in $\mathcal{O}_{X_0}^2$ generated by the columns of the matrix obtained from the Jacobian matrix of $X_0$ by erasing the second row.
Then $$e(J(Z_0)',\mathcal{O}_{X_0}^2)=9>e(J(Z_0),\mathcal{O}_{X_0}^2).$$ But $W_0'=\mathbb{V}(y,z)$. So $I_{0}(X_0, W_0')=3$. 
A smoothing of $X_0$ is given by the $2 \times 2$ minors of
$\begin{pmatrix} x & y & z \\ y & z & x^2-t \end{pmatrix}.$ Then $X \times _{\mathbb{C}^{4}} W' = \mathbb{V}(x,y,z) \cup \mathbb{V}(y,z,x^2-t)$ and $X_t$ and $W'_t$ meet locally transversally at $3$ points
for $t \neq 0$. As predicted by \cite[Thm.\ 3.9]{BR24} we have  $$e(J(Z_0),\mathcal{O}_{X_0}^2) - I_{0}(X_0,W_0)=e(J(Z_0)',\mathcal{O}_{X_0}^2)-I_{0}(X_0, W_0')=6.$$
 
Finally, observe that for the special choice $Z''_0 = \mathbb{V}(f_1,f_2)$, in the induced deformation of $Z''$ by the smoothing of $X_0$ above, we have $W''=\mathbb{V}(z,x^2-t)$ and $X \times _{\mathbb{C}^{4}} W'' = \mathbb{V}(z,y^2,x^2-t)$. For $t\neq 0$, one finds that $W_t''$ consists of two lines each tangent to $X_t$, i.e. $X_t$ and $W''_t$ do not intersect locally transversally. As $t \rightarrow 0$, these two lines degenerate to the double line $W''_{0}=\mathbb{V}(z,x^2)$.}
\end{example}
 
\subsection{The relative Nash blowup}
Preserve the setup from the beginning of Sct.\ \ref{proof mpt}. The relative Nash blowup $N_h(X)$ is defined as the closure in $X \times \mathrm{Gr}(1,n)$ of the pairs $(x,T_xX_{h(x)})$ where $x$ is a smooth point in $X_{h(x)}$ and $T_{x}X_{h(x)}$ is the tangent space to $X_{h(x)}$ at $x$. Denote by $\mathrm{Jac}_h(Z)$ the relative Jacobian ideal of $Z$ in $\mathcal{O}_{X,0}$, i.e.\ the ideal of $(n-1) \times (n-1)$ minors of $J_h(Z)$. The argument of \cite[Prop.\ 1.2.1 c)]{Teissier} shows that $N_h(X) \cong \mathrm{Bl}_{\mathrm{Jac}_h(Z)}(X)$: the relative Gauss map $x \mapsto T_xX_{h(x)}$ is given, in Pl\"ucker coordinates, by the maximal minors of the relative Jacobian matrix of $Z$, because these span the kernel of an $(n-1) \times n$ matrix of rank $n-1$.

Set $B:=\mathrm{Bl}_{\mathrm{Jac}_h(Z)}(X)$ and denote by $D$ its exceptional divisor. Denote by $B_{y}$ and $D_y=b_h^{-1}(T_y)$ the fibers of $B$ and $D$ over $y \in Y$, and set $B(y):=\mathrm{Bl}_{\mathrm{Jac}(Z_y)}(X_y)$. Since $\mathrm{Jac}_h(Z)\mathcal{O}_{X_y}=\mathrm{Jac}(Z_y)$, $B(y)$ is a closed subspace of $B_y$, and by the same argument $B(y)=N(X_y)$ for every $y$ close enough to $0$. Define the {\it vertical cycle} of $B$ over $0$ as the difference of fundamental cycles
$$[D_{\mathrm{vert}}]:=[B_0]-[B(0)].$$
Because $B$ and $X$ are isomorphic off $T$ and $X_0$ is reduced, $[B_0]$ has multiplicity one along each irreducible component of $B(0)$; hence $[D_{\mathrm{vert}}]$ is an effective cycle supported on the irreducible components of $B_0$ not contained in $B(0)$. These lie in $b_h^{-1}(0) \subset \{0\} \times \mathrm{Gr}(1,n)$, so they are compact, and each has dimension at least one; when $\dim Y=1$ they are one-dimensional. Set $l:=c_{1}\mathcal{O}_{B}(1)$ and, when $\dim Y=1$, $\mathrm{deg}(D_{\mathrm{vert}}):=\int l\,[D_{\mathrm{vert}}]$. Finally, set $e(\mathrm{Jac}(Z_y)):=\sum_{x \in T_y}e(\mathrm{Jac}(Z_y,x))$, which equals the left-hand side of (\ref{MPT2}).

\begin{proposition}\label{vertical} The following holds. 
\begin{enumerate}
\item[\rm{(i)}] Assume $\dim Y=1$. For $y \neq 0$ close enough to $0$ we have
$$e(\mathrm{Jac}(X_0)) - \mathrm{cid}(X_0)-\bigl( e(\mathrm{Jac}(X_y))-\mathrm{cid}(X_y)\bigr)=\mathrm{deg}(D_{\mathrm{vert}}).$$
\item[\rm{(ii)}] The following are equivalent: {\rm (a)} $[B_0]=[B(0)]$, and then $[B_y]=[B(y)]$ for all $y \in (Y,0)$; {\rm (b)} $D \rightarrow Y$ has finite fibers; {\rm (c)} $y \mapsto e(\mathrm{Jac}(X_y))-\mathrm{cid}(X_y)$ is constant for $y$ in a small enough neighborhood of $0$. 
\end{enumerate}
\end{proposition}
\begin{proof} By the Excess--Degree Formula \cite[Thm.\ 2.1 (i)]{Rangachev} we have
$$e(\mathrm{Jac}(Z_0))-e(\mathrm{Jac}(Z_y))=\int l\,\bigl([B_0]-[B(0)]\bigr)=\mathrm{deg}(D_{\mathrm{vert}})$$
for $y \neq 0$ close enough to $0$. By (\ref{MPT2}), (\ref{flatness-cid}) and (\ref{flatness}), which hold for every $y$ close enough to $0$, the left-hand side equals $e(\mathrm{Jac}(X_0)) - \mathrm{cid}(X_0)-\bigl( e(\mathrm{Jac}(X_y))-\mathrm{cid}(X_y)\bigr)$ which proves \rm{(i)}.

For \rm{(ii)}, fix $y \in (Y,0)$. The spaces $B_y$ and $B(y)$ are isomorphic to $X_y$ off $D_y$, $B(y)$ is reduced, $B(y) \cap D_y$ is finite, and every irreducible component of $B_y$ has dimension at least one, because $B$ is of pure dimension $k+1$ and $Y$ is smooth of dimension $k$. Hence $[B_y] \neq [B(y)]$ if and only if $B_y$ has an irreducible component of dimension at least one contained in $D_y$ if and only if $D_y$ is not finite. In particular, (b) implies $[B_0]=[B(0)]$. By the identities used in \rm{(i)}, $y \mapsto e(\mathrm{Jac}(X_y))-\mathrm{cid}(X_y)$ is constant near $0$ if and only if $y \mapsto e(\mathrm{Jac}(Z_y))$ is, so $[B_0]=[B(0)]$ is equivalent to (c) by \cite[Thm.\ 2.1 (ii)]{Rangachev}. If (c) holds, then it holds near every $y \in (Y,0)$, and \cite[Thm.\ 2.1 (ii)]{Rangachev} applied at $y$ gives $[B_y]=[B(y)]$, hence $D_y$ is finite by the first part of the proof; this proves the second assertion in (a) and (c) $\Rightarrow$ (b).
\end{proof}

\textbf{Gorenstein curves}. Following \cite[Prop.\ 1, p.\ 508]{Pi} we can define the relative $\omega$-jacobian ideal $J_h$ as
$$J_h:=\mathrm{Ann}(\mathrm{Coker}(\Omega^1_{X/Y} \xrightarrow{} \omega_{X/Y}))$$
where $\Omega^1_{X/Y}$ denotes the sheaf of relative holomorphic $1$-forms of $X$ over $Y$ and $\omega_{X/Y}$ is the relative dualizing sheaf associated with the family $X \rightarrow Y$. 
As usual, set $W:=\overline{Z\setminus X}$. Denote by $I_W$ the ideal of $W$ in $\mathcal{O}_{X,0}$. 
The proof of \cite[Prop.\ 4.2]{BR24} applies verbatim to the flat family of complete intersections $Z \rightarrow Y$ (with $\omega_{Z/Y}$ invertible and the relative class map in place of the absolute ones), and one sees that
$$J_h=(\Jac_h(Z):_{\mathcal{O}_{X,0}} I_W).$$
Clearly, $J_h$ is supported over the relative singular locus of $h: X \rightarrow Y$ defined by the ideal $\mathrm{Jac}_h(X)$. Suppose $(X,0) \rightarrow (Y,0)$ is a one-parameter family of Gorenstein curves; then $X$ is Gorenstein, being flat over $Y$ with Gorenstein fibers, and the argument of \cite[Thm.\ 2, p. 516]{Pi} gives $N_h(X) \cong \mathrm{Bl}_{J_h}(X)$. Denote the image of $J_h$ in $\mathcal{O}_{X_y}$ by $J_{h}(y)$. Because $X$ is Gorenstein, $I_W$ is locally principal, generated by a nonzerodivisor $g$, so $J_h=g^{-1}\mathrm{Jac}_h(Z)$; since $W_y=\overline{Z_y \setminus X_y}$ by Thm.\ \ref{constancy}, the image $g_y$ of $g$ generates $I_{W_y}$ and $J_h(y)=g_y^{-1}\mathrm{Jac}(Z_y)$ is the $\omega$-jacobian ideal of $X_y$ (\cite[Prop.\ 4.2]{BR24}). Set $B^{\omega}:=\mathrm{Bl}_{J_h}(X)$ and $B^{\omega}(y):=\mathrm{Bl}_{J_{h}(y)}(X_y)$, and denote by $D^{\omega}(y)$ the exceptional divisor of $B^{\omega}(y)$. Each $D^{\omega}(y)$ is a zero-dimensional cycle $D^{\omega}(y)=\sum_{}m_{p_y}[p_y]$ in  $A_0(B^{\omega}(y))$ (the group of zero cycles modulo rational equivalence). Its degree is defined as $\mathrm{deg}(D^{\omega}(y)):=\sum_p m_{p_y}[k(p_y):\mathbb{C}]=\sum_{p_y} m_{p_y}$ (see \cite[Dfn.\ 1.4]{Ful}). By \cite[Cor.\ 4.3]{BR24} applied to $X_y$ and $J_h(y)$ we have $$\mathrm{deg}(D^{\omega}(y))= e(\mathrm{Jac}(X_y))-\mathrm{cid}(X_y).$$
As $X$-spaces, $B^{\omega}$ and $B$ are both $N_h(X)$, and $B^{\omega}(0)=B(0)=N(X_0)$; hence the vertical cycle $[B^{\omega}_0]-[B^{\omega}(0)]$ of $B^{\omega}$ over $0$ is $[D_{\mathrm{vert}}]$. On the other hand, $\mathrm{Jac}_h(Z)=gJ_h$ gives $\mathcal{O}_{B}(1)=\mathcal{O}_{B^{\omega}}(1) \otimes b_h^{*}I_W$, so $l^{\omega}:=c_{1}\mathcal{O}_{B^{\omega}}(1)$ differs from $l$ by the pullback of a line bundle from $X$, which is trivial on $b_h^{-1}(0)$; thus $\int l^{\omega}[D_{\mathrm{vert}}]=\mathrm{deg}(D_{\mathrm{vert}})$. Therefore, for one-parameter families of Gorenstein curves, \cite[Thm.\ 2.1 (i)]{Rangachev} applied to the ideal $J_h$, whose zero locus is finite over $Y$, gives
$$\boxed{\mathrm{deg}(D^{\omega}(0))-\mathrm{deg}(D^{\omega}(y))=\mathrm{deg}(D_{\mathrm{vert}})}$$
with $y \neq 0$: the right-hand side of Prop.\ \ref{vertical} \rm{(i)} is the difference of the degrees of the exceptional divisors of the Nash blowups of $X_0$ and $X_y$.

\section{A Riemann--Hurwitz formula for singular curves}\label{Sct-RH}

In this section we prove a Riemann--Hurwitz formula for a singular curve
projecting to a disc. A Riemann--Hurwitz formula for finite
separable morphisms $f: C\to D$, where $C$ and $D$ are complete irreducible curves and $D$ is nonsingular was proved in \cite[Thm.\ 2]{Ch84}. Chiarli relates the arithmetic genera of $C$ and $D$ with the degree of $f$ and the valuations of the ramification divisor of $f$ in $D$. Another Riemann--Hurwitz-type formula, stated in terms of ramification modules of differential forms, was obtained in \cite{MVS}.
 
Let $\mathcal{X}$ be a reduced curve in a Euclidean neighborhood $\mathcal{V}$ of $0$ in $\mathbb{C}^n$. Fix coordinates $x_1, \ldots, x_n$ on $\mathbb{C}^n$, let $F(x_1, \ldots, x_n)=\sum_{i=1}^n\alpha_ix_i$ be a linear functional with $\alpha_i$ general, and let $\mathbb{D}=\mathbb{D}(0,\delta)\subset \mathbb{C}$ be a closed disk with coordinate $t$. Throughout this section we assume that
$$F|_{\mathcal{X}}: \mathcal{X} \rightarrow \mathring{\mathbb{D}}$$
is a finite morphism, where $\mathcal{X}$ now denotes a suitable open representative, relatively compact in $\mathcal{V}$: for instance $\mathcal{X} \cap \mathring{\mathbb{B}} \cap F^{-1}(\mathring{\mathbb{D}})$, where $\mathring{\mathbb{B}}=\mathring{\mathbb{B}}(0,\epsilon)$ is an open ball with $\mathbb{B} \subset \mathcal{V}$ and $\delta$ is small with respect to $\epsilon$ (this is possible because $F^{-1}(0) \cap \mathcal{X} \cap \partial\mathbb{B}=\emptyset$ for general $F$, see \cite[Prop.\ 1), p.\ 63]{GrR84}). We assume moreover that the $\alpha_i$ are general so that $F|_{\mathcal{X}_{\mathrm{sm}}}$ has only non-degenerate (quadratic) critical points (see Prop.\ \ref{relative Morse} with $Y$ a point), and that for each $x=(z_1, \ldots, z_n) \in \mathcal{X}_{\mathrm{sing}}$ the ideal $(\sum_{i=1}^n\alpha_i(x_i-z_i))$ is a reduction of $\mathfrak{m}_{\mathcal{X},x}$, i.e.\ $F^{-1}(F(x))$ meets the tangent cone of $\mathcal{X}$ at $x$ only at $x$. Because $\mathcal{X}$ is Cohen--Macaulay and $\mathring{\mathbb{D}}$ is smooth, $F|_{\mathcal{X}}$ is flat, so $d:=\#(F^{-1}(t) \cap \mathcal{X})$ does not depend on $t \in \mathring{\mathbb{D}}$ off the finitely many critical values of $F|_{\mathcal{X}_{\mathrm{sm}}}$ and images of singular points, and $d=\sum_{x \in F^{-1}(t) \cap \mathcal{X}}\dim_{\mathbb{C}}\mathcal{O}_{\mathcal{X},x}/(F-t)$ for every $t \in \mathring{\mathbb{D}}$. Denote by $\mathcal{X}^{\mathrm{cl}}$ the closure of $\mathcal{X}$ in $\mathcal{V}$. By properness, $\mathcal{X}^{\mathrm{cl}}$ is compact and $\mathcal{X}^{\mathrm{cl}} \setminus \mathcal{X}=\mathcal{X}^{\mathrm{cl}} \cap F^{-1}(\partial\mathbb{D})=:\partial\mathcal{X}$. We choose $\delta$ so that $\partial\mathcal{X}$ contains neither singular points of the curve nor critical points of $F$ restricted to its smooth part; this excludes only finitely many values of $\delta$. By a result of \cite{Loj}, $\mathcal{X}^{\mathrm{cl}}$, viewed as a compact semi-analytic set, admits a triangulation in which $\partial\mathcal{X}$, the singular points of $\mathcal{X}$ and the critical points of $F|_{\mathcal{X}_{\mathrm{sm}}}$ are subcomplexes; we denote its Euler characteristic by $\chi(\mathcal{X})$. Denote the number of critical points of $F|_{\mathcal{X}_{\mathrm{sm}}}$ by $\#\mathrm{Crit}(F|_{\mathcal{X}_{\mathrm{sm}}})$ and the multiplicity of $\mathcal{X}$ at $x$ by $m_x$.

 
\begin{proposition}\label{prop:Euler_char}
Suppose $\mathcal{X}$ is a reduced curve. Then for $F$ as above, we have
\[\chi(\mathcal{X})=d-\Big(\#\mathrm{Crit}(F|_{\mathcal{X}_{\mathrm{sm}}})
+\sum_{x\in\mathcal{X}_{\mathrm{sing}}}(m_x-1)\Big).\]
\end{proposition}
 
\begin{proof}
Denote by $\nu:\overline{\mathcal{X}}\to\mathcal{X}$ the normalization morphism. Note that $\nu$ is the restriction over $\mathcal{X}$ of the normalization of a curve in $\mathcal{V}$ containing $\mathcal{X}^{\mathrm{cl}}$; we denote by $\overline{\mathcal{X}}^{\mathrm{cl}}$ the preimage of $\mathcal{X}^{\mathrm{cl}}$ under the latter, which is a compact space containing $\overline{\mathcal{X}}$ as a dense open subset, and we keep writing $\nu$ for the induced map $\overline{\mathcal{X}}^{\mathrm{cl}} \rightarrow \mathcal{X}^{\mathrm{cl}}$. Fix a triangulation of $\mathcal{X}^{\mathrm{cl}}$ as in the setup, so that each singular point of $\mathcal{X}$ and each critical point of $F|_{\mathcal{X}_{\mathrm{sm}}}$ is a vertex and $\partial\mathcal{X}$ is a subcomplex. Each open simplex then consists of smooth points of the curve (recall that $\partial\mathcal{X}$ contains no singular points), over which $\nu$ is an isomorphism. Such a triangulation can be lifted to a triangulation of $\overline{\mathcal{X}}^{\mathrm{cl}}$ so that the number of triangles and edges is preserved but each singular point $x$ is replaced by $r_x:=|\nu^{-1}(x)|$ vertices (this procedure separates some of the triangles with a common vertex at a singular point). The result is the addition in the triangulation of $\overline{\mathcal{X}}^{\mathrm{cl}}$ of $r_x-1$ new vertices for each $x \in \mathcal{X}_{\mathrm{sing}}$. Therefore, writing $\chi(\overline{\mathcal{X}})$ for the Euler characteristic of $\overline{\mathcal{X}}^{\mathrm{cl}}$,
\begin{equation}\label{eq:RH_norm}
\chi(\overline{\mathcal{X}})=\chi(\mathcal{X})+\sum_{x\in\mathcal{X}_{\mathrm{sing}}}(r_x-1).
\end{equation}
Consider the decomposition of
$\mathcal{X}=\bigcup_{j=1}^{R}\mathcal{X}_j$ into irreducible components. The normalization
is given by $\overline{\mathcal{X}}=\bigsqcup_{j=1}^{R}\overline{\mathcal{X}_j}$, where
$\overline{\mathcal{X}_j}$ is the normalization of $\mathcal{X}_j$. Each
$\overline{\mathcal{X}_j}$ is connected and nonsingular, and its closure $\overline{\mathcal{X}_j}^{\mathrm{cl}}$ in $\overline{\mathcal{X}}^{\mathrm{cl}}$ is a compact surface with boundary $\overline{\mathcal{X}_j}^{\mathrm{cl}} \setminus \overline{\mathcal{X}_j}=\overline{\mathcal{X}_j}^{\mathrm{cl}} \cap \nu^{-1}(\partial\mathcal{X})$, because the points of $\partial\mathcal{X}$ are smooth points of the curve at which the restriction of $F$ to the curve is a submersion. Set $f:=F|_{\mathcal{X}}$ and
$\nu_j:=\nu|_{\overline{\mathcal{X}_j}}$ for the restrictions, and set $G:=f\circ\nu,\:
G_j:=f\circ\nu_j$. These maps fit into the following diagram:
\[\begin{tikzcd}
\overline{\mathcal{X}}\ar[r,"\nu"] \ar[rd,"G"'] &
    \mathcal{X} \ar[r,"\iota",hook] \ar[d,"f"] & \mathbb{C}^n\ar[ld,"F"]\\
& \mathbb{D}.
\end{tikzcd}\]
For any $p\in\overline{\mathcal{X}_j}$,
let $q=G_j(p)$, and let $s_p,t_q$ be local uniformizing parameters at $p$ and $q$.
Then $G_j^*(t_q)=us_p^{e_p}$ for some unit $u$ in $\Or_{\overline{\mathcal{X}_j},p}$.
The natural number $e_p$ is called the {\it ramification index} of $G_j$ at $p$. Set $d_j:=\#(\mathcal{X}_j \cap f^{-1}(t))$ for generic $t$. Then $d=\#(\mathcal{X} \cap f^{-1}(t))=\sum_{j=1}^{R}d_j$, because generic $t$ avoids the images of the finitely many points of $\mathcal{X}_i \cap \mathcal{X}_j$, $i \neq j$.
Furthermore, because $\nu$ is an isomorphism away from finitely many points we have $d_j=\#G_j^{-1}(t)$ and thus $d=\#G^{-1}(t)$ for generic $t$. Denote by $\chi(\overline{\mathcal{X}_j})$ the Euler characteristic of $\overline{\mathcal{X}_j}^{\mathrm{cl}}$. The map $G_j$ extends continuously to $\overline{\mathcal{X}_j}^{\mathrm{cl}}$, and the extension $G_j:\overline{\mathcal{X}_j}^{\mathrm{cl}} \rightarrow \mathbb{D}$ is a finite branched covering of degree $d_j$ with $G_j^{-1}(\partial\mathbb{D})=\overline{\mathcal{X}_j}^{\mathrm{cl}} \setminus \overline{\mathcal{X}_j}$, which is unramified over $\partial\mathbb{D}$ and has finitely many critical values. The Riemann--Hurwitz formula for such coverings gives
\begin{equation}\label{eq:RH_partial}
\chi(\overline{\mathcal{X}_j})=d_j-\sum_{p\in\overline{\mathcal X_j}}(e_p-1) \ \ \text{for} \ \ j=1,\ldots, R.
\end{equation}
Indeed, triangulate $\mathbb{D}$ so that the critical values of $G_j$ are among the vertices and $\partial\mathbb{D}$ is a subcomplex, and lift the triangulation to $\overline{\mathcal{X}_j}^{\mathrm{cl}}$: every edge and every face of $\mathbb{D}$ has $d_j$ preimages, whereas a vertex $q$ has $d_j-\sum_{p \in G_j^{-1}(q)}(e_p-1)$ preimages. Since $\chi(\mathbb{D})=1$, (\ref{eq:RH_partial}) follows.

 
Combining \eqref{eq:RH_norm} and \eqref{eq:RH_partial}, we obtain
\begin{align}\label{eq:sums}
\chi(\mathcal{X}) &= \chi(\overline{\mathcal{X}})-
    \sum_{x\in\mathcal{X}_{\mathrm{sing}}}(r_x-1) =
    \sum_{j=1}^R\chi(\overline{\mathcal{X}_j})-
    \sum_{x\in\mathcal{X}_{\mathrm{sing}}}(r_x-1) \nonumber\\
&= \sum_{j=1}^R\Big(d_j-\sum_{p\in\overline{\mathcal X_j}}(e_p-1)\Big)-
    \sum_{x\in\mathcal{X}_{\mathrm{sing}}}(r_x-1) \\
&= d-\sum_{p\in\overline{\mathcal{X}}}(e_p-1)-
    \sum_{x\in\mathcal{X}_{\mathrm{sing}}}(r_x-1).\nonumber
\end{align}
Rewrite $\sum_{p\in\overline{\mathcal{X}}}(e_p-1)$ as
$\sum_{x\in\mathcal{X}}\sum_{p\in\nu^{-1}(x)}(e_p-1)$. We separate this double sum according to
whether $x$ is smooth or singular.
 
Suppose $x\in\mathcal{X}_{\mathrm{sm}}$. There exists an open neighborhood
$U\subset\mathcal{X}$ of $x$ such that $\nu$ induces an isomorphism $\nu^{-1}(U)
\xrightarrow{\simeq}U$ and thus the ramification index
of $G_j$ at $p$ is identical to the ramification index of $f$ at $x$.
By the choice of $F$, the restriction of $F$ to $\mathcal{X}_{\mathrm{sm}}$ has only
non-degenerate (or quadratic) critical points. This amounts to saying that for any critical
point $x$ of $F|_{\mathcal{X}_{\mathrm{sm}}}$ with $q=F(x)$, the hyperplane $F^{-1}(q)$ is tangent to $\mathcal{X}$ at $x$
and $\dim_\C\Or_{\mathcal{X},x}/(\mathcal{I}_x)=2$, where $\mathcal{I}_x$ is the image of
the ideal of the hyperplane $F^{-1}(q)$ inside $\Or_{\mathcal{X},x}$. Since $F(x)=q$,
the ideal $\mathcal{I}_x$ is generated by $f^*(t_q)$ and thus $e_p=2$ for $p=\nu^{-1}(x)$. We conclude that
\begin{equation}\label{eq:sum_ep_sm}
\sum_{x\in\mathcal{X}_{\mathrm{sm}}}\sum_{p\in\nu^{-1}(x)}(e_p-1)=
\#\mathrm{Crit}(F|_{\mathcal{X}_{\mathrm{sm}}}).
\end{equation}
 
Suppose $x\in \mathcal{X}_{\mathrm{sing}}$. We claim that for general $F$ we have 
\begin{equation}\label{ram-mult}
\sum_{p\in\nu^{-1}(x)}e_p=m_x.
\end{equation}

Fix $p\in\nu^{-1}(x)$ and write $q=F(x)=G(p)$.
We abbreviate the local maps between stalks to $F^*,G^*,\nu^*$ and $\iota^*$. Let $s_p$
and $t_q$ be respective local uniformizers for the DVRs $\Or_{\overline{\mathcal{X}},p}$
and $\Or_{\mathbb{D},q}$. Write $x=(z_1,\ldots,z_n)$, so that
$x_1-z_1,\ldots,x_n-z_n$ are local regular parameters of $\mathbb{C}^n$ at $x$. Thus
$F^*(t_q)=\sum_{i=1}^n\alpha_i({x}_i-z_i)$.
Define $m_p$ as the multiplicity of the ideal $\nu^*(\m_{\mathcal{X},x})$,
i.e.\ $m_p=\ord_{s_p}\nu^*(\m_{\mathcal{X},x})$.
 
Observe that $\m_{\mathcal{X},x}$ is generated by $\iota^*(x_1-z_1),\ldots,\iota^*(x_n-z_n)$.
Thus
\begin{equation}\label{eq:local_mult}
m_p=\ord_{s_p}\nu^*(\m_{\mathcal{X},x})=
\min_{1\leqslant i\leqslant n}\ord_{s_p}\nu^*\iota^*(x_i-z_i).
\end{equation}
By definition we have $G^*(t_q)=\nu^*f^*(t_q)=\nu^*\Big(\sum_{i=1}^n\alpha_i\iota^*(x_i-z_i)\Big)=
\sum_{i=1}^n\alpha_i\nu^*\iota^*(x_i-z_i)$. Hence, for general $\alpha_i$ we have
\begin{equation}\label{eq:minimum}
e_p=\ord_{s_p}G^*(t_q)=\min_{1\leqslant i\leqslant n}\ord_{s_p}\nu^*\iota^*(x_i-z_i).
\end{equation}
Note that algebraically the genericity condition on the $\alpha_i$ used here is equivalent to asking that the ideal $(\sum_{i=1}^n\alpha_i(x_i-z_i))$ is a reduction of $\mathfrak{m}_{\mathcal{X},x}$ (cf.\ \cite[A13, pg.\ 122]{Mumford}); this is precisely the reduction condition imposed in the setup.
Combining \eqref{eq:local_mult} with \eqref{eq:minimum}, we obtain that for general
$F$, for every $x\in\mathcal{X}_{\mathrm{sing}}$ and every $p\in\nu^{-1}(x),\,e_p=m_p$.
 
We conclude by applying the projection formula for multiplicities \cite[Thm.\ 5]{Kl17}
to $\nu$. The relative degree of $\nu$ is $1$ (since $\nu$ is birational) and the residue
field extensions are trivial since the base field is algebraically closed.
Thus $m_x=\sum_{p\in\nu^{-1}(x)}m_p$.

We pursue the computation in \eqref{eq:sums}
\begin{align*}\chi(\mathcal{X}) &=
d-\sum_{x\in\mathcal{X}}\sum_{p\in\nu^{-1}(x)}(e_p-1)-
    \sum_{x\in\mathcal{X}_\mathrm{sing}}(r_x-1) \\
&= d-\sum_{x\in\mathcal{X}_\mathrm{sm}}\sum_{p\in\nu^{-1}(x)}(e_p-1)-
    \sum_{x\in\mathcal{X}_\mathrm{sing}}\Big((r_x-1)+\sum_{p\in\nu^{-1}(x)}(e_p-1)\Big).
\end{align*}
The middle term is equal to $\#\mathrm{Crit}(F|_{\mathcal{X}_{\mathrm{sm}}})$ by
\eqref{eq:sum_ep_sm}. For the last term (\ref{ram-mult}) gives
\[(r_x-1)+\sum_{p\in\nu^{-1}(x)}(e_p-1)=(r_x-1)+(m_x-r_x)=m_x-1,\]
which finishes the proof.
\end{proof}
We derive an inequality between the number of local and global irreducible components of $\mathcal{X}$. As an application we characterize the connected curves $\mathcal{X}$ with a vanishing first Betti number. Also, we derive an inequality between the critical points of a linear function and the number of its zeroes, which in the case when $\mathcal{X}$ is smooth and contractible (so that $\mathcal{X}$ is a disk and the inequality reads $\#\mathrm{Crit}(F|_{\mathcal{X}}) \geq \#(F^{-1}(0)\cap \mathcal{X})-1$) resembles Rolle's theorem from real analysis. 
 
Denote by $R$ the number of irreducible components of $\mathcal{X}$. When $R \geq 2$ denote by $\{p_1, \ldots, p_k\}$ the points of intersection of the irreducible components of $\mathcal{X}$. Define the {\it frame} $(\mathcal{X}^{\mathrm{fra}},\{p_1, \ldots, p_k\})$ of $\mathcal{X}$ to be the $k$-pointed curve formed by removing those irreducible components of $\mathcal{X}$ that pass through exactly one point of intersection. To $\mathcal{X}^{\mathrm{fra}}$ one can naturally associate a hypergraph $G(\mathcal{X}^{\mathrm{fra}})$ whose vertices are $\{p_1, \ldots, p_k\}$ and whose edges represent irreducible components of $\mathcal{X}^{\mathrm{fra}}$. A {\it cycle} in $G(\mathcal{X}^{\mathrm{fra}})$ is a sequence $p_{i_1}, e_1, p_{i_2}, e_2, \ldots, p_{i_l}, e_l, p_{i_1}$ with $l \geq 2$, where $p_{i_1}, \ldots, p_{i_l}$ are pairwise distinct vertices and $e_1, \ldots, e_l$ are pairwise distinct edges such that $p_{i_j}, p_{i_{j+1}} \in e_j$ for $j=1, \ldots, l$ (indices modulo $l$); its {\it length} is $l$. We say that $G(\mathcal{X}^{\mathrm{fra}})$ is a {\it tree} if it is connected and has no cycles. Note that $G(\mathcal{X}^{\mathrm{fra}})$ is empty (it has no vertices, or no edges) if $R=1$, or if all irreducible components of $\mathcal{X}$ pass through a single point; in the latter case $\mathcal{X}^{\mathrm{fra}}=\emptyset$.

 
\begin{proposition}\label{nonneg}
Suppose $\mathcal{X}$ is a reduced and connected curve. 
\begin{enumerate}
    \item[\rm{(i)}] We have 
    \begin{equation}\label{inq-comps}
    1+\sum_{x \in \mathcal{X}_{\mathrm{sing}}}(r_x-1)-R\geq 0
    \end{equation}
    with an equality if and only if $\nu_{j}:\overline{\mathcal{X}_j}\rightarrow \mathcal{X}_j$ is a homeomorphism ($\mathcal{X}_j$ is unibranch at each of its points) for $j=1, \ldots, R$ and either $G(\mathcal{X}^{\mathrm{fra}})$ is empty, or it is a tree. In particular, $1-\chi(\mathcal{X}) \geq  0$ with an equality if and only if $\mathcal{X}$ viewed as a CW-complex is contractible.
    \item[\rm{(ii)}] If $\mathcal{X}$ passes through 
$0 \in \mathbb{C}^n$, then we have
$$\#\mathrm{Crit}(F|_{\mathcal{X}_{\mathrm{sm}}})+\sum_{x \in \mathcal{X}, x \neq 0} (m_x-1) \geq \#(F^{-1}(0)\cap \mathcal{X} \setminus \{0\}).$$
\end{enumerate}
\end{proposition}
\begin{proof} Consider $\rm{(i)}$. Denote by $R_{p_i}$ for $i=1, \ldots, k$ the number of irreducible components of $\mathcal{X}$ passing through $p_i$. Trivially, $r_{p_{i}} \geq R_{p_i}$ with an equality if and only if each irreducible component of $\mathcal{X}$ passing through $p_i$ is analytically irreducible. Therefore,
$$1+\sum_{x \in \mathcal{X}_{\mathrm{sing}}}(r_x-1) \geq 1+\sum_{i=1}^{k} (r_{p_i}-1) \geq 1+\sum_{i=1}^k (R_{p_i}-1) \geq R$$
where the last inequality follows by induction on $R$, noting that we can always remove an irreducible component of $\mathcal{X}$ so that the resulting curve is connected (the graph whose vertices are the irreducible components of $\mathcal{X}$, two of them being adjacent when they meet, is connected; remove a component that is a leaf of a spanning tree of this graph). 
 
Assume we have equalities in the above sequence of inequalities.  Then $r_x-1=0$ for each $x \in \mathcal{X}_{\mathrm{sing}}$ and $x \not \in \{p_1, \ldots, p_k\}$ and $r_{p_i}=R_{p_i}$ for each $i=1, \ldots, k$. Thus each $\mathcal{X}_j$ is unibranch at each of its points and so $\nu_j: \overline{\mathcal{X}_j} \rightarrow \mathcal{X}_j$ is a homeomorphism. Note that removing an irreducible component $\mathcal{X}_j$ that passes through only one intersection point $p_i$ decreases the total number of irreducible components and $r_{p_i}$ by $1$. Thus, if $\mathcal{X}^{\mathrm{fra}} \neq \emptyset$,
\begin{equation}\label{fra.eq.}
1+\sum_{x \in \mathcal{X}_{\mathrm{sing}}^{\mathrm{fra}}}(r_{\mathcal{X}^{\mathrm{fra}},x}-1)
=R_{\mathcal{X}^{\mathrm{fra}}}
\end{equation}
where $r_{\mathcal{X}^{\mathrm{fra}},x}$ is the number of irreducible components of $\mathcal{X}^{\mathrm{fra}}$ passing through $x$ and $R_{\mathcal{X}^{\mathrm{fra}}}$ is the number of irreducible components of $\mathcal{X}^{\mathrm{fra}}$. If $R=1$ or $\mathcal{X}^{\mathrm{fra}}=\emptyset$, then clearly $G(\mathcal{X}^{\mathrm{fra}})$ is empty. Assume that $\mathcal{X}^{\mathrm{fra}} \neq \emptyset$, $R \geq 2$ and that there exists a cycle in $G(\mathcal{X}^{\mathrm{fra}})$; choose one of minimal length $l$. Denote by $\mathcal{C}$ the curve in $\mathcal{X}^{\mathrm{fra}}$ formed by the irreducible components $e_1, \ldots, e_l$ of the cycle. Denote by $R_{\mathcal{C}}=l$ the number of irreducible components of $\mathcal{C}$ and denote by $r_{\mathcal{C},p_i}$ the number of irreducible components of $\mathcal{C}$ passing through $p_i$. By minimality, no point lies on three of the components of $\mathcal{C}$ (otherwise there would be a shorter cycle), so $r_{\mathcal{C},p_i} \leq 2$ for every $p_i \in \mathcal{C}$, and $r_{\mathcal{C},p_i}=2$ for the $l$ vertices $p_{i_1}, \ldots, p_{i_l}$ of the cycle. Thus
\begin{equation}\label{cycle id}
1+\sum_{p_i \in \mathcal{C}}(r_{\mathcal{C},p_i}-1) \geq 1+R_\mathcal{C}.
\end{equation}
Let $\mathcal{X}^{\mathrm{fra}}_1,\ldots, \mathcal{X}^{\mathrm{fra}}_m$ ($m \geq 0$) be the connected components of the union of the irreducible components of $\mathcal{X}^{\mathrm{fra}}$ not contained in $\mathcal{C}$. Because $\mathcal{X}^{\mathrm{fra}}$ is connected, $\mathcal{X}^{\mathrm{fra}}_u \cap \mathcal{C} \neq \emptyset$ for each $u=1, \ldots, m$. Denote by $R_{\mathcal{X}^{\mathrm{fra}}_u}$ the number of irreducible components of $\mathcal{X}^{\mathrm{fra}}_u$ and by $r_{\mathcal{X}^{\mathrm{fra}}_u,p_i}$ the number of irreducible components of $\mathcal{X}^{\mathrm{fra}}_u$ passing through $p_i$; by convention $r_{\mathcal{C},p_i}-1=0$ if $p_i \not\in \mathcal{C}$, and $r_{\mathcal{X}^{\mathrm{fra}}_u,p_i}-1=0$ if $p_i \not\in \mathcal{X}^{\mathrm{fra}}_u$. For $p_i \in \mathcal{X}^{\mathrm{fra}}$ denote by $g_{p_i}$ the number of curves among $\mathcal{C}, \mathcal{X}^{\mathrm{fra}}_1, \ldots, \mathcal{X}^{\mathrm{fra}}_m$ passing through $p_i$. Since the irreducible components of $\mathcal{X}^{\mathrm{fra}}$ are partitioned among these curves, we have
\begin{equation}\label{local ineq}
r_{\mathcal{X}^{\mathrm{fra}},p_i}-1 = (r_{\mathcal{C},p_i}-1)+\sum_{u=1}^{m}(r_{\mathcal{X}^{\mathrm{fra}}_u,p_i}-1)+(g_{p_i}-1),
\end{equation}
and $\sum_{p_i \in \mathcal{X}^{\mathrm{fra}}}(g_{p_i}-1) \geq m$ because each $\mathcal{X}^{\mathrm{fra}}_u$ meets $\mathcal{C}$. Applying the inequality $1+\sum_{i}(R_{p_i}-1) \geq R$ proved above to each connected curve $\mathcal{X}^{\mathrm{fra}}_u$ in place of $\mathcal{X}$, we get $\sum_{p_i \in \mathcal{X}^{\mathrm{fra}}}(r_{\mathcal{X}^{\mathrm{fra}}_u,p_i}-1) \geq R_{\mathcal{X}^{\mathrm{fra}}_u}-1$. Summing (\ref{local ineq}) over $p_i \in \mathcal{X}^{\mathrm{fra}}$ and using (\ref{cycle id}), we obtain
$$1+\sum_{p_i \in \mathcal{X}^{\mathrm{fra}}}(r_{\mathcal{X}^{\mathrm{fra}},p_i}-1)
\geq 1 + R_{\mathcal{C}}+ \sum_{u=1}^{m} (R_{\mathcal{X}^{\mathrm{fra}}_u}-1)+m=1+R_{\mathcal{X}^{\mathrm{fra}}}>R_{\mathcal{X}^{\mathrm{fra}}}.$$
Since the singular points of $\mathcal{X}^{\mathrm{fra}}$ that are not among the $p_i$ lie on a single irreducible component of $\mathcal{X}^{\mathrm{fra}}$, the left-hand side equals the left-hand side of (\ref{fra.eq.}), and we get a contradiction. Thus $G(\mathcal{X}^{\mathrm{fra}})$ is a tree.
 
Conversely, assume that each $\nu_j$ is a homeomorphism and that $G(\mathcal{X}^{\mathrm{fra}})$ is empty or a tree. Then $r_x$ is the number of irreducible components of $\mathcal{X}$ passing through $x$ for every $x$, so $r_x=1$ for $x \in \mathcal{X}_{\mathrm{sing}} \setminus \{p_1, \ldots, p_k\}$. If $G(\mathcal{X}^{\mathrm{fra}})$ is empty, then either $R=1$, or all irreducible components of $\mathcal{X}$ pass through a single point $p_1$ and $r_{p_1}=R$; in both cases (\ref{inq-comps}) is an equality. If $G(\mathcal{X}^{\mathrm{fra}})$ is a tree with vertex set $P=\{p_1, \ldots, p_k\}$ and edge set $E$, consider its incidence graph, i.e.\ the bipartite graph with vertex set $P \sqcup E$ in which $p \in P$ and $e \in E$ are adjacent when $p \in e$. It is connected and has no cycles, so it is a tree with $|P|+|E|-1$ edges. On the other hand its number of edges is $\sum_{p \in P}r_{\mathcal{X}^{\mathrm{fra}},p}$, whence $1+\sum_{p \in P}(r_{\mathcal{X}^{\mathrm{fra}},p}-1)=|E|=R_{\mathcal{X}^{\mathrm{fra}}}$. Adding back the irreducible components of $\mathcal{X}$ that pass through exactly one point of intersection increases both $\sum_{x \in \mathcal{X}_{\mathrm{sing}}}(r_x-1)$ and $R$ by one for each such component, so (\ref{inq-comps}) is an equality. To prove the second statement in $\rm{(i)}$, observe that (\ref{eq:RH_norm}) gives us
 
\begin{equation}
1-\chi(\mathcal{X})=1+\sum_{x\in\mathcal{X}_{\mathrm{sing}}}(r_x-1)-\sum_{j=1}^{R}\chi(\overline{\mathcal{X}_j}). 
\end{equation}
 
Denote by $b(\overline{\mathcal{X}_j})$ the number of boundary components of the compact surface $\overline{\mathcal{X}_j}^{\mathrm{cl}}$ and by $g(\overline{\mathcal{X}_j})$ its genus. Then $\chi(\overline{\mathcal{X}_j})=2-2g(\overline{\mathcal{X}_j})-b(\overline{\mathcal{X}_j}) \leq 1$ because $g(\overline{\mathcal{X}_j}) \geq 0$ and $b(\overline{\mathcal{X}_j}) \geq 1$ (the boundary of $\overline{\mathcal{X}_j}^{\mathrm{cl}}$ is $G_j^{-1}(\partial\mathbb{D})$, which is nonempty because $G_j$ is surjective). Thus
$$1-\chi(\mathcal{X}) \geq 1+\sum_{x \in \mathcal{X}_{\mathrm{sing}}}(r_x-1)-R\geq 0$$
with equalities if and only if $g(\overline{\mathcal{X}_j})=0$, $b(\overline{\mathcal{X}_j})=1$ ($\overline{\mathcal{X}_j}$ is a disk),  $\nu_{j}:\overline{\mathcal{X}_j}\rightarrow \mathcal{X}_j$ is a homeomorphism, and either $G(\mathcal{X}^{\mathrm{fra}})$ is empty or it is a tree. Thus, $1-\chi(\mathcal{X})=0$ implies that up to a homeomorphism, $\mathcal{X}$ is a union of disks that intersect at points, forming a tree-like configuration. Contract each disk that passes through exactly one point of intersection $x_j$ to $x_j$; since such a disk is attached to the union of the remaining disks at the single point $x_j$, this is a deformation retraction. Because $G(\mathcal{X}^{\mathrm{fra}})$ is a tree, the union of the remaining disks again contains a disk passing through at most one point of intersection of the remaining disks (a tree with at least two edges has an edge all of whose vertices but one have degree one), and we can continue the process until there is one disk left. Thus $\mathcal{X}$ is contractible. 
 
Consider $\rm{(ii)}$. By Prop.\ \ref{prop:Euler_char} and $\rm{(i)}$ we have
$$\#\mathrm{Crit}(F|_{\mathcal{X}_{\mathrm{sm}}})
+\sum_{x\in \mathcal{X}, x \neq 0}(m_x-1)-(d-m_0)=1-\chi(\mathcal{X}) \geq 0.$$
The inequality follows from $d-m_0 \geq \#(F^{-1}(0)\cap \mathcal{X} \setminus \{0\})$, which is proved as follows. As recalled in the setup, $d=\sum_{x \in F^{-1}(0) \cap \mathcal{X}}\dim_{\mathbb{C}}\mathcal{O}_{\mathcal{X},x}/(F)$ because $F|_{\mathcal{X}}$ is finite and $\mathcal{X}$ is Cohen--Macaulay. The summand at $x=0$ equals $m_0$ because $(F)$ is a reduction of $\mathfrak{m}_{\mathcal{X},0}$ (this is the reduction condition of the setup if $0 \in \mathcal{X}_{\mathrm{sing}}$, and it means that $F^{-1}(0)$ is transverse to $\mathcal{X}$ at $0$ if $0 \in \mathcal{X}_{\mathrm{sm}}$, which holds for general $F$), and each of the remaining summands is at least $1$. An illustration is provided in Figure \ref{otherfig1}. 
\begin{figure}[ht]
\includegraphics[scale=0.12]{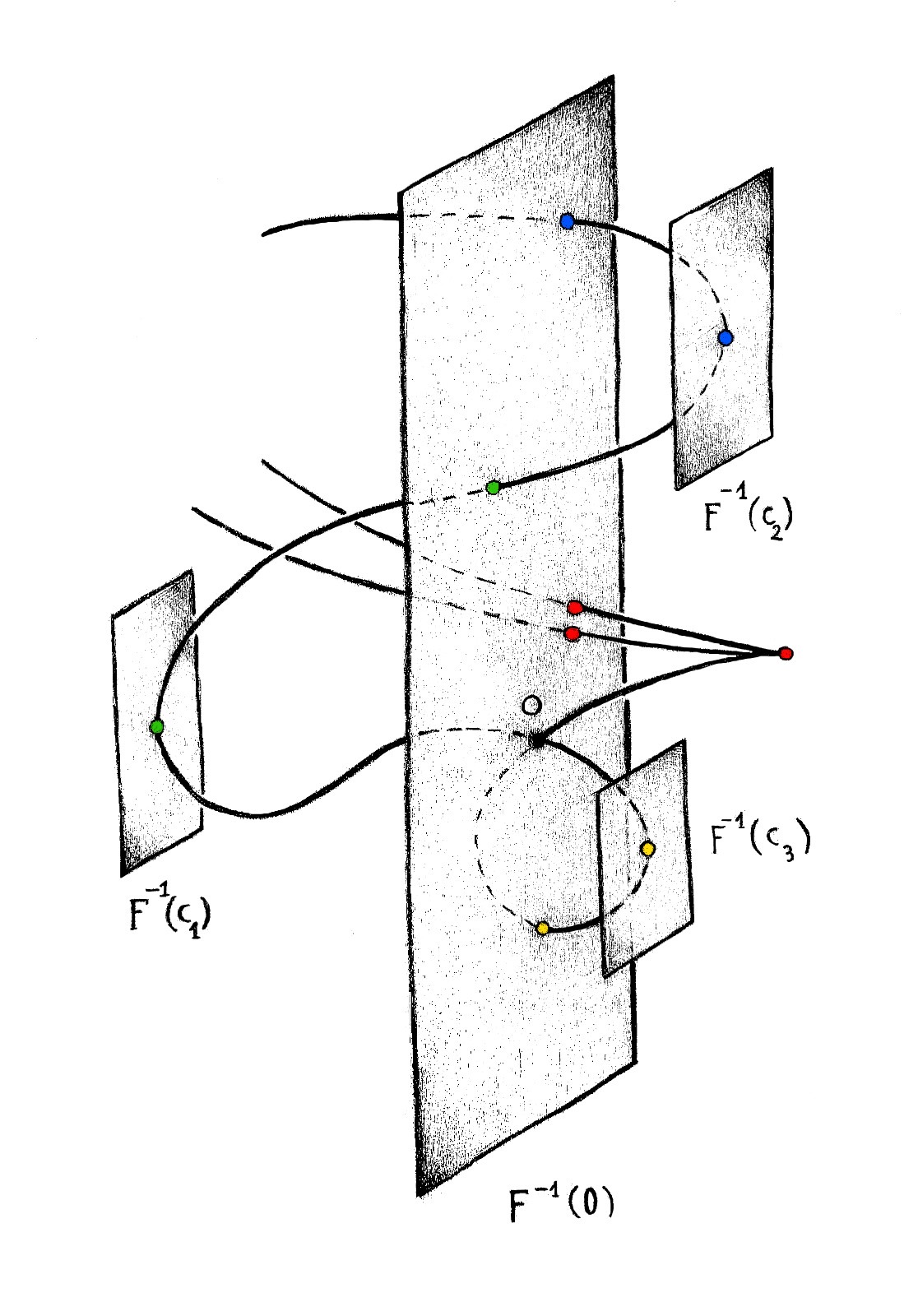}
\caption{An illustration of Prop.~\ref{nonneg} \rm{(ii)}.}\label{otherfig1}
\end{figure}
\end{proof}
The implication $1=\chi(\mathcal{X}) \Rightarrow \mathcal{X}$ is contractible was already proved by Buchweitz and Greuel in \cite[Thm.\ 4.2.4]{BG80}. In their setting $\mathcal{X}$ has vanishing hypercohomology $H^1(\mathcal{X},\mathbb{C})$. They show that the equality $H^1(\mathcal{X},\mathbb{C})=H^1(\overline{\mathcal{X}},\mathbb{C})=0$ and classification of open Riemann surfaces yields that $\overline{\mathcal{X}}$ is a union of disks. Then by a Mayer--Vietoris argument they conclude that $\mathcal{X}$ up to a homeomorphism  has a tree-like configuration of disks which implies contractibility. Our contribution in Prop.\ \ref{nonneg} \rm{(i)} is that the tree-like configuration of $\mathcal{X}$ is a consequence of the equality in (\ref{inq-comps}),  which is weaker than $1=\chi(\mathcal{X})$.

\section{Zariski upper semicontinuity}\label{ZU section}
In this section we establish the upper semicontinuity of several invariants assuming only that $(Y,0)$ is a reduced complex analytic germ of positive dimension.
 
\begin{definition}\label{def:ZU}
We say that a function $\alpha:(Y,0) \rightarrow \mathbb{Z}_{\geq 0}$ is {\it Zariski upper semicontinuous} if for each $n \in \mathbb{Z}$ the set $\{y \in Y \ |  \ \alpha(y) \geq n \}$ is a complex analytic subset of $(Y,0)$.
\end{definition}
Note that if $\alpha$ is Zariski upper semicontinuous, then $\alpha(0) \geq \alpha(y)$ for all $y$ close enough to $0$, because $\{\alpha \geq \alpha(0)+1\}$ is a closed analytic subset of $Y$ which does not contain $0$. We shall establish Zariski upper semicontinuity through the following characterization.
\begin{lm}\label{ZU-equiv}
A function $\alpha:(Y,0) \rightarrow \mathbb{Z}_{\geq 0}$ is Zariski upper semicontinuous if and only if for each irreducible complex analytic subgerm $(A,0)$ of $(Y,0)$ there exist an integer $c_A$ and a nonempty Zariski open subset $U_A$ of $A$ such that $\alpha_{|U_A} \equiv c_A$ and $\alpha(y) \geq c_A$ for each $y \in A$.
\end{lm}
\begin{proof}
Suppose $\alpha$ is Zariski upper semicontinuous and let $(A,0)$ be an irreducible subset of $(Y,0)$. The sets $A \cap \{\alpha \geq n\}$ are analytic subsets of $A$ which decrease as $n$ increases. Let $c_A$ be the largest integer $n$ with $A \subset \{\alpha \geq n\}$. Then $\alpha \geq c_A$ on $A$, and $A \cap \{\alpha \geq c_A+1\}$ is a proper analytic subset of $A$; its complement $U_A$ in $A$ is a nonempty Zariski open subset of $A$ on which $\alpha \equiv c_A$.
 
Assume the converse. Fix $n$. We show by induction on $\dim A$ that $A \cap \{\alpha \geq n\}$ is an analytic subset of $A$ for each irreducible analytic subgerm $(A,0)$ of $(Y,0)$. Applied to the irreducible components of $(Y,0)$ this proves the claim. If $n \leq c_A$, then $A \cap \{\alpha \geq n\}=A$. If $n > c_A$, then $A \cap \{\alpha \geq n\}$ is contained in the proper analytic subset $A \setminus U_A$ of $A$, and so it is the union of the sets $A' \cap \{\alpha \geq n\}$, where $A'$ runs through the finitely many irreducible components of $A \setminus U_A$. Each $A'$ has dimension smaller than $\dim A$, so by induction each $A' \cap \{\alpha \geq n\}$ is analytic.
\end{proof}

Throughout this section we adopt the family setup of Thm.\ \ref{main mpt}; in Prop.\ \ref{Zariski-ups-cid}, Cor.\ \ref{Ballesteros} and Thm.\ \ref{ZU of Milnor} \rm{(ii)} we allow more generally $(Y,0)$ to be a reduced complex analytic germ. We make extensive use of fibre products; for their construction in the complex analytic setting, as well as their basic properties, we refer to \cite[Chps.\ 0 and 3]{Fisch}. For the basic theory of Hilbert-Samuel functions and multiplicities in the complex analytic setting, see \cite{HIO88}.
 
\begin{proposition}\label{Zariski-ups-cid}
Suppose $h:(X,0) \rightarrow (Y,0)$ is a flat family of reduced curves as in the family setup, except that $(Y,0)$ is only assumed to be a reduced complex analytic germ. Then the functions $y \mapsto e(\mathrm{Jac}(X_y))-\mathrm{cid}(X_y)$ and $y \mapsto \mathrm{cid}(X_y)$ are Zariski upper semicontinuous.
\end{proposition}
\begin{proof}
Set $\alpha(y):=e(\mathrm{Jac}(X_y))-\mathrm{cid}(X_y)$ and $\beta(y):=\mathrm{cid}(X_y)$. By shrinking the representatives of $(X,0)$ and $(Y,0)$ we can ensure that the relative singular locus $S_h$ of $h$ is finite over $Y$ with $(S_h)_0=\{0\}$. In particular, $\alpha$ and $\beta$ do not depend on the choice of $\mathring{B_0}$ for $y$ close enough to $0$. By Lem.\ \ref{ZU-equiv} it suffices to show that for each irreducible complex analytic subgerm $(A,0)$ of $(Y,0)$ there exist a nonempty Zariski open subset $U_A$ of $A$ and integers $c_A, c_A'$ such that $\alpha \equiv c_A$ and $\beta \equiv c_A'$ on $U_A$, and $\alpha(y) \geq c_A$ and $\beta(y) \geq c_A'$ for all $y \in A$. Replacing $Y$ by $A$, equipped with its reduced structure, and $X$ by $X \times_{Y} A$, which is a flat family of reduced curves over $A$ with reduced total space and with the same fibers, we may assume that $Y=A$ is irreducible. Denote by $Y_{\mathrm{sm}}$ the smooth locus of $Y$. It is a nonempty Zariski open subset of $Y$, and it is connected because $Y$ is irreducible. We shall use repeatedly that for $y_0 \in Y_{\mathrm{sm}}$ and $s \in (X_{y_0})_{\mathrm{sing}}$ the germ $(X,s) \rightarrow (Y,y_0)$ is a flat embedded deformation of the reduced curve $(X_{y_0},s)$ over a smooth base, to which Thm.\ \ref{main mpt} applies.
 
{\bf Generic constancy of $\alpha$.} Let $T \subset S_h$ be the set of points $s$ at which the fiber $c_h^{-1}(s)$ of the relative conormal space $C_h(X)$ has dimension $n-1$. By upper semicontinuity of fiber dimension, $T$ is an analytic subset of $S_h$. Because $C_h(X)$ is the closure of its restriction to $X \setminus S_h$, which is a $\mathbb{P}^{n-2}$-bundle over $X \setminus S_h$, we have $\dim c_h^{-1}(S_h)<\dim C_h(X)=\dim Y+n-1$. Since $\dim c_h^{-1}(T) \geq \dim T+n-1$, we get $\dim T \leq \dim Y-1$. As $S_h \rightarrow Y$ is finite, $U_\Gamma:=Y \setminus h(T)$ is a nonempty Zariski open subset of $Y$. Set $U_\alpha:=U_\Gamma \cap Y_{\mathrm{sm}}$; it is a nonempty Zariski open subset of $Y$ which is connected. Let $y_0 \in U_\alpha$ and let $s_1, \ldots, s_l$ be the singular points of $X_{y_0}$. For each $j$ we have $\dim c_h^{-1}(s_j)=n-2$, so for a general $H \in \check{\mathbb{P}}^{n-1}$ the set $\lambda^{-1}(H)$ does not meet $c_h^{-1}(s_j)$; that is, the polar variety of dimension $\dim Y$ of the germ $(X,s_j) \rightarrow (Y,y_0)$ is empty. Applying Thm.\ \ref{main mpt} to this germ for $j=1, \ldots, l$, and using that every singular point of $X_y$ for $y$ close to $y_0$ lies on a component of $S_h$ that passes through one of the $s_j$ (because $S_h \rightarrow Y$ is finite), we obtain by summing over $j$ that $\alpha(y_0)=\alpha(y)$ for $y$ in a Zariski open dense subset of a neighborhood of $y_0$ in $U_\alpha$. Thus $\alpha$ is locally generically constant on $U_\alpha$: if two such neighborhoods meet, the corresponding dense open subsets meet, and the two values agree. Because $U_\alpha$ is connected, $\alpha$ is constant on $U_\alpha$. Denote its value by $c_\alpha$.

{\bf Generic constancy of $\beta$.} Denote by $B$ the blowup of $X$ with center the relative Jacobian ideal $\mathrm{Jac}_h(X)$, by $D$ its exceptional divisor, and by $B_y$ and $D_y$ the fibers of $B$ and $D$ over $y$. Set $B(y):=\mathrm{Bl}_{\mathrm{Jac}(X_y)}X_y$ and denote by $D(y)$ its exceptional divisor; since $\mathrm{Jac}_h(X)\mathcal{O}_{X_y}=\mathrm{Jac}(X_y)$, $B(y)$ is a closed subspace of $B_y$ with $D \cap B(y)=D(y)$. Because $S_h \rightarrow Y$ is finite, the map $D \rightarrow Y$ is proper. By Frisch's theorem \cite{Frisch} there is a nonempty Zariski open subset $U'$ of $Y$ over which $B$ and $D$ are flat. The space $B$ is reduced and of pure dimension $\dim Y+1$, as $X$ is, so $D$ is of pure dimension $\dim Y$; hence, by flatness, for $y \in U'$ every irreducible component of $B_y$ is one-dimensional and $D_y$ is finite. The argument in the proof of Prop.\ \ref{vertical} \rm{(ii)} now gives $[B_y]=[B(y)]$ for $y \in U'$. We follow the analysis in the proof of the Excess-Degree Formula \cite[Thm.\ 2.1]{Rangachev}. Let $y \in U' \cap Y_{\mathrm{sm}}$, which is a nonempty Zariski open subset of $Y$ hence connected, and denote by $[B]_y$ and $[D]_y$ the specializations of the cycles $[B]$ and $[D]$ to the fiber over $y$ (see \cite[Sct.\ 10.1]{Ful}). Because $B$ is flat over $U'$, $[B]_y=[B_y]$ \cite[Thm.\ 6.2 (b), Lem.\ 1.7.1]{Ful}, and because $B$ is reduced, $[D]=c_1(\mathcal{O}_B(D)) \cap [B]$ (see \cite[Sct.\ 2.5]{Ful}). Hence, by \cite[Prop.\ 6.3]{Ful}, $$[D]_y=c_1(\mathcal{O}_B(D)) \cap [B_y]=D \cdot [B(y)]=[D(y)],$$ and $\int [D(y)]=e(\mathrm{Jac}(X_y))$ by \cite[Ex.\ 4.3.4]{Ful} and \cite{Ramanujam}. By conservation of number \cite[Prop.\ 10.2]{Ful} (cf.\ \cite[App.\ A]{Massey} in the complex analytic setting), applied to the cycle $[D]$ and the proper map $D \rightarrow Y$ over the connected manifold $U' \cap Y_{\mathrm{sm}}$, $y \mapsto \int[D]_y=e(\mathrm{Jac}(X_y))$ is constant on $U' \cap Y_{\mathrm{sm}}$. Set $U_\beta:=U_\alpha \cap U'$ which is a nonempty Zariski open subset of $Y$. Then $\beta=e(\mathrm{Jac}(X_y))-\alpha$ is constant on $U_\beta$. Denote its value by $c_\beta$. Note that $\alpha \equiv c_\alpha$ on $U_\beta$ as well.
 
{\bf Lower bounds.} Fix $y \in Y$ and let $x_1, \ldots, x_w$ be the singular points of $X_y$. Since $Y$ is irreducible at $0$, each irreducible component of the germ $(Y,y)$ has dimension $\dim Y>\dim (Y \setminus U_\beta)$. 
So by prime avoidance (or the curve selection lemma) there is an analytically irreducible curve germ $(Y',y) \subset (Y,y)$ with $Y' \setminus \{y\} \subset U_\beta$. Let $n':\overline{Y'} \rightarrow Y'$ be the normalization. It is a homeomorphism, $\overline{Y'}$ is a smooth curve germ at $y':=(n')^{-1}(y)$, and $X':=X \times_{Y} \overline{Y'} \rightarrow \overline{Y'}$ is a flat family of reduced curves with $X'_{z}=X_{n'(z)}$ for $z \in \overline{Y'}$; in particular $X'_{y'}=X_y$. Assume each $(X_y,x_i)$ is contained in a small open ball $\mathring{B_i} \subset \mathring{B_0}$ centered at $x_i$ with $B_i \cap B_j = \emptyset$ for $i \neq j$, and let $h_i':X_i' \rightarrow \overline{Y'}$ with $X_i' \subset \mathring{B_i} \times (\overline{Y'},y')$ be the induced embedded deformation of $(X_y,x_i)$; it is a flat embedded deformation of a reduced curve over a smooth curve germ. Because $S_h \rightarrow Y$ is finite, after shrinking $\overline{Y'}$ every singular point of $X'_z$ lies in $\bigcup_{i=1}^{w}\mathring{B_i} \times \{z\}$.
 
First we treat $\alpha$. By Thm.\ \ref{main mpt} applied to $h_i'$ we have
$$e(\mathrm{Jac}(X_y,x_i))-\mathrm{cid}(X_y,x_i)-\bigl( e(\mathrm{Jac}((X_i')_z))-\mathrm{cid}((X_i')_z)\bigr)=\mathrm{deg}_{\overline{Y'}}\Gamma_{h_i'}^{1}(X_i') \geq 0$$
for $z \neq y'$ close enough to $y'$. Summing over $i$ we obtain $\alpha(y) \geq \alpha(X'_z)=\alpha(n'(z))=c_\alpha$, because $n'(z) \in U_\beta$.
 
Next we treat $\beta$. As in Sct.\ \ref{proof mpt}, let $Z' \subset (\mathbb{C}^n \times \overline{Y'},(0,y'))$ be the complete intersection defined by $A(f_1, \ldots, f_p)^{T}$, where $f_1, \ldots, f_p$ define $X'$ and $A \in \mathrm{Mat}((n-1)\times p,\mathbb{C})$ is general, so that, by Prop.\ \ref{reducedness} and Cor.\ \ref{jacobian ideal}, $Z'_{y'}$ is reduced along $X_y$ and $e(J(Z'_{y'},x_i),\mathcal{O}_{X_y,x_i}^{n-1})=e(\mathrm{Jac}(X_y,x_i))$ for $i=1, \ldots, w$. Set $W':=\overline{Z' \setminus X'}$ and $S':=X' \times_{Z'} W'$. By Thm.\ \ref{constancy} we have $W'_{y'}=\overline{Z'_{y'} \setminus X_y}$, and so by \cite[Thm.\ 3.9]{BR24} $I_{x_i}(X_y,W'_{y'})=\mathrm{cid}(X_y,x_i)$ for each $i$. Hence $\mathrm{cid}(X_y)=\sum_{i=1}^{w}I_{x_i}(X_y,W'_{y'})$. On the other hand, for $z$ close to $y'$ the curve $Z'_z$ is reduced along $X'_z$ (see Sct.\ \ref{proof mpt}), and we claim that for every $x \in (X'_z)_{\mathrm{sing}}$
$$I_{x}(X'_z,W'_z) \geq \mathrm{cid}(X'_z,x),$$
where we set $I_{x}(X'_z,W'_z):=0$ if $x \notin W'_z$. Indeed, if $x \notin W'_z$, then $(X'_z,x)=(Z'_z,x)$ is a complete intersection and $\mathrm{cid}(X'_z,x)=0$. If $x \in W'_z$, then $Z'_z$ is generically reduced along $X'_z$, so \cite[Thm.\ 3.9]{BR24} applied to $Z'_z$ and to a general complete intersection containing $(X'_z,x)$ gives
$$I_{x}(X'_z,W'_z)-\mathrm{cid}(X'_z,x)=e(\mathrm{Jac}(Z'_z,x))-e(\mathrm{Jac}(X'_z,x)).$$
The right-hand side is nonnegative: by the Cauchy--Binet formula the maximal minors of the Jacobian matrix of $Z'_z$ are linear combinations of the maximal minors of the Jacobian matrix of $X'_z$, so $\mathrm{Jac}(Z'_z,x)\mathcal{O}_{X'_z,x} \subset \mathrm{Jac}(X'_z,x)$, and the Hilbert--Samuel multiplicity does not increase under inclusion of $\mathfrak{m}$-primary ideals. In other words, $\mathrm{cid}(X'_z,x)$ is the minimum of $I_{x}(X'_z,W'_z)$ over all complete intersections $Z'_z$ containing $X'_z$ that are generically reduced along $X'_z$.
Now apply Thm.\ \ref{constancy} to the germ $(X',x_i) \rightarrow (\overline{Y'},y')$ and the germ of $Z'$ at $x_i$, for $i=1, \ldots, w$. Since $S'$ is finite over $\overline{Y'}$, for $z$ close enough to $y'$ every point of $S'_z$ lies in $\bigcup_{i=1}^{w}\mathring{B_i} \times \{z\}$, and we get
\begin{align*}
\mathrm{cid}(X_y)&=\sum_{i=1}^{w}I_{x_i}(X_y,W'_{y'})=\sum_{x \in S'_z}I_{x}(X'_z,W'_z) \geq \sum_{x \in (X'_z)_{\mathrm{sing}}}I_{x}(X'_z,W'_z)\\
&\geq \sum_{x \in (X'_z)_{\mathrm{sing}}}\mathrm{cid}(X'_z,x)=\mathrm{cid}(X'_z)=c_\beta
\end{align*}
for $z \neq y'$ close enough to $y'$, because $n'(z) \in U_\beta$.
 
We have shown that $\alpha \equiv c_\alpha$ and $\beta \equiv c_\beta$ on the nonempty Zariski open subset $U_\beta$ of $Y$, and that $\alpha \geq c_\alpha$ and $\beta \geq c_\beta$ on $Y$. By Lem.\ \ref{ZU-equiv}, applied to each irreducible subgerm $(A,0)$ of $(Y,0)$, the functions $\alpha$ and $\beta$ are Zariski upper semicontinuous.
\end{proof}
We recall the definition of the Milnor number of a linear function on $X_0$ (\cite{MvS01}). Let $F \colon (\mathbb{C}^n,0) \rightarrow (\mathbb{C},0)$ be a general linear functional. Set $f:=F|_{X_0}$. Let $n \colon \overline{(X_0,0)} \rightarrow (X_0,0)$ be the normalization morphism. Denote by $\omega_{X_0,0}$ the dualizing module of Grothendieck. 
Consider the composition of the maps $\mathcal{O}_{X_0,0} \rightarrow \Omega_{X_0,0}^{1}$ given by $df$ and the map formed from the compositions $\Omega_{X_0,0}^{1}\rightarrow n_{*}\Omega_{\overline{X_0,0}}^1 \cong n_{*}\omega_{\overline{X_0,0}} \rightarrow \omega_{X_0,0}$.
This produces an injective map $\mathcal{O}_{X_0,0} \rightarrow \omega_{X_0,0}$ with cokernel $\mathcal{M}_{f}$:
$$0 \rightarrow \mathcal{O}_{X_0,0} \rightarrow \omega_{X_0,0} \rightarrow \mathcal{M}_f \rightarrow 0.$$
Note that the module $\mathcal{M}_{f}$ is of finite length.
\begin{definition}
The Milnor number of $f:=F|_{X_0}:X_0 \rightarrow (\mathbb{C},0)$ is defined as
$$\mu_0(f)=\mu(F|X_0):=\dim \mathcal{M}_f.$$
\end{definition}

In the family setup, by \cite[Cor.\ 3.2]{NT08} the Milnor number of a linear functional $\ell_y$ on a germ $(X_y,x)$ with $x \in (X_y)_{\mathrm{sing}}$ equals $\mu(X_y,x)+m(X_y,x)-1$ whenever $\ell_y$ is general for that germ, i.e.\ $\ell_y^{-1}(\ell_y(x))$ contains no tangent line of a branch of $(X_y,x)$.
Define $\mu_y(\ell_y):=\sum_{x \in (X_y)_{\mathrm{sing}}}\mu(\ell_y|X_y,x)$, where $\ell_y$ is general for every germ $(X_y,x)$, $x \in (X_y)_{\mathrm{sing}}$; it does not depend on the choice of $\ell_y$, and we set $\mu_y(f):=\mu_y(\ell_y)$. When $\dim Y=1$, the tangent lines of the branches of $X_y$ at its singular points converge, as $y \rightarrow 0$, to finitely many lines, so $F$ is general for every germ $(X_y,x)$, $x \in (X_y)_{\mathrm{sing}}$, with $y$ close enough to $0$, and $\mu_y(f)$ is the sum of the Milnor numbers of $f=F|_{X_y}$ at the singular points of $X_y$. By \cite[Cor.\ 3.2]{NT08} and \cite[Thm.\ 3.9]{BR24}, we obtain that 
\begin{equation}\label{algebraic MF}
\mu_y(f)=e(\mathrm{Jac}(X_y)) - \mathrm{cid}(X_y).
\end{equation}
This formula can be used to compute $\mu_y(f)$ in case the parametrization of $(X_y,x)$ with $x \in (X_y)_{\mathrm{sing}}$ is not known. Thm.\ \ref{main mpt} gives for $y$ generic
\begin{equation}\label{change in MF}
\mu_0(f)-\mu_y(f)=\mathrm{deg}_{Y}  \Gamma_{h}^{k}(X).
\end{equation}
Alternatively, for one-parameter families one can establish (\ref{change in MF}) for general $y \neq 0$ by noting that $\mu_0(f)=\mu(F|X_y)$ by \cite[Prop.\ 2.2]{MvS01} and $\mu(F|X_y)=\mu_y(f)+\mathrm{deg}_{Y}  \Gamma_{h}^{k}(X)$ by \cite[Rmk.\ 2.3]{MvS01} and Prop.\ \ref{relative Morse}. As an immediate corollary to Prop.\ \ref{Zariski-ups-cid} and (\ref{algebraic MF}), we obtain
\begin{corollary}\label{Ballesteros}
Let $(Y,0)$ be as in Prop.\ \ref{Zariski-ups-cid}. The function $y \mapsto \mu_y(f)$ is Zariski upper semicontinuous.
\end{corollary}

Using the generic smoothness of $\Gamma_{h}^{k}(X)$ (see Sct.\ \ref{polar}), Pablo Portilla Cuadrado and the third author proved the following result. 
 
\begin{proposition}\label{relative Morse}
    Let $F:(\mathbb{C}^n,0) \rightarrow (\mathbb{C},0)$ be a general linear functional. There exists a nonempty Zariski open subset $U_F\subset Y$ such that $F|_{(X_y)_{\mathrm{sm}}}$ with $y \in U_F$ has only Morse critical points. 
\end{proposition}
\begin{proof}
Adopt the setting of Sct.\ \ref{polar}. The relative conormal space $C_h(X)$ comes with two projection maps to $X$ and $\check{\mathbb{P}}^{n-1}$ which we denote by $c_h$ and $\lambda$, respectively. Recall that $\Gamma_h^k(X)$, the relative polar variety of dimension $k$, is defined as $\Gamma_h^k(X):=c_h(\lambda^{-1}(H))$ where $H\in\check{\mathbb{P}}^{n-1}$ is a general point. Set $H:=[\alpha_1: \cdots : \alpha_n]$. Define  $F:(\mathbb{C}^n,0) \rightarrow (\mathbb{C},0)$ as $F(x_1, \ldots, x_n): = \sum_{i=1}^n \alpha_ix_i$. Then set-theoretically $\Gamma_h^k(X)$ is precisely the closure of the critical points of $F |_{(X_y)_{\mathrm{sm}}} \rightarrow (\mathbb{C},0)$ as $y$ varies through $Y$. Set $f:=F|_{X}=F \circ \Phi$ where $\Phi$ is the composition of maps $X \hookrightarrow \mathbb{C}^n \times Y \rightarrow \mathbb{C}^n$. The equations for the subset $\Gamma_h^k(X) \setminus S_h$ of $X \setminus S_h$ are given by the augmented relative Jacobian ideal $\mathrm{Jac}_h(X,f)$ in $\mathcal{O}_X$.

Furthermore, by Kleiman's transversality, a general $H$ guarantees that $\Gamma_h^k(X)$ is reduced as explained in Sct.\ \ref{polar}. Moreover, again by Kleiman's transversality result \cite[Rmk.\ 6 and 7]{Kl74} for a general $F$, the subspace of $X \setminus S_h$ defined by $\mathrm{Jac}_h(X,f)$ is generically reduced.

Note that $\codim(c_h^{-1}(S_h), C_h(X)) \geq 1$. Thus we have $\dim c_h(c_h^{-1}(S_h) \cap \lambda^{-1}(H)) \leq k-1$ for a general $H$. So $\dim(\Gamma_h^k(X) \cap S_h) \leq k-1$.
By generic smoothness applied to $\Gamma_h^k(X)$ and the morphism
$\Gamma_h^k(X) \rightarrow Y$, there exists a Zariski open subset
$U_F$ of $Y$ such that
$\Gamma_{U_F}:= \Gamma_h^k(X) \times_Y U_F \rightarrow U_F$
is smooth with smooth fibers and $\Gamma_{U_F} \cap S_h = \emptyset$.
Let $x_0 \in \Gamma_{U_F}$. Set $y_0:=h(x_0)$.
Let $V$ be a small Euclidean neighborhood of $x_0$ in
$X\setminus S_h$ with $h(V) \subset U_F$.
Pick local coordinates $(u,t_1,\ldots,t_k)$ for $V$, where $u$
is a uniformizing parameter for $(X_{y_0},x_0)$.
Consider the map $(f,h): X \rightarrow \mathbb{C} \times Y$
defined by $x \mapsto (f(x),h(x))$.
Let $g: V \rightarrow \mathbb{C} \times U_F$ be the restriction
of $(f,h)$ to $V$.
Because the formation of Fitting ideals commutes with base change,
the ideal sheaf $C_{g,x_0}$ of the critical locus $(C_g,x_0)$ of $g$
is defined by
$\mathrm{Jac}_h(X,f) \otimes_{\mathcal{O}_X} \mathcal{O}_{V,x_0}$
(see \cite[Chp.\ 4 A]{Loo84}).
Choosing the above local coordinates adapted to $h$, so that
$h(u,t_1,\ldots,t_k)=(t_1,\ldots,t_k)$, this ideal is
$(\partial f/\partial u)$.
Thus $(C_g,x_0)$ is a hypersurface in the smooth space $V$, and
hence Cohen--Macaulay; being generically reduced, it is reduced
and therefore coincides with $(\Gamma_{U_F},x_0)$.
Since $\Gamma_{U_F} \rightarrow U_F$ is smooth of relative
dimension zero, its fiber at $x_0$ is reduced.
Therefore, the restriction of $\partial f/\partial u$
to $(X_{y_0},x_0)$ has a simple zero at $x_0$, so
$\partial^2 f/\partial u^2(x_0)\neq 0$.
This proves that $F|_{(X_y)_{\mathrm{sm}}}$ with $y \in U_F$
has only Morse critical points.
\end{proof}
\textbf{Remark.}
Under the same standing hypotheses, Prop.\ \ref{relative Morse}
extends to equidimensional, generically reduced fibers of any
fixed dimension $d>0$ with at most isolated singularities.
Kleiman's transversality ensures that $\Gamma_h^k(X)\to Y$ is finite after shrinking representatives. On $X\setminus S_h$, the critical scheme is locally a  complete intersection of codimension $d$ defined by $(f_{u_1},\ldots,f_{u_d})$. Hence it is Cohen--Macaulay and by generic reducedness, it is reduced. Its reduced zero-dimensional fibers over a suitable $U_F$ then imply that the Hessian is nondegenerate.

As in Prop.\ \ref{Euler}, we choose representatives for $(X,0)$ and $(Y,0)$ as follows. Assume $(X_0,0)$ is contained in an open ball $\mathring{B_0}=\mathring{B}(0,\epsilon)\subset(\C^n,0)$. View $(X,0)$ as analytic subspace of $\mathring{B_0} \times (Y,0)$.  Write $\chi(X_y)$ for the Euler characteristic of the Euclidean closure of $X_y$ in $B_0 \times \{y\}$.
Set $\mu_y:=\mu(X_y)=\sum_{x \in (X_y)_{\mathrm{sing}}}\mu(X_y,x)$.

 
 
 
\begin{theorem}\label{ZU of Milnor}
The following hold. 
\begin{enumerate}
\item[\rm{(i)}] Suppose $Y$ is smooth. For $\mathring{B_0}$ and $Y$ sufficiently small we have
\begin{equation}\label{Gen-BG}
\mu_0-\mu_y = 1-\chi(X_y)
\end{equation}
for $y$ in a Zariski open subset $U_\chi$ of $Y$.
\item[\rm{(ii)}] Suppose, more generally, that $h:(X,0) \rightarrow (Y,0)$ is a flat family of reduced curves as in the family setup, except that $(Y,0)$ is only assumed to be a reduced complex analytic germ. Then the function $y \mapsto \mu_y$ is Zariski upper semicontinuous.
\end{enumerate}
\end{theorem}
\begin{proof} Consider \rm{(i)}. Denote by $m(X_y,x)$ the multiplicity of $x \in X_y$. Set $m_y:=m(X_y,0)$. Thm.\ \ref{main mpt} and \cite[Thm.\ 3.9]{BR24} give us
 
\begin{equation}\label{Milnor diff.}
\mu_0-\mu_y = \mathrm{deg}_{Y}  \Gamma_{h}^{k}(X) + \sum_{x \in (X_y)_{\mathrm{sing}}} (m(X_y,x)-1) - (m_0-1).
\end{equation}
for $y$ in a Zariski open subset $U_{\Gamma}$  of $Y$. Consider the ideal  $J':=\sqrt{\mathrm{Jac}_{h}(X)}$ in $\mathcal{O}_{X,0}$. Set $S_h^r:=\mathbb{V}(J').$ Note that by shrinking $\mathring{B_0}$ and $Y$ we can assume that $S_h^r \rightarrow Y$ is finite. If $S_h^r \rightarrow Y$ is not surjective, then there exists a Zariski open subset $U$ of $Y$ such  that $X_y$ is smooth for $y \in U$. In this case (\ref{Gen-BG}) follows upon combining (\ref{Milnor diff.}) and (\ref{Le}). Thus we may assume that $S_h^r \rightarrow Y$ is surjective.
 
{\bf Generic constancy.} Denote by $J'(y)$ the image of $J'$ in $\mathcal{O}_{X_y}$. Because $S_h^r$ is reduced, by generic smoothness (\cite[\href{https://stacks.math.columbia.edu/tag/01V9}{Tag 01V9}]{Stacks}), we get that the fibers $(S_h^r)_y$ are smooth for $y$ in a Zariski open dense subset $U_{S_h^r}$ of $Y$. Thus, for  $y \in U_{S_h^r}$ we have $J'(y) = \cap_{x \in (X_y)_{\mathrm{sing}}}\mathfrak{m}_{y,x}$ where $\mathfrak{m}_{y,x}$ is the maximal ideal of $x$ in $\mathcal{O}_{X_y}$. Set $e(J'(y)):=\sum_{x \in (X_y)_{\mathrm{sing}}}e(J'(y)\mathcal{O}_{X_y,x})$ where  $e(J'(y)\mathcal{O}_{X_y,x})$ is the Hilbert-Samuel multiplicity of $J'(y)\mathcal{O}_{X_y,x}=\mathfrak{m}_{y,x}$ in $\mathcal{O}_{X_y,x}$. Following the proof of Prop.\ \ref{Zariski-ups-cid}, we obtain that $y \mapsto e(J'(y))$ is constant for $y$ in a Zariski open dense subset $U_{J'}$ of $Y$. Thus for $y \in U_m:=U_{S_h^r} \cap U_{J'}$ the function $y \mapsto \sum_{x \in (X_y)_{\mathrm{sing}}} (m(X_y,x)-1)$ is constant. Finally, we conclude that $y \mapsto \mu_y$ is constant for $y \in U_{\mu}:=U_{\Gamma} \cap U_m$.

{\bf Connectedness.}
Next, we show that there exists a Zariski open subset $U_c$ in $Y$ such that the fibers $X_y$ are connected. We give an outline of a proof that appears in \cite{PR26}. One can select $U_c$ to be the Zariski open subset over which $S_{h}^r\times_Y U_c \rightarrow U_c$ is a holomorphic covering and the Whitney conditions hold locally at the points $(S_{h}^r)_y$. This gives a local topological trivialization of $X \times_Y U_c \rightarrow U_c$ (see \cite[proof of Lem.\ 2.13 in the arXiv version]{Gre20}). By a monodromy argument (\cite[Thm.\ 8.2]{Gre17}) one can find a smooth curve $(T,0)$ in $(Y,0)$ with $T\setminus \{0\} \subset U_c$ such that $X \times_Y T \rightarrow T$ has connected fibers. Thus $X_y$ with $y \in U_c$ has connected fibers.

{\bf Euler characteristic and upper semicontinuity.} Let $(\mathbf{x},\mathbf{y}):=(x_1,\ldots,x_n,y_1, \ldots, y_k)$ be coordinates for $\mathring{B_0} \times (Y,0)$. Let $F \colon \mathring{B_0} \rightarrow (\mathbb{C},0)$ be the linear functional $\mathbf{x} \mapsto \sum_{i=1}^{n} \alpha_i x_i$. Below we specify the genericity conditions on the $\alpha_i$s, and we also specify our choices for $\mathring{B_0}$ and a representative of $(Y,0)$:
\begin{itemize}
 
\item[(i)] The $\alpha_i$ are general so that $\Gamma_{h}^{k}(X)$ is reduced. Thus by Prop.\ \ref{relative Morse} there exists a Zariski open subset $U_F\subset Y$ such that $F_{|(X_y)_{\mathrm{sm}}}$ with $y \in U_F$, has only Morse critical points. 
 
\item[(ii)] Pick $y \in U_m$. Pick $\alpha_i$ general so that for each $x \in (S_{h}^r)_y$ we have $F^{-1}(F(x)) \cap E_{X_y,x} = \{x\}$ where $E_{X_y,x}$ is the tangent cone of $X_y$ at $x$. Algebraically, this condition means that $(\sum_{i=1}^n\alpha_i(x_i-z_i))$ generates a reduction of the maximal ideal of $\mathcal{O}_{X_y,x}$, where $x=(z_1,\ldots,z_n)$ and $x \in (S_{h}^r)_y$. Because $y \mapsto e(J'(y))$ is constant on $U_m$, by Teissier's Principle of Specialization of Integral Dependence \cite{T73}, it follows that $(\sum_{i=1}^n\alpha_i(x_i-z_i))$ will generate a reduction of the maximal ideals of $\mathcal{O}_{X_y,x}$ for each $y \in U_m$ and $x \in (S_{h}^r)_y$ after possibly shrinking $U_m$. Finally, assume the $\alpha_i$s are general so that $(\sum_{i=1}^n\alpha_ix_i)$ generates a reduction of the maximal ideal of $\mathcal{O}_{{X_0},0}$.
 
\vspace{.2cm}
Fix an $F$ satisfying $\rm{(i)}$ and $\rm{(ii)}$. Next we describe the conditions on $\mathring{B_0}$ and $(Y,0)$.
\vspace{.01cm}
 
\item[(iii)] The restriction of the map $\tilde{F}: \mathring{B_0} \times (Y,0) \rightarrow (\mathbb{C},0) \times (Y,0)$ to $X$ given by $\tilde{F}(\mathbf{x},\mathbf{y})=(F(\mathbf{x}),\mathbf{y})$ is a finite map after possibly shrinking $\mathring{B_0}$ and $Y$ and taking a sufficiently small closed disk $\mathbb{D}=\mathbb{D}(0,\delta)$ as a representative of $(\mathbb{C},0)$. Set $f:=\tilde{F}|X$ and $f_y:=F|_{X_{y}}$. Note that $\mathcal{O}_{{X},0}/f$ is Cohen--Macaulay and that our choices for $\mathring{B_0}$ and $Y$ ensured that $\mathcal{O}_{{X},0}/f$ is a finite $\mathcal{O}_{Y,0}$-module. Thus, by flatness $\dim_{\mathbb{C}}\mathcal{O}_{X_0,0}/(f_0)=\dim_{\mathbb{C}}\mathcal{O}_{X_y}/(f_y)$ for all $y \in Y$. Here and below, whenever the finite map $f_y$ is involved, $X_y$ stands for the representative $X_y \cap \mathring{B_0} \cap F^{-1}(\mathring{\mathbb{D}})$ on which $f_y$ is finite; by (\ref{ball-polydisc}) below this does not affect $\chi(X_y)$.
 
\item[(iv)] $\mathring{B_0}$ is sufficiently small and $Y$ is sufficiently small with respect to $\mathring{B_0}$ so that $(\mathring{B_0}\times \{0\}) \cap (X_0)_{\mathrm{sing}}=\{0\}$,  $(\partial B_0 \times \{y\}) \cap (X_y)_{\mathrm{sing}} = \emptyset$, and $S_h^r \rightarrow Y$ is finite.
Furthermore, one can take $\mathring{B_0}$ small enough so that $X_0$ is contractible, which by Prop.\ \ref{prop:Euler_char} and our choice of $F$ in \rm{(ii)} is equivalent to $\mathring{B_0}\times \{0\} \cap \Gamma_{h}^{k}(X) \subset \{0\}$. 
 
\item[(v)] Choose $\mathring{B_0}$ so that $\partial B_0$ intersects $X_0$ transversally; by openness of transversality after possibly shrinking $Y$ we have that $\partial B_0 \times \{y\}$ intersects $X_y$ transversally.

 \end{itemize}
Before proceeding, we compare the Euler characteristic of $X_y$ with that of the representative of $X_y$ on which $f_y$ is finite. We claim that for all $y$ close enough to $0$
\begin{equation}\label{ball-polydisc}
\chi(X_y)=\chi\big(X_y \cap B_0 \cap F^{-1}(\mathbb{D})\big),
\end{equation}
where $\mathbb{D}$ is the representative of $(\mathbb{C},0)$ fixed in \rm{(iii)}. Indeed, set $A_y:=X_y \setminus F^{-1}(\mathring{\mathbb{D}})$, where $X_y$ is identified with its closure in $B_0 \times \{y\}$. By \rm{(iv)}, $(\mathring{B_0}\times \{0\}) \cap \Gamma_{h}^{k}(X) \subset \{0\}$ and $(X_0)_{\mathrm{sing}}=\{0\}$; since $\Gamma_{h}^{k}(X)$ and $S_h^r$ are closed and finite over $Y$, for $y$ close to $0$ the set $A_y$ contains neither singular points of $X_y$ nor critical points of $F|_{X_y}$, and $X_y \cap \partial B_0 \cap F^{-1}(\mathbb{D})=\emptyset$. For $y=0$, each branch of $(X_0,0)$ meets $B_0$ in a closed disk $\Delta_i$ (the cone over its link) on which $F$ is, in a uniformizing parameter $s$, of the form $s^{q_i}u(s)$ with $u$ a unit; hence for $\delta$ small $\Delta_i \cap F^{-1}(\mathbb{D})$ is a closed disk in the interior of $\Delta_i$ and $A_0$ is a disjoint union of $r_0$ annuli. By \rm{(v)} and the choice of $\delta$, $X_y$ is transverse to $\partial B_0 \times \{y\}$ and to $F^{-1}(\partial\mathbb{D}) \times \{y\}$ for $y$ close to $0$. By Ehresmann's fibration theorem applied to $\bigcup_{y}A_y \rightarrow (Y,0)$, which is a proper submersion of manifolds with boundary, $A_y$ is diffeomorphic to $A_0$. Now $X_y=A_y \cup (X_y \cap B_0 \cap F^{-1}(\mathbb{D}))$ and the two pieces meet along $X_y \cap F^{-1}(\partial\mathbb{D})$, a disjoint union of $r_0$ circles. Since annuli and circles have vanishing Euler characteristic, additivity of $\chi$ gives (\ref{ball-polydisc}).

Fix $y \in Y$. Consider the trivial family $X_y \times (\mathbb{C},0) \rightarrow (\mathbb{C},0)$, where for $(\mathbb{C},0)$ we choose the representative fixed in \rm{(iii)}.
We have $X_y \times (\mathbb{C},0) \subset \mathring{B_0} \times (\mathbb{C},0)$. Let $(x_1, \ldots,x_n,u)$ be coordinates on $\mathring{B_0} \times (\mathbb{C},0)$. View $g_u: = \sum_{i=1}^n \alpha_{i}x_i - u$ as an element in $\mathcal{O}_{X_y \times (\mathbb{C},0)}$. Set  $\mathbb{V}(g_u):=\mathrm{Specan}(\mathcal{O}_{X_y \times (\mathbb{C},0)}/(g_u))$. Because $X_y \times (\mathbb{C},0)$ is Cohen--Macaulay, so is $\mathbb{V}(g_u)$. Because $f_y: X_y \rightarrow (\mathbb{C},0)$ is finite by \rm{(iii)} we have that $\mathcal{O}_{X_y \times (\mathbb{C},0)}/(g_u)=\mathcal{O}_{X_y}[u]/(g_u)=\mathcal{O}_{X_y}$ is a finite $\mathcal{O}_{\mathbb{C},0}$-module. By Bertini's theorem $\dim_{\mathbb{C}}\mathcal{O}_{X_y}/g_c=\#(f^{-1}(c))=\#(F^{-1}(c) \cap X_y)$ for $c$ generic. Because  $\mathbb{V}(g_u)\rightarrow (\mathbb{C},0)$ is finite and flat, we get that 
$$\dim_{\mathbb{C}} \mathcal{O}_{X_y}/f_y=\dim_{\mathbb{C}}\mathcal{O}_{X_y}/g_0=\dim_{\mathbb{C}}\mathcal{O}_{X_y}/g_c=\# (F^{-1}(c) \cap X_y)=\mathrm{deg}(F|X_y)=:d_y$$
for generic $c \in (\mathbb{C},0)$. By our choices for $F$ and $(Y,0)$ in (ii) and (iii) we have

 
 
\begin{equation}\label{mult. diff.}
m_0=\dim_{\mathbb{C}}\mathcal{O}_{X_0,0}/(f_0)=\dim_{\mathbb{C}}\mathcal{O}_{X_y}/(f_y)=d_y.
\end{equation}
Combining (\ref{Milnor diff.}) with (\ref{mult. diff.}), we obtain
\begin{equation}\label{Milnor diff.2}
\mu_0-\mu_y = \mathrm{deg}_{Y}  \Gamma_{h}^{k}(X) + \sum_{x \in (X_y)_{\mathrm{sing}}} (m(X_y,x)-1) - (d_y-1).
\end{equation}
for $y \in U_\Gamma$. Further, restrict $y$ to $U_\chi:=U_\Gamma \cap  U_m \cap U_F$. By (i) and (ii), $F$ satisfies the genericity conditions of Sct.\ \ref{Sct-RH} on $X_y$, and by (iii) its restriction to $X_y \cap \mathring{B_0} \cap F^{-1}(\mathring{\mathbb{D}})$ is a finite map of degree $d_y$. Note that $\mathrm{deg}_{Y} \Gamma_{h}^{k}(X)=\#\mathrm{Crit}(F|_{(X_y)_{\mathrm{sm}}})$. Therefore, by Prop.\ \ref{prop:Euler_char} applied with $\mathcal{X}=X_y \cap \mathring{B_0} \cap F^{-1}(\mathring{\mathbb{D}})$, together with (\ref{ball-polydisc}), we can rewrite (\ref{Milnor diff.2}) as 
$$
\mu_0-\mu_y = 1-\Big(d_y-\Big(\#\mathrm{Crit}(F|_{(X_y)_{\mathrm{sm}}})
+\sum_{x\in(X_y)_{\mathrm{sing}}}(m_x-1)\Big)\Big)=1-\chi(X_y)
$$
proving (\ref{Gen-BG}). Because $X_y$ is reduced and connected for $y \in U_c$, Prop.\ \ref{nonneg} \rm{(i)} and (\ref{Gen-BG}) give $\mu_0 \geq \mu_y$ for $y \in U_{\chi} \cap U_c$. This proves \rm{(i)}.
 
We now prove \rm{(ii)}. Let $(Y,0)$ be a reduced complex analytic germ and let $(A,0)$ be an irreducible complex analytic subgerm of $(Y,0)$ equipped with its reduced structure. By Lem.\ \ref{ZU-equiv} it suffices to show that there exist a nonempty Zariski open subset $U_A$ of $A$ and an integer $c_A$ such that $\mu_y=c_A$ for $y \in U_A$ and $\mu_y \geq c_A$ for all $y \in A$. Replacing $Y$ by $A$ and $X$ by $X \times_{Y} A$, which is a flat family of reduced curves over $A$ with reduced total space, we may assume that $Y=A$ is irreducible. Note that, because $S_h^r \rightarrow Y$ is finite with $(S_h^r)_0=\{0\}$, the function $y \mapsto \mu_y$ does not depend on the choice of $\mathring{B_0}$ for $y$ close enough to $0$.
 
{\bf Generic constancy.} First, we show that $y \mapsto \mu_y$ is constant on a nonempty Zariski open subset of $Y$. Denote by $Y_{\mathrm{sm}}$ the smooth locus of $Y$; it is a nonempty Zariski open subset of $Y$, and it is connected because $Y$ is irreducible. Let $T \subset S_h$ be the set of points $s$ at which the fiber $c_h^{-1}(s)$ of the relative conormal space $C_h(X)$ has dimension $n-1$. By upper semicontinuity of fiber dimension, $T$ is an analytic subset of $S_h$. Because $C_h(X)$ is the closure of its restriction to $X \setminus S_h$, we have $\dim c_h^{-1}(S_h) < \dim C_h(X)=\dim Y+n-1$, and so $\dim T \leq \dim Y-1$. As $S_h \rightarrow Y$ is finite, $U_\Gamma:=Y \setminus h(T)$ is a nonempty Zariski open subset of $Y$. Let $U_{S_h^r}$ be the complement in $Y$ of the image of the relative singular locus of the finite map $S_h^r \rightarrow Y$; it is a nonempty Zariski open subset of $Y$, over which $S_h^r \rightarrow Y$ is a covering. Locally on $U_{S_h^r} \cap Y_{\mathrm{sm}}$ write $S_h^r=\bigsqcup_{i=1}^{l}S^{(i)}$ and $(S_h^r)_y=\{s^{(1)}_y, \ldots, s^{(l)}_y\}$. Following the proof of Prop.\ \ref{Zariski-ups-cid} over the smooth connected base $U_{S_h^r} \cap Y_{\mathrm{sm}}$, we obtain a nonempty Zariski open subset $U_{J'}$ of $Y$ such that $y \mapsto e(J'(y))$, and hence each $y \mapsto m(X_y,s^{(i)}_y)$, is constant on $U_m:=U_{S_h^r} \cap U_{J'} \cap Y_{\mathrm{sm}}$. Set $U_\mu:=U_\Gamma \cap U_m$; it is a nonempty Zariski open subset of $Y$ contained in $Y_{\mathrm{sm}}$, and it is connected. For $y \in U_\mu$ and each $i$, the germ $(Y,y)$ is smooth and $\dim c_h^{-1}(s^{(i)}_y)=n-2$, so the polar variety of dimension $\dim Y$ of the family $(X,s^{(i)}_y) \rightarrow (Y,y)$ is empty. Applying (\ref{Milnor diff.}) to this family we conclude that $\mu(X_{y'},s^{(i)}_{y'})=\mu(X_y,s^{(i)}_y)$ for $y'$ in a Zariski open dense subset of a neighborhood of $y$ in $U_\mu$. Thus $y \mapsto \mu_y=\sum_{i=1}^{l}\mu(X_y,s^{(i)}_y)$ is locally generically constant on $U_\mu$; because $U_\mu$ is connected, its generic value is a constant $c_\mu$, and $\mu_y=c_\mu$ for every $y \in U_\mu$.
 
{\bf Lower bound.} Second, we show that $\mu_{y} \geq c_\mu$ for each $y \in Y$. Fix $y \in Y$ and let $x_1, \ldots, x_w$ be the singular points of $X_y$. Since $Y$ is irreducible at $0$, every irreducible component of the germ $(Y,y)$ has dimension $\dim Y > \dim (Y \setminus U_\mu)$; so by prime avoidance (or the curve selection lemma) there is an analytically irreducible curve germ $(Y',y) \subset (Y,y)$ with $Y' \setminus \{y\} \subset U_\mu$. Let $n':\overline{Y'} \rightarrow Y'$ be the normalization; it is a homeomorphism, $\overline{Y'}$ is a smooth curve germ at $y':=(n')^{-1}(y)$, and $X':=X \times_{Y} \overline{Y'} \rightarrow \overline{Y'}$ is a flat family of reduced curves with $X'_{y'}=X_y$. Assume each $(X_y,x_i)$ is contained in a small open ball $\mathring{B_i} \subset \mathring{B_0}$ centered at $x_i$ with $B_i \cap B_j = \emptyset$ for $i \neq j$, and let $h_i':X_i' \rightarrow \overline{Y'}$ with $X_i' \subset \mathring{B_i} \times (\overline{Y'},y')$ be the induced embedded deformation of $(X_y,x_i)$. Because $S_h^r \rightarrow Y$ is finite, after shrinking $\overline{Y'}$ every singular point of $X'_{z}$ lies in $\bigcup_{i=1}^{w}\mathring{B_i} \times \{z\}$, so that $\mu(X'_z)=\sum_{i=1}^{w}\mu((X_i')_z)$ for $z \in \overline{Y'}$. Each $h_i'$ is a flat embedded deformation of the reduced curve $(X_y,x_i)$ over a smooth curve germ. By \rm{(i)} and Prop.\ \ref{nonneg} \rm{(i)} applied to $h_i'$ (whose generic fibers are connected, as shown above), we get
$$\mu(X_y,x_i)-\mu((X_i')_z)=1-\chi((X_i')_z) \geq 0 \ \ \text{for} \ z \neq y' \ \text{close enough to} \ y'.$$
Summing over $i$ and using $n'(z) \in U_\mu$, we obtain $\mu_y \geq \mu(X'_z)=\mu_{n'(z)}=c_\mu$.
 
Finally, by Lem.\ \ref{ZU-equiv} applied with $U_A:=U_\mu$ and $c_A:=c_\mu$ for each irreducible subgerm $(A,0)$ of $(Y,0)$, the function $y \mapsto \mu_y$ is Zariski upper semicontinuous. In particular, $\mu_0 \geq \mu_y$ for all $y$ close enough to $0$ (see the remark after Def.\ \ref{def:ZU}). The proof is now complete. 
\end{proof}

{\bf A Topological Approach.} Below we provide an alternative, topological proof of the inequality $\mu_0 \geq \mu_{y}$ for $y$ sufficiently close to $0$, which is a consequence of the Zariski upper semicontinuity of $y \mapsto \mu_y$. We consider the case when $X_0$ and $X_y$, with $y \in U_\chi$, have the same number of irreducible components, and where the irreducible components of each $X_y$ meet at a single point and have no other intersections. This configuration arises in Thm.\ \ref{top. equis.}, which constitutes the main application of Zariski upper semicontinuity of the Milnor number in this paper. 
 
Assume $Y$ is smooth. Fix a point $y' \in Y$. Let $Y'$ be a smooth curve
passing through $y'$ such that $Y' \setminus\{y'\}$ lies in $U_{\chi} \cap U_c$. Set $h':X':=X \times_{Y}Y' \rightarrow Y'$. Let $x_1, \ldots, x_w$ be the singular points of $X'_{y'}=X_{y'}$. Assume each $(X'_{y'},x_i)$ is contained in a sufficiently small open ball $\mathring{B_i}$ in $\mathbb{C}^n$ centered at $x_i$ such that $B_i \cap B_j = \emptyset$ for each $i \neq j$ with $i,j \in \{1, \ldots, w\}$. Let $X_i' \rightarrow Y'$ with $X_{i}' \subset \mathring{B_i} \times (Y',y')$ be the induced embedded deformation of $(X'_{y'},x_i)$ from $X'\rightarrow Y'$. Let $y \in Y'$ with $y' \neq y$ be sufficiently small with respect to all families $X_i' \rightarrow Y'$ so that (\ref{Gen-BG}) is valid for each one of them. For $i=1, \ldots, w$ we have 
\begin{equation}\label{Milnor-germs}
\mu(X_{y'},x_i)-\mu((X_i')_y)=1-\chi((X_i')_y).
\end{equation}
Note that $\mu_{y'}=\sum_{i=1}^{w}\mu(X_{y'},x_i)$ and $\mu_y=\sum_{i=1}^{w}\mu((X_i')_y)$. Summing up the identities (\ref{Milnor-germs}), we obtain 
\begin{equation}\label{Milnor-germs-sum}
\mu_{y'}-\mu_y=w-\sum_{i=1}^{w}\chi((X_i')_y)
\end{equation}
Subtracting (\ref{Milnor-germs-sum}) from (\ref{Gen-BG}), we get
\begin{equation}\label{Milnor-comp.}
\mu_0-\mu_{y'}=1-w - \chi(X_y)+\sum_{i=1}^{w}\chi((X_i')_y).
\end{equation}
View $X_y$ and $(X_i')_y$ as real analytic surfaces and identify them with their closures in the Euclidean topology in $B_0 \times \{y\}$ and $B_i \times \{y\}$, respectively. Let $X_y^c$ be the closure of $X_y \setminus \bigcup_{i=1}^{w}(X_i')_y$. By (\rm{iv}), $S_h^r \rightarrow Y$ is finite. Thus $S_{h'}^r: =S_h^r \times_{Y}Y' \rightarrow Y'$ is finite. So by shrinking $Y'$ if necessary we can guarantee that set-theoretically $(S_{h'}^r)_{y'}=\{x_1, \ldots, x_w \}$ and each singular point of $X_y$ is contained in $\mathring{B_i} \times \{y\}$ for some $i \in \{1, \ldots, w\}$. Observe that the intersection of  $X_y^c$ with each $(X_i')_y$ is homeomorphic to a  disjoint union of circles. Thus $\chi(X_y)=\chi(X_y^c)+\sum_{i=1}^{w}\chi((X_i')_y).$
So we can rewrite (\ref{Milnor-comp.}) as 
\begin{equation}\label{Milnor-comp.2}
\mu_0-\mu_{y'}=1-w -\chi(X_y^c).
\end{equation}
Note that $X_y^c$ is a smooth, compact (possibly disconnected) surface. Write $X_y^c=\bigsqcup_{j=1}^m (X_y^c)_j$ where $(X_y^c)_j$ are the connected components of $X_y^c$. Denote by $b((X_y^c)_j)$ the number of boundary components of $(X_y^c)_j$. By applying the genus formula for each $(X_y^c)_j$ we get $$\chi(X_y^c)=\sum_{j=1}^m\chi((X_y^c)_j)=2m-2\sum_{j=1}^m  g((X_y^c)_j)-\sum_{j=1}^mb((X_y^c)_i).$$
Thus to prove  $1-w -\chi(X_y^c) \geq 0$ we have to show that 
 $$1-w-2m + \sum_{j=1}^mb((X_y^c)_i)+ 2\sum_{j=1}^m  g((X_y^c)_j) \geq 0.$$
Because $g((X_y^c)_j) \geq 0$, it is  enough to show that 
\begin{equation}\label{genus-components}
1+\sum_{j=1}^mb((X_y^c)_j) \geq 2m + w.
\end{equation}
Suppose $w=1$. This situation occurs, for example, when the relative singular locus does not split, i.e.\ $(S_h^r)_y$ is a single point for all $y \in (Y,0)$. In that case (\ref{genus-components}) holds because $b((X_y^c)_j) \geq 2$ for each $j$. 
 
Now assume that $X_0$ and $X_y$, with $y \in U_\chi$, have the same number of irreducible components, and assume that the irreducible components of each $X_y$ meet at a single point and do not have other intersections. Denote by $r_0$ the number of irreducible components of $(X_0,0)$. By Ehresmann's fibration theorem applied to $(\partial B_0 \times (Y,0)) \cap X \rightarrow (Y,0)$ we get that the number of boundary components of $X_y$ is $r_0$, which by assumption is the number of irreducible components of $X_y$. Thus the boundary of each irreducible component of $X_y$ is diffeomorphic to a circle that lies on the sphere $\partial B_0 \times \{y\}$. 
 
Identify each $B_i \times \{y\}$ with $B_i$. Note that intersecting each irreducible component of $X_y$ with a small ball $B_i$, and removing that intersection from the component, does not disconnect it. Therefore, the connected components $(X_y^c)_j$ correspond precisely to the irreducible components of $X_y$ with certain number of open disks (up to diffeomorphism) removed from them. Without loss of generality we can assume that the point of intersection of the irreducible components of $X_y$ is contained in $B_1$. Then each $(X_y^c)_j$ has at least two boundary components: one lies of $\partial B_0$ and the other one on $\partial B_1$. Because the generic fiber $(X_{i}')_y$ is connected, each ball $B_i$ with $i>1$, intersects exactly one irreducible component of $X_y$. Therefore the balls $B_2, \ldots, B_w$ contribute in total at least $w-1$ further boundary components to the $(X_y^c)_j$. Thus 
$$1+\sum_{j=1}^mb((X_y^c)_j) \geq 1+(w-1)+2m =2m+w$$
proving (\ref{genus-components}). An illustration is provided in the figure below.
\begin{figure}[ht]
\includegraphics[scale=0.38]{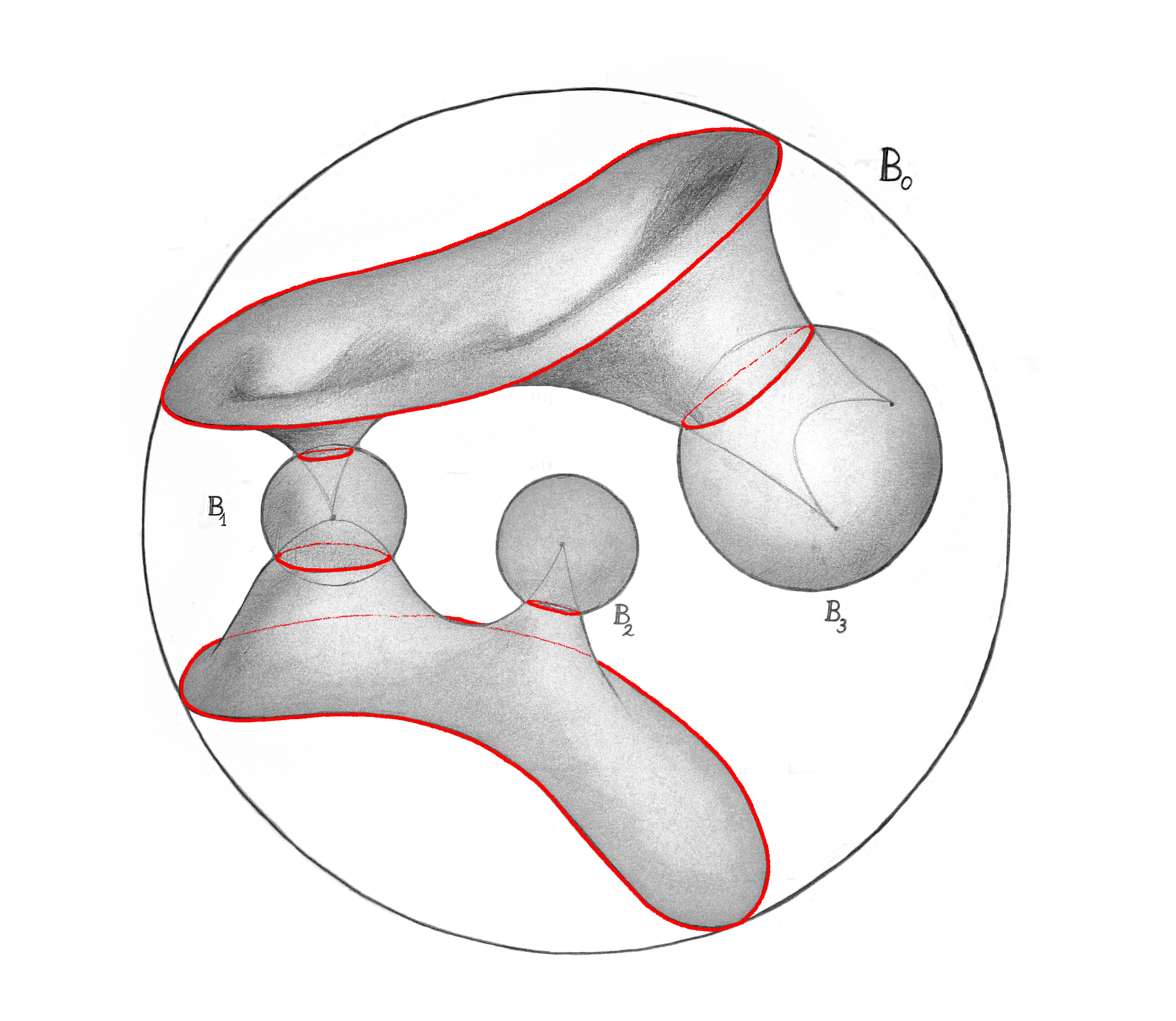}
\caption{Here $X_y$ has two irreducible components meeting at point. The boundary components of $(X_{y}^c)_1$ and $(X_{y}^c)_2 $ are colored with red.}\label{otherfig}
\end{figure}
 
 
\textbf{Remarks.} Note that Prop.\ \ref{Zariski-ups-cid}, Cor.\ \ref{Ballesteros} and part \rm{(ii)} of Thm.\ \ref{ZU of Milnor} hold for an arbitrary reduced base $(Y,0)$, whereas (\ref{Gen-BG}) is only established for $Y$ smooth. 
Furthermore, observe that the function $y \mapsto \chi(X_y)$ is constant by restricting $y$ to the Zariski open subset $U_{\chi} \cap U_\mu$. We obtained this without using the standard technique of invoking the genericity of Whitney conditions and the Thom--Mather isotopy lemma to obtain a generic topological triviality, which yields constancy of the Euler characteristic over a Zariski open subset of $Y$.
For one-parameter families ($\dim Y=1$) the relation (\ref{Gen-BG}) was proved by Buchweitz and Greuel in \cite[Thm.\ 4.2.2 (2)]{BG80} using coherence of hypercohomology with respect to the direct image functor $h_{*}$, and making use of an associated hypercohomological Gysin sequence. Another proof of differential-topological nature was given in \cite[Thm.\ 1.10]{MVS}.

Denote by $\nu: \overline{X} \rightarrow X$ the normalization morphism. For each $y \in Y$ denote by $\nu_y: \overline{X_y} \rightarrow X_y$ the normalization morphism on the level of the fibers. Set $\delta_y:= \dim_{\mathbb{C}}(\nu_y)_{*}(\overline{\mathcal{O}_{X_y}})/\mathcal{O}_{X_y}$. Denote  by $r(X_y,x)$ the number of branches of $(X_y,x)$. Set $r_y:=\sum_{x \in (X_y)_{\mathrm{sing}}}r(X_y,x)$. 
 
If $(X,0) \rightarrow (Y,0)$ is a smoothing of a reduced and reducible curve, then $X_y$ with $y \neq 0$ is smooth and connected and $r_0>r_y=0$ for $y \neq 0$. When $(X,0) \rightarrow (Y,0)$ is the degeneration of a node into a cusp, then $1=r_0<r_y=2$ for $y \neq 0$.  Thus, the  function $y  \mapsto r_y$ is neither upper nor lower semicontinuous.
\begin{corollary}
    There exists a Zariski open subset $U$ of $Y$ such that $y \mapsto r_y$ is constant.
\end{corollary}
\begin{proof}
Note that $$r_y = 2 \delta_y-\mu_y+\#(X_y)_{\mathrm{sing}}. $$ The Zariski upper semicontinuity of the Milnor number $\mu_y$ was established above. The singular locus of $\overline{X}$ has codimension at least two in $\overline{X}$, so its image in $Y$ is a proper analytic subset, and by generic smoothness applied to the smooth part of $\overline{X}$ the fibers $\overline{X}_y$ are smooth for $y$ in a Zariski open subset of $Y$. Thus for generic $y$ we have $\overline{X}_y=\overline{X_y}$. By \cite[Thm.\ 5.6]{CL06}) $y \mapsto \delta_y$ is constant for $y$ in a Zariski open subset $U_\delta$ of $Y$. The constancy of $y \mapsto \#(x \in (X_y)_{\mathrm{sing}})$ over a Zariski open subset of $Y$ follows from the generic smoothness of $S_h^r \rightarrow Y$. Combining these results with the identity above finishes the proof.
\end{proof}
\begin{corollary}[Buchweitz-Greuel]\label{BG ineq}
    We have $\mu_0-\mu_y \geq \delta_0-\delta_y$ for $y$ in a Zariski open subset of $Y$.
    \end{corollary}
\begin{proof} Consider the Zariski open subset $U_\delta \cap U_{\mu}$ of $Y$ over which $\delta_y$ and $\mu_y$ are both constant. Let $Y'$ be a smooth curve in $Y$ such that $Y' \setminus \{0\} \subset U_\delta \cap U_\mu$.  \cite[Thm.\ 4.2.2 (3)]{BG80} 
(which is essentially (\ref{Gen-BG}) applied to $\overline{X \times_{Y} Y'} \rightarrow Y'$) yields $\mu_0-\mu_y \geq \delta_0-\delta_y$ for $y \in Y'$ which is valid more generally for $y \in U_{\delta} \cap U_{\mu}$.
\end{proof}

\section{Equisingularity of curves}
\subsection{Strong simultaneous resolution}\label{simultaneous}
We begin with a result about the delta invariant. 
 
\begin{proposition}[Chiang-Hsieh--Lipman]\label{delta trivial}
 Let $(X,0) \rightarrow (Y,0)$ be a flat family of reduced curves with $Y$ smooth. Suppose $\delta_0 = \delta_y$ for $y$ in a nonempty Zariski open subset $U$ of $Y$, where $\delta_y$ denotes the (total) delta invariant of $X_y$. Then $y \mapsto \delta_y$ is constant after possibly shrinking $Y$.
\end{proposition}
\begin{proof} The proof is essentially the same as the second half of the proof of \cite[Thm.\ 5.6]{CL06}. Denote by $w:=h\circ \nu$, where $\nu: \overline{X} \rightarrow X$ is the normalization map. Denote by $K_y$ the fraction field of $\mathcal{O}_{Y,y}$. Then $y \mapsto \dim_{K_y}(w_{*}(\overline{\mathcal{O}_X}/\mathcal{O}_X)_y \otimes_{\mathcal{O}_{Y,y}}K_y)$ is constant, being the generic rank of the coherent sheaf $w_{*}(\overline{\mathcal{O}_X}/\mathcal{O}_X)$ on the irreducible $Y$; here $w_{*}(\overline{\mathcal{O}_X}/\mathcal{O}_X)_y$ is the stalk of $w_{*}(\overline{\mathcal{O}_X}/\mathcal{O}_X)$ at $y$. By generic smoothness there exists a sequence of points $y_1, y_2, \ldots$ approaching $0$ such that $\overline{X}_{y_i}=\overline{X_{y_i}}$ locally at $y_i$ in $U$. In particular, by \cite[Thm.\ 4.1]{CL06} applied at $y_i$, 
$$\delta_{y_i}=\dim_{K_{y_i}}(w_{*}(\overline{\mathcal{O}_X}/\mathcal{O}_X)_{y_i} \otimes_{\mathcal{O}_{Y,y_i}}K_{y_i})=\dim_{K_{0}}(w_{*}(\overline{\mathcal{O}_X}/\mathcal{O}_X)_{0} \otimes_{\mathcal{O}_{Y,0}}K_{0}).$$
For $i \gg0$ we have $y_i \in U$, and so $\delta_0=\delta_{y_i}$. Thus $\delta_0=\dim_{K_{0}}(w_{*}(\overline{\mathcal{O}_X}/\mathcal{O}_X)_{0} \otimes_{\mathcal{O}_{Y,0}}K_{0})$ which by the algebraic \cite[Thm.\ 4.1]{CL06} implies that $\overline{X}_y=\overline{X_y}$ for all $y$ close enough to $0$. Again, by \cite[Thm.\ 5.6]{CL06} this is equivalent to $y \mapsto \delta_y$ constant for all $y \in Y$ after possibly shrinking $Y$.
\end{proof}
Adopt the family setup with a section. Denote by $r_y:=|\nu_y^{-1}(0)|$ the \textit{number of branches at} $0 \in X_y$. Let $m_y:=m(X_y,0)$ be the multiplicity of $X_y$ at $0$.

\begin{theorem}\label{equimult} Suppose that $\Gamma_{h}^{k}(X)$ is empty, and suppose that the union of singular points of the fibers $X_y$ is $Y$ (i.e.\ $(S_h)_{\mathrm{red}}=Y$). Then $\delta_y$, $r_y$ and $m_y$ are constant for all $y$ sufficiently close to $0$. In particular, $(X,0) \rightarrow (Y,0)$ admits strong simultaneous resolution.
\end{theorem}
\begin{proof} 
Because $\Gamma_{h}^{k}(X)$ is empty, Thm.\ \ref{main mpt} implies that for $y$ in a Zariski open set $U_{\Gamma}$ of Y we have
$$e(\mathrm{Jac}(X_0)) - \mathrm{cid}(X_0)- (e(\mathrm{Jac}(X_y))-\mathrm{cid}(X_y))=0.$$
By \cite[Thm.\ 3.9]{BR24}, we get 
$$\mu_0-\mu_y+m_0-m_y=0$$
for $y \in U_{\Gamma}$. By Thm.\ \ref{ZU of Milnor} $\mu_0 \geq \mu_y$ for $y$ in a Zariski open subset $U_\mu$ of $Y$. Thus, $m_0=m_y$ for $y \in U_{\Gamma} \cap U_{\mu}$. But $m_{y'} \geq m_y$ for all $y$ close enough to $y'$, for every $y' \in Y$: indeed, let $f$ be a general linear form in $x_1, \ldots, x_n$ and denote by $f_y$ its image in $\mathcal{O}_{X_y,0}$; then $m_{y'}=\dim_{\mathbb{C}}\mathcal{O}_{X_{y'},0}/(f_{y'}) \geq \dim_{\mathbb{C}}\mathcal{O}_{X_y,0}/(f_y) \geq e(\mathfrak{m}_{X_y,0})=m_y$, by the upper semicontinuity of the fiber length of the finite $\mathcal{O}_Y$-module $\mathcal{O}_X/(f)$ along the section, and because $\dim_{\mathbb{C}}\mathcal{O}/(g) \geq e(\mathfrak{m})$ for every $g \in \mathfrak{m}$ in a one-dimensional Cohen--Macaulay local ring $(\mathcal{O},\mathfrak{m})$. Taking $y \in U_{\Gamma} \cap U_{\mu}$ close to $y'$ we get $m_0 \geq m_{y'} \geq m_y=m_0$. Therefore, $m_0=m_y$ for all $y$ close enough to $0$. Applying Thm.\ \ref{main mpt} to the germ of the family at $(0,y')$ for each $y' \in Y$ (its relative polar variety is empty as well), we get that $y \mapsto e(\mathrm{Jac}(X_y,0))-\mathrm{cid}(X_y,0)=\mu_y+m_y-1$ is constant for all $y$ close enough to $0$; hence $y \mapsto \mu_y$ is constant. By Cor.\ \ref{BG ineq} we have $\delta_0 \leq \delta_y$ for $y$ in a Zariski open subset of $Y$, while $\delta_0 \geq \delta_y$ by \cite[Sct.\ 3.3]{T77} applied along a smooth curve through $0$. Thus $\delta_0=\delta_y$ for generic $y$, and $y \mapsto \delta_y$ is constant by Prop.\ \ref{delta trivial}. Finally, $r_y=2\delta_y-\mu_y+1$ is constant.
 
Below we present an alternative proof of the constancy of $\delta_y$, $r_y$ and $m_y$ that does not use Thm.\ \ref{ZU of Milnor}. By the displayed identity above and $\mu_y=2\delta_y-r_y+1$, for $y \in U_\Gamma$ we have
\begin{equation}\label{m-r identity}
2(\delta_0-\delta_y)+(m_0-m_y)+(r_y-r_0)=0.
\end{equation}
 
Given that set-theoretically $\mathrm{Sing}(X)_y = \mathrm{Sing}(X_y)$ and $\overline{X}_y$ is smooth for $y$ in a Zariski open subset $U_\delta$ of $Y$, it follows that $y \mapsto \delta_y$ is constant for $U_\delta$. By selecting a smooth curve in $U_{\delta} \cup \{0\}$ and applying the results of \cite[Sct.\ 3.3]{T77}, we obtain $\delta_0 \geq \delta_y$. Denote by $U_m$ the Zariski open subset of $Y$ over which $y \mapsto m_y$ is constant. To get the constancy of the three functions in (\ref{m-r identity}) for all $y$ close enough to $0$ it is enough to show that $r_y \geq r_0$ along a smooth curve $Y'$ through $0$ with $Y' \setminus \{0\} \subset U_\Gamma \cap U_{\delta} \cap U_{m}$: this gives $\delta_0=\delta_y$, $m_0=m_y$ and $r_0=r_y$ for $y \in Y'$, hence for $y \in U_\Gamma \cap U_\delta \cap U_m$; then $y \mapsto \delta_y$ is constant near $0$ by Prop.\ \ref{delta trivial}, $y \mapsto m_y$ is constant near $0$ by the inequality $m_{y'} \geq m_y$ established in the first proof, and $y \mapsto \mu_y$, hence $y \mapsto r_y=2\delta_y-\mu_y+1$, is constant near $0$ by applying Thm.\ \ref{main mpt} at each $y' \in Y$ as in the first proof.
So in what follows, for the first part of the theorem, we may assume that $\dim Y=1$ and $Y \setminus \{0\} \subset U_\Gamma \cap U_{\delta} \cap U_{m}$.
 
Consider the normalization morphism $\nu:\overline{X}\to (X,0)$. Assume that for each  $x' \in \nu^{-1}(0)$ we have that $\overline{X}_0$ is smooth at $x'$ (this is equivalent to $\delta_0=\delta_y$ by \cite[Sct.\ 3.3]{T77}). Denote by $\nu^{-1}(Y)_{\mathrm{red}}$ the reduction of $\nu^{-1}(Y)$. The morphism $p:\nu^{-1}(Y)_{\mathrm{red}} \rightarrow Y$ is generically \'etale. We have $\#(p^{-1}(0))=r_0$ and $\#(p^{-1}(y))=r_y$ for $y$ generic. Some of the branches of $\nu^{-1}(Y)_{\mathrm{red}}$ might merge at $\nu^{-1}(0)$. Thus $r_0 \leq r_y$. Therefore, (\ref{m-r identity}) yields $r_0=r_y$, $m_0=m_y$ and $\delta_0=\delta_y$ establishing the first part of the theorem. 
 
We now show that $\overline{X}_0$ is smooth at every point of $\nu^{-1}(0)$, so that the previous case applies. Assume $(X,0) \subset (\mathbb{C}^{n+1},0)$. As usual embed $X_0$ in a small open $n$-ball $\mathring{B_0}:=\mathring{B}(0, \epsilon) \subset (\mathbb{C}^n,0)$. Replace $Y$ by an open disk centered at $0$. We may assume that $X$ is a closed analytic subset of $\mathring{B_0} \times Y$. Let $(x_1,\ldots,x_n,y_1)$ be coordinates for $(\mathbb{C}^{n+1},0)$. Let $F \colon (\mathbb{C}^{n+1},0) \rightarrow (\mathbb{C},0)$ be the linear functional $(x_1, \ldots, x_n, y_1) \mapsto (\sum_{i=1}^{n} \alpha_i x_i)$ where the $\alpha_i$ are general complex numbers. Recall that  $\Gamma_{h}^{1}(X)$ is $\overline{\Sigma_{X/Y}(F) \setminus \Sigma_{X/Y}}$ where $\Sigma_{X/Y}(F)$ is the union of the critical points of $F|X_y$ and $\Sigma_{X/Y}$ is the union of singular points of $X_y$. By assumption $\Gamma_{h}^{1}(X)$ is empty. So after $\mathring{B_0}$ is replaced by a smaller ball if necessary, we can assume that $\Sigma_{X/Y}(F)=\Sigma_{X/Y}$.
 

 
Fix $x' \in \nu^{-1}(0)$ and consider the composition $h \circ \nu:\overline{X} \rightarrow (Y,0)$. Set $(X',x'):=(\overline{X},x')$. Set $\nu': (X',x') \rightarrow (X,0)$ and $h':=h \circ \nu'$.  Note that $X'\setminus\{x'\}$ is smooth because $X'$ is normal. Therefore, $h': X' \rightarrow Y$ is a smoothing of $X'_0$, i.e.\ there exists a Zariski open subset $U$ in $Y$ such that $X'_y$ is smooth for $y \in U$. By \cite[Sct.\ 3.3, pg.\ 608]{T77} the germ $(X'_0,x')$ is reduced. 
 
Suppose $(X',x') \subset (\mathbb{C}^{N+1},x')=(\mathbb{C}^N,x') \times (Y,0)$ such that the restriction of the projection $\mathbb{C}^{N+1} \rightarrow \mathbb{C}^{n+1}$ to $X'$ is the finite map $X' \rightarrow X$. Denote by $\tilde{F}:=F \circ \nu'$ the pullback of $F$ to $X'$. Because $F$ is general, $X_y \not \subset \mathrm{Ker}(F)$; so, $\tilde{F}$ does not vanish identically on any irreducible component of $X'_y$ for $y$ close to $0$. As in Sct.\ \ref{Sct-RH}, choose a small open ball $\mathring{B}_{x'}:=\mathring{B}(x', \epsilon) \subset (\mathbb{C}^N,x')$ and a closed disk $\mathbb{D}=\mathbb{D}(0,\delta)$ with $\delta$ small with respect to $\epsilon$, and replace $Y$ by a small enough open disk centered at $0$, so that $X'$ is a closed analytic subset of $\mathring{B}_{x'} \times Y$, the restriction of $\tilde{F}$ to $X'_y \cap \mathring{B}_{x'} \cap \tilde{F}^{-1}(\mathring{\mathbb{D}})$ is a finite morphism for every $y \in Y$ (see \cite[Prop.\ 1), p.\ 63]{GrR84}), and the analogues of conditions \rm{(iii)}--\rm{(v)} in the proof of Thm.\ \ref{ZU of Milnor} hold for $h'$ and $\tilde{F}$ (this is possible because $(X'_0,x')$ is reduced with $(X'_0)_{\mathrm{sing}}=\{x'\}$, and $\tilde{F}|_{X'_0}$ has no critical points on $X'_0 \setminus \{x'\}$, as $\nu'$ is an isomorphism over $X_0 \setminus \{0\}$ and $F|_{X_0}$ has no critical points there). The degree $d$ of $\tilde{F}$ restricted to $X'_y \cap \mathring{B}_{x'} \cap \tilde{F}^{-1}(\mathring{\mathbb{D}})$ does not depend on $y$, because $\mathcal{O}_{X',x'}/(\tilde{F})$ is a finite $\mathcal{O}_{Y,0}$-module which is Cohen--Macaulay of dimension one, hence flat over $\mathcal{O}_{Y,0}$. Let $y \in U$ and let $p \in X'_y \cap \mathring{B}_{x'} \cap \tilde{F}^{-1}(\mathring{\mathbb{D}})$ be a critical point of $\tilde{F}|_{X'_y}$. If $\nu'(p) \neq 0$, then $\nu'(p) \in (X_y)_{\mathrm{sm}}$, the map $\nu'$ is a local isomorphism at $p$, and $\nu'(p)$ is a critical point of $F|_{X_y}$ which is not in $\Sigma_{X/Y}$, contradicting $\Sigma_{X/Y}(F)=\Sigma_{X/Y}$. Hence $\tilde{F}(p)=F(0)=0$, so that all critical points of $\tilde{F}|_{X'_y}$ lie in $Z:=\tilde{F}^{-1}(0) \cap X'_y \cap \mathring{B}_{x'}$. Denote by $e_p$ the ramification index of $\tilde{F}|_{X'_y}$ at $p \in Z$. Since $\tilde{F}$ is finite and flat of degree $d$, we have $\sum_{p \in Z}e_p=d$, and the Riemann--Hurwitz formula (\ref{eq:RH_partial}), applied to the connected components of the smooth curve $X'_y \cap \mathring{B}_{x'} \cap \tilde{F}^{-1}(\mathring{\mathbb{D}})$ (its proof uses only the finiteness of $\tilde{F}$), gives
$$\chi\big(X'_y \cap \mathring{B}_{x'} \cap \tilde{F}^{-1}(\mathring{\mathbb{D}})\big)=d-\sum_{p \in Z}(e_p-1)=\#Z \geq 1.$$
On the other hand, by (\ref{ball-polydisc}), whose proof applies verbatim to $h'$ and $\tilde{F}$, the left-hand side equals $\chi(X'_y)$, the Euler characteristic of the closure of $X'_y$ in $B_{x'}$, and by Prop.\ \ref{Euler} applied to the smoothing $h'$ of the reduced curve $(X'_0,x')$ we have $1-\chi(X'_y) \geq 0$ for $y \in U$ close to $0$, with equality if and only if $(X'_0,x')$ is smooth. Therefore $(X'_0,x')$ is smooth (and $\#Z=1$). As $x' \in \nu^{-1}(0)$ was arbitrary, $\overline{X}_0$ is smooth at every point of $\nu^{-1}(0)$, and the first part of the theorem follows from the case treated above.
 
We are left with proving the second part of the theorem; here $\dim Y=k$ is arbitrary. By the first part, $\delta_0=\delta_y$ for all $y$ close enough to $0$, so $\overline{X}_y=\overline{X_y}$ is smooth for all such $y$ by \cite[Thm.\ 5.6]{CL06}; moreover $\overline{X}$ is Cohen--Macaulay (see the proof of Lem.\ \ref{comps. ineq.}), so $\overline{X} \rightarrow Y$ is flat and $\overline{X}$ is smooth. 

Set $r:=r_y$, which does not depend on $y$ by the first part, and $\nu_{y}^{-1}(0):=\{\bar{x}_y(1), \ldots, \bar{x}_{y}(r)\}$ for $y$ close enough to $0$. Denote by $\mathfrak{m}_y$ the maximal ideal of $\mathcal{O}_{X_y,0}$. The fiber over $y$ of the finite map $\nu^{-1}(Y)=\overline{X} \times_{X} Y \rightarrow Y$ is $\mathrm{Specan}\, \mathcal{O}_{\overline{X_y}}/\nu_{y}^{*}(\mathfrak{m}_y)$, whose length is $\sum_{i=1}^{r}o_i(y)$, where $o_i(y)$ is the order of $\nu_{y}^{*}(\mathfrak{m}_y)$ in the discrete valuation ring $\mathcal{O}_{\overline{X_y},\bar{x}_y(i)}$. By the proof of (\ref{ram-mult}) we have $\sum_{i=1}^{r}o_i(y)=m_y$, which is independent of $y$ by the first part. A finite map to a reduced space whose fibers have constant length is flat; hence $\nu^{-1}(Y)$ is flat over $Y$, and so every irreducible component of $\nu^{-1}(Y)$, and of $\nu^{-1}(Y)_{\mathrm{red}}$, dominates $Y$. The finite map $\nu^{-1}(Y)_{\mathrm{red}} \rightarrow Y$ has exactly $r$ points in every fiber, so its degree is $r$. Since $Y$ is normal and every irreducible component of $\nu^{-1}(Y)_{\mathrm{red}}$ dominates $Y$, the local degrees of this map at the $r$ points of a fiber are at least one and add up to $r$; hence the germ of $\nu^{-1}(Y)_{\mathrm{red}}$ at each $\bar{x}_0(i)$ is irreducible of degree one over $(Y,0)$, i.e.\ it maps bijectively, and therefore isomorphically, onto $(Y,0)$. Denote by $Y_i$ the germ of $\nu^{-1}(Y)$ at $\bar{x}_0(i)$, with its natural (in general non-reduced) structure, and number the points so that $\bar{x}_y(i) \in Y_i$; the fiber of $Y_i \rightarrow Y$ over $y$ has length $o_i(y)$. By upper semicontinuity of the fiber length of the finite map $Y_i \rightarrow Y$ we have $o_i(0) \geq o_i(y)$ for $y$ close to $0$, and since $\sum_{i=1}^{r}o_i(0)=m_0=m_y=\sum_{i=1}^{r}o_i(y)$, we get $o_i(y)=o_i(0)=:l_i$ for all $y$ close to $0$.

Thus all fibers of the finite map $Y_i \rightarrow Y$ have the same length $l_i$, so $Y_i$ is flat over $Y$, and by \cite[Cor.\ 18.17]{Eis95} $Y_i$ is Cohen--Macaulay. But $\codim(Y_i,\overline{X})=1$ and $\overline{X}$ is smooth, so by \cite[\href{https://stacks.math.columbia.edu/tag/0BXH}{Tag 0BXH}]{Stacks}, we obtain that
 $Y_i$ is a Cartier divisor.
Because $\overline{X} \rightarrow Y$ is smooth at $\bar{x}_0(i)$, there are local coordinates $(u,y_1, \ldots, y_k)$ on $(\overline{X},\bar{x}_0(i))$ such that $y_1, \ldots, y_k$ are coordinates on $(Y,0)$ and $u$ is a uniformizing parameter of $\mathcal{O}_{\overline{X_y},\bar{x}_y(i)}$ for each $y$. Let $G(u,y_1, \ldots, y_k)$ be a generator of the ideal of $Y_i$ in $\mathcal{O}_{\overline{X},\bar{x}_0(i)}$. For each $y$ close to $0$, the function $G(\cdot,y)$ generates the ideal $\nu_{y}^{*}(\mathfrak{m}_y)$ of $(Y_i)_y$ in $\mathcal{O}_{\overline{X_y},\bar{x}_y(i)}$, so it vanishes only at the $u$-coordinate $u_i(y)$ of $\bar{x}_y(i)$, to order $l_i$. By the Weierstrass preparation theorem $G=(u-u_i(y))^{l_i}g(u,y_1, \ldots, y_k)$ where $g$ does not vanish in a small enough neighborhood of $\bar{x}_0(i)$ in $\overline{X}$ and $u_i(y)$ depends analytically on $y$. After the change of coordinates $u \mapsto u-u_i(y)$ the ideal of $Y_i$ is generated by $u^{l_i}$. Therefore, $(Y_i,\bar{x}_0(i)) \cong ((Y_{i})_0, \bar{x}_0(i)) \times (Y,0)$, i.e.\ $\nu^{-1}(Y) \rightarrow Y$ is locally analytically trivial. 
\end{proof}
\textbf{Remark.} Unlike the case of families of plane curves, the assumption that $\Gamma_{h}^{k}(X)$ is empty does not imply non-splitting of the singular locus of $(X,0) \rightarrow (Y,0)$ as the following example demonstrates (\cite[Ex.\ 7.2.5]{BG80}). Consider the one-parameter family with parametrization $\psi: (\mathbb{C},0)\times (\mathbb{C},0) \rightarrow (\mathbb{C}^3,0) \times (\mathbb{C},0)$ given by $\psi(s,y):=(s^3-3ys, s^4-2ys^2,s^5-7ys^3+16y^2s)$. The generic fiber $X_y$ for $y \neq 0$ has two singular points $s = \pm \sqrt{y}$ each analytically isomorphic to the cusp $\mathbb{C}\{t^2,t^3\}$. One computes 
$$2\delta_0+m_0-r_0 = 2\delta_y+m_y-r_y=6$$
where $\delta_y$ is the sum of the delta invariants at the two singular points of $X_y$, $m_y$ is the sum of their multiplicities, and $r_y$ is the sum of the number of branches at the singular points of $X_y$. Thus by Thm.\ \ref{main mpt}, we get $\Gamma_{h}^{k}(X)=\emptyset$. Moreover, the Milnor numbers of $X_0$ and  $X_y$ ($y \neq 0$) are $4$. So, the singular locus may split even if the Milnor number remains constant and the  relative polar curve is empty. In particular, the hypothesis that the singular locus does not split cannot be dropped from Thm.\ \ref{equimult}.

Denote by $m(X,y)$ the multiplicity of $X$ at $y$. We record a result we will need later due to Hironaka and Schickhoff, whose proof can be found in \cite[Scts.\ 4 and 5]{Lip}; it is applied to the projection $(x,y) \mapsto y$, whose fibers are the $X_y$. 
\begin{theorem}[Hironaka--Schickhoff]\label{HSch}
 The function $y \mapsto m_y$ is constant across $(Y,0)$ if and only if $y \mapsto m(X,y)$ is constant across $(Y,0)$.
\end{theorem}
\subsection{Topological equisingularity}\label{top. equising.}
Suppose $h \colon (X,0) \rightarrow (Y,0)$ is a flat family such that 
$h$ factors through the composition of maps 
\begin{tikzcd}[column sep=small]
(X,0) \arrow[hookrightarrow]{r} & (\mathbb{C}^n,0) \times (Y,0) \ar[r,"pr_2"] & (Y,0)
\end{tikzcd}
with $(h^{-1}(0),0)=(X_0,0)$. Assume  $Y$ is smooth. Assume $(X_0,0)$ is contained in an open ball $\mathring{B_0}=\mathring{B}_0(0,\epsilon)\subset(\C^n,0)$. Identify $Y$ with an open $k$-ball. We may assume that $X$ is an analytic subset of $B_0 \times Y$. Assume that $B_0$ and $Y$ are sufficiently small (see conditions \rm{(iv)} and \rm{(v)} in the proof of Thm.\ \ref{ZU of Milnor}). 
 
Recall that $(X,0) \rightarrow (Y,0)$ is {\it topologically equisingular} if there exists a homeomorphism $\phi: (X,0) \rightarrow (X_0,0) \times (Y,0)$ such that $h=pr_2 \circ \phi$. 
In this subsection, we will show that topological equisingularity is characterized by the constancy of the Milnor number, generalizing a result of Buchweitz and Greuel (see the equivalences of (1)-(4) with (6) in \cite[Thm.\ 5.2.2]{BG80}). 
 
Our treatment allows for a more general setup than the one in \cite[Thm.\ 5.2.2]{BG80}: we do not need to assume the existence of a section from $Y$ to $X$ that parametrizes the singular points of the fibers $X_y$. The assumption on the 
intersection of the irreducible components of the fibers $X_y$ is necessary to have a homeomorphism between $X_y$ and the germ $(X_0,0)$. From the analysis in Prop.\ \ref{nonneg} \rm{(i)} we see that the assumption $y \mapsto \mu_y$, which yields $1=\chi(X_y)$, does not alone impose the structure of the intersection of the irreducible components needed for the existence of a topological trivialization. However, as follows from the proof of Thm.\ \ref{top. equis.} below, the intersection points $x_y$ form an irreducible component $S_c$ of the relative singular locus of $X \rightarrow Y$ such that $(S_c,0)$, identified with its reduced structure, is analytically isomorphic to $(Y,0)$. Indeed, the map $(S_c,0) \rightarrow (Y,0)$ is finite and bijective, hence bimeromorphic. But $(Y,0)$ is normal, so $(S_c,0) \cong (Y,0)$. Thus our hypotheses allow for the existence of analytically irreducible singularities of the fibers $X_y$ away from $x_y$. We begin with a lemma. 
 
\begin{lm}\label{comps. ineq.}
Let $l:(Z,0) \rightarrow (T,0)$ be a flat family of reduced curves with $(T,0)$ smooth. Assume that $T$ is a curve or $\overline{Z}_t=\overline{Z_t}$ for all $t \in (T,0)$. Denote the number of irreducible components of $Z$ and $Z_t$ by $R(Z)$ and $R(Z_t)$, respectively. Then $R(Z)=R(Z_t)$ and $R(Z_0) \geq R(Z_t)$ for $t \neq 0$ generic.
\end{lm}
\begin{proof}
Let $Z_1, \ldots, Z_{R(Z)}$ be the irreducible components of $Z$. Denote by $\nu:\overline{Z} \rightarrow Z$ the normalization morphism. Then $\overline{Z} =\bigsqcup_{j=1}^{R(Z)}\overline{Z_j}$ and so $\mathcal{O}_{\overline{Z}}=\bigoplus_{j=1}^{R(Z)}\mathcal{O}_{\overline{Z_j}}$.

Suppose $T$ is a smooth curve. We have $\mathcal{O}_{\overline{Z}_0}=\bigoplus_{j=1}^{R(Z)}\mathcal{O}_{(\overline{Z_j})_0}$. By \cite[Sct.\ 3.3, pg.\ 608]{T77} $\mathcal{O}_{\overline{Z}_0}$ is a subring of $\overline{\mathcal{O}_{Z_0}}$, which is reduced. Then $(\overline{Z_j})_0$ is reduced and $\overline{Z_j} \rightarrow T$  is flat for each $j$.
 
Suppose  $\overline{Z}_t=\overline{Z_t}$ for all $t \in (T,0)$ (we also say that $l$ is equinormalizable). View $\mathcal{O}_{\overline{Z}}/\mathcal{O}_Z$ as a finite $\mathcal{O}_{T,0}$-module. By \cite[Thm.\ 5.6]{CL06} we have $$\delta_0=\dim_{\mathbb{C}}(\mathcal{O}_{\overline{Z}_{0}}/\mathcal{O}_{Z_0})=\dim_{\mathbb{C}}(\mathcal{O}_{\overline{Z}_{t}}/\mathcal{O}_{Z_t})=\delta_t$$
because $l$ is equinormalizable. In particular, $\mathcal{O}_{\overline{Z}}/\mathcal{O}_Z$ is a locally free $\mathcal{O}_{T,0}$-module.
Let $\mathfrak{n}$ be the maximal ideal of $\mathcal{O}_{T,0}$ and let $\mathfrak{m}$ be the maximal ideal of $\mathcal{O}_{Z,0}$. Identify $\mathfrak{n}$ with its image $l^{\#}(\mathfrak{n})$ in $\mathfrak{m}$. Consider the exact sequence
of local cohomology of finitely generated $\mathcal{O}_Z$-modules
$$\mathrm{H}_{\mathfrak{n}}^{i}(\mathcal{O}_{Z}) \rightarrow \mathrm{H}_{\mathfrak{n}}^{i}(\mathcal{O}_{\overline{Z}})\rightarrow \mathrm{H}_{\mathfrak{n}}^{i}(\mathcal{O}_{\overline{Z}}/\mathcal{O}_Z).$$
The two extreme terms vanish for $i \leq \dim T-1$, because $\mathcal{O}_Z$ is Cohen--Macaulay (being flat over $T$ with Cohen--Macaulay fibers) and $\mathcal{O}_{\overline{Z}}/\mathcal{O}_Z$ is a locally free $\mathcal{O}_{T,0}$-module. Thus $\mathrm{depth}_{\mathfrak{n}}(\mathcal{O}_{\overline{Z}}) \geq \dim T$, so a regular system of parameters of $\mathcal{O}_{T,0}$ is an $\mathcal{O}_{\overline{Z}}$-regular sequence. Because $l$ is equinormalizable, $\mathcal{O}_{\overline{Z}}/\mathfrak{n}\mathcal{O}_{\overline{Z}}=\mathcal{O}_{\overline{Z}_0}=\mathcal{O}_{\overline{Z_{0}}}$. In particular, $\mathcal{O}_{\overline{Z}_0}$ is reduced of dimension one, hence of depth one. Hence   $\mathrm{depth}_{\mathfrak{m}}(\mathcal{O}_{\overline{Z}}) = \dim T+1$ and thus $\overline{Z}$ is Cohen--Macaulay. Therefore, the morphism is $\overline{Z} \rightarrow T$ is flat and so are the morphisms $\overline{Z_j} \rightarrow T$. 
 
In both cases we obtained that $(\overline{Z_j})_0$ is reduced and $\overline{Z_j} \rightarrow T$  is flat for each $j$. Therefore, for $t\neq 0$ the fiber $(\overline{Z_j})_t$ is connected (the germ $Z_j$ is irreducible, so $\overline{Z_j}$ is a germ at a single point of $\nu^{-1}(0)$ and $(\overline{Z_j})_0$ is connected; now apply the argument for the connectedness of the generic fibers in the proof of Thm.\ \ref{ZU of Milnor} to $\overline{Z_j} \rightarrow T$) and normal because $\overline{Z}_t = \overline{Z_t}$. Thus $(\overline{Z_j})_t$ is irreducible for $t \neq 0$. So, $\overline{Z}_t = \bigsqcup_{j=1}^{R(Z)}(\overline{Z_j})_t$ is a decomposition of $\overline{Z}_t$ into irreducible components. Because the normalization morphism $\overline{Z_t} \rightarrow Z_t$ induces a bijection between the irreducible components of the source and the target, we conclude that $R(Z)=R(Z_t)$. Furthermore, under our hypothesis $\overline{Z}_0 \rightarrow Z_0$ is a finite birational map with $\overline{Z}_0 = \bigsqcup_{j=1}^{R(Z)}(\overline{Z_j})_0$, and a finite birational map induces a bijection between the irreducible components of the source and the target. Therefore, $R(Z_0) \geq R(Z)$.
\end{proof}
 
\vspace{.2cm}
\begin{center}
{\textbf{Proof of Theorem \ref{top. equis.}}}
\end{center}
\vspace{.2cm}
 
Suppose $y \mapsto \mu_y$ is constant. By Cor.\ \ref{BG ineq} $\delta_0 \leq \delta_y$ for $y$ in a Zariski open subset of $Y$, while $\delta_0 \geq \delta_y$ by \cite[Sct.\ 3.3]{T77} applied along a smooth curve through $0$; so $\delta_0 = \delta_y$ for $y$ in a Zariski open subset of $Y$. Therefore, by Prop.\ \ref{delta trivial} $\delta_0=\delta_y$ for each $y \in Y$ after possibly shrinking $Y$.
 
By assumption the irreducible components of each $X_y$ meet at a single point of intersection which we will denote by $x_y$ (when $X_y$ is irreducible there is no such point, and everything below involving $x_y$, $S_c$ and $\tilde{S_j}$ is to be omitted). As usual, denote by $r(X_y,x)$ the number of branches of the germ $(X_y,x)$. Set $U:=U_\chi \cap U_c$, with the notation of Thm.\ \ref{ZU of Milnor} and its proof. For $y \in U$ the curve $X_y$ is connected and $\mu_0=\mu_y$ gives $\chi(X_y)=1$ by (\ref{Gen-BG}); so by Prop.\ \ref{nonneg} \rm{(i)} each irreducible component of $X_y$ is unibranch at each of its points. So, the number of irreducible components $R(X_y)$ of $X_y$ is equal to $r(X_y,x_y)$, and the only point of $X_y$ with more than one branch is $x_y$. Because $\mu_0=\mu_y$ and $\delta_0 = \delta_y$ we have $r(X_0,0)=r(X_y,x_y)$ for $y \in U$.
We will show that for  all $y \in Y$ each irreducible component of $X_y$ is unibranch at each of its points, $R(X_y)=r(X_y,x_y)$ and furthermore, $y \mapsto R(X_y)$ is constant. 
 
We proceed as in the proofs of Prop.\ \ref{Zariski-ups-cid} and Thm.\ \ref{ZU of Milnor}.
Fix $y' \in Y$. Because $Y$ is smooth, we can find a smooth curve $Y'$ passing through $y'$ such that $Y' \setminus\{y'\}$ lies in $U$. Set $h':X':=X \times_{Y}Y' \rightarrow Y'$. Let $x_1, \ldots, x_w$ be the singular points of $X'_{y'}=X_{y'}$. Assume each $(X'_{y'},x_i)$ is contained in a sufficiently small open ball $\mathring{B_i}$ in $\mathbb{C}^n$ centered at $x_i$ such that $B_i \cap B_j = \emptyset$ for each $i \neq j$ with $i,j \in \{1, \ldots, w\}$. Let $X_i' \rightarrow Y'$ with $X_{i}' \subset \mathring{B_i} \times (Y',y')$ be the induced embedded deformation of $(X'_{y'},x_i)$ from $X'\rightarrow Y'$. As in the proof of Thm.\ \ref{ZU of Milnor}, shrink $(Y',y')$ so that $(X_y)_{\mathrm{sing}} \subset \bigcup_{i=1}^w \mathring{B_i} \times \{y\}$ for $y \in Y'$. Let $y \in Y'$ with $y' \neq y$ be sufficiently small with respect to all families $X_i' \rightarrow Y'$ so that (\ref{Gen-BG}) is valid for each one of them. By upper semicontinuity and our assumption that $y \mapsto \mu_y$ is constant, for $i=1, \ldots, w$ we have 
$$
\mu(X_{y'},x_i) \geq \mu((X_i')_y) \ \text{and} \ \mu_{y'}=\sum_{i=1}^w\mu(X_{y'},x_i)=\mu_y=\sum_{i=1}^w\mu((X_i')_y).
$$
Therefore, $\mu(X_{y'},x_i)= \mu((X_i')_y)$.
Similarly, because $\delta(X_{y'})=\delta(X_y)$, and $\delta(X_{y'},x_i) \geq \delta((X_i')_y)$, we have $\delta(X_{y'},x_i)= \delta((X_i')_y)$. Thus, for each $i=1, \ldots, w$ we have $$r(X_{y'},x_i)-1=\sum_{x \in ((X_i')_y)_{\mathrm{sing}}}(r((X_i')_y,x)-1).$$
Since $y \in U$, the only point of $X_y$ with more than one branch is $x_y$ (if it exists). If $R(X_y) \geq 2$, then $R(X_{y'}) \geq R(X')=R(X_y) \geq 2$ by Lem.\ \ref{comps. ineq.} applied to $h'$, so $r(X_{y'},x_{y'}) \geq 2$, where $x_{y'}$ is the intersection point of the irreducible components of $X_{y'}$, and the above identity for the ball $B_i$ containing $x_{y'}$ forces $x_y$ to lie in that ball. Hence, if $x_i \neq x_{y'}$, then $r((X_i')_y,x)=1$ for each $x \in ((X_i')_y)_{\mathrm{sing}}$, and so $r(X_{y'},x_i)=1$. If $x_i=x_{y'}$, we obtain from the above identity that $r(X_{y'},x_{y'})=r(X_y,x_y)$. By Lem.\ \ref{comps. ineq.} $r(X_{y},x_{y})=R(X_y)$ is equal to  the number of irreducible components  $R(X')$ of $X'$.  Thus $R(X')=r(X_{y'},x_{y'})$. But $r(X_{y'},x_{y'}) \geq R(X'_{y'})$ with an equality if and only if each irreducible component of $X_{y'}$ is unibranch at $x_{y'}$. By Lem.\ \ref{comps. ineq.} $R(X'_{y'}) \geq R(X')$. So, $R(X'_{y'})=R(X_{y'})=r(X_{y'},x_{y'})$. Again, by Lem.\ \ref{comps. ineq.} we have $R(X_y)=R(X)$ because $\overline{X}_y=\overline{X_y}$ for all $y \in (Y,0)$. Therefore, $r(X_y,x_y)=R(X_y)=R(X)$ for all $y \in (Y,0)$ and the irreducible components of each $X_y$ are unibranch at each of their points (at $x_y$ by the equality just proved, and at the other singular points $x_i$ by $r(X_{y'},x_i)=1$).


Set $r:=R(X)$ and $\nu_{0}^{-1}(0):=\{\bar{x}_0(1), \ldots, \bar{x}_{0}(r)\}$. Let $X_1, \ldots, X_{r}$ be the irreducible components of $X$, numbered so that $\bar{x}_0(j) \in \overline{X_j}$. Set $S_c:=\bigcup_{i<j}X_i \cap X_j$ with its reduced structure; it is a closed analytic subset of $X$, and we show below that it consists of the intersection points $x_y$, $y \in Y$.

 
Next we proceed as in the proof of  \cite[Thm.\ 5.2.2]{BG80}.
Recall that $\nu:\overline{X} \rightarrow X$ is the normalization morphism, and that $\overline{X} =\bigsqcup_{j=1}^{r}\overline{X_j}$.
Because $\delta_0=\delta_y$ for all $y \in Y$, by \cite[Thm.\ 5.6]{CL06} $\overline{X}_y=\overline{X_y}$ is smooth for all $y \in Y$; hence $\overline{X}$ is Cohen--Macaulay (see the proof of Lem.\ \ref{comps. ineq.}), $\overline{X} \rightarrow Y$ is flat, and $\overline{X}$ is smooth, and so are the $\overline{X_j}$. Since $\overline{X}_y=\overline{X_y}=\bigsqcup_{j}(\overline{X_j})_y$, the fiber $(\overline{X_j})_y$ of $\overline{X_j}$ is the normalization of the irreducible component $(X_j)_y$ of $X_y$, taken with its reduced structure. Moreover, $\overline{X_j} \rightarrow X_j$ is a homeomorphism for each $j$: it is finite and surjective, and it is injective because $X_j$ is unibranch at each of its points as $(\overline{X_j})_y \cap \nu_y^{-1}(x)$ consists of a single point for $x \in (X_j)_y$, the component $(X_j)_y$ being unibranch at $x$. In particular, $(X_i)_y \cap (X_j)_y=\{x_y\}$ for $i \neq j$, so $S_c \cap X_y=\{x_y\}$ for every $y$. Thus $S_c \rightarrow Y$ is finite and bijective. For $j=1, \ldots, r$ set $\tilde{S_j}:=(\nu^{-1}(S_c) \cap \overline{X_j})_{\mathrm{red}}$. Since $(X_j)_y$ is unibranch at $x_y$, the finite map $\tilde{S_j} \rightarrow Y$ has exactly one point in each fiber. Hence $\tilde{S_j}$ is irreducible. Indeed, let $C$ be an irreducible component of $\tilde{S_j}$ dominating $Y$. Such a $C$ exists because $\tilde{S_j}$ maps onto $Y$ while the images of the components not dominating $Y$ are proper closed subsets of $Y$. The image of $C$ is closed and contains the complement of these proper subsets, so $C$ maps onto $Y$; any further component of $\tilde{S_j}$ would then meet some fiber of $\tilde{S_j} \rightarrow Y$ in a second point.
Thus $\tilde{S_j}$ is reduced and irreducible and maps bijectively, hence isomorphically, onto the normal germ $(Y,0)$. In particular, $(\tilde{S_j},\bar{x}_0(j))$ is smooth, and $S_c=\nu(\tilde{S_1})$ is irreducible with $(S_c,0) \cong (Y,0)$.

Because $\partial B_0 \times \{y\}$ intersects $X_y$ transversally, the map from the closure in the Euclidean topology of $\overline{X_j}$ to $Y$ is a submersion. By Ehresmann's fibration theorem 
we can take a lift of the coordinate vector fields on $(Y,0)$ such that the lifts are tangent to $(\tilde{S_j},0)$.  These lifts determine a local flow, which yields a diffeomorphism
$\overline{X_j} \cong (\overline{X_j})_0 \times Y$ such that the point $(\tilde{S_j})_y$ is sent to the point $((\tilde{S_j})_0,y)$.
Composing it with the homeomorphisms $\overline{X_j} \rightarrow X_j$ and $(\overline{X_j})_0 \rightarrow (X_j)_0$ induced by $\nu$, we obtain a trivializing homeomorphism $\phi_j: X_j \rightarrow (X_j)_0 \times Y$ such that $x_y$ is sent to the point $(x_0,y)$. These homeomorphisms agree on the intersections of the components $X_j$, which are contained in $S_c$ because $(X_j)_y$ is a single irreducible component of $X_y$ for every $y$ (as $R(X_y)=R(X)$); this gives rise to a homeomorphism $\phi: X \rightarrow X_0 \times Y$ such that $h=pr_2 \circ \phi$.
 
To prove the converse, observe that if $X \rightarrow Y$ is topologically equisingular, then each $X_y$ (after shrinking the representatives) is homeomorphic to $X_0$, which is contractible. Thus $\chi(X_y)=1$ for each $y$. Therefore, by Thm.\ \ref{ZU of Milnor} $\mu_0=\mu_y$ for $y \in U_{\chi}$. But $y \mapsto \mu_y$ is Zariski upper semicontinuous. Thus $\mu_0 \geq \mu_{y'} \geq \mu_y$ for each $y' \in Y$ and $y \in U_{\chi}$. Therefore, $y \mapsto \mu_y$ is constant across $Y$.
\hfill\qedsymbol 
\vspace{.1cm}
 
\textbf{Remark.} Thm.\ \ref{top. equis.} applied to the example considered in Sct.\ \ref{simultaneous} of a family of branches with a constant Milnor number shows that the family is topologically equisingular despite the fact that the relative singular locus is reducible. 
 
 
\subsection{Whitney equisingularity} 
Preserve the setup from Sct.\ \ref{simultaneous} (the family setup with a section). Define the {\it conormal space} $C(X)$ of $X$ (sometimes referred as the absolute conormal space) to be the closure in $X \times \check{\mathbb{P}}^{n+k-1}$ of the set of pairs $(x,H)$ where $x$ is a smooth point of $X$ and $H$ is a (tangent) hyperplane in $\mathbb{C}^{n+k}$ at $x$ containing the tangent space $T_{x}X$. Note that $C(X) = \mathrm{Projan}(\mathcal{R}(J(X)))$ where $\mathcal{R}(J(X))$ is the Rees algebra of the Jacobian module $J(X)$.
Denote by $c: C(X) \rightarrow X$ the structure morphism
and by  $\beta$ the composition of  maps  $C(X) \hookrightarrow X \times \check{\mathbb{P}}^{n+k-1} \rightarrow  \check{\mathbb{P}}^{n+k-1}$. Note that $\dim C(X)=n+k-1$.
Similarly to the way we defined the relative polar varieties in Sct.\ \ref{polar}, in the absolute setting we define the {\it absolute polar variety} $\Gamma^{l}(X)$ of dimension $l$  as $\Gamma^{l}(X):=c(\beta^{-1}(H))$ where $H$ is  a general linear space in $\check{\mathbb{P}}^{n+k-1}$ of codimension $n+k-1-l$.
 
Note that $\Gamma^{k+1}(X)=X$ where $k+1= \dim X$.
Denote by $m(\Gamma^{k}(X),y)$ the multiplicity of $\Gamma^{k}(X)$ at $y$. Denote by $D_Y$ the exceptional divisor of $\mathrm{Bl}_{c^{-1}Y}C(X)$. Teissier \cite{Teissier} described the compatibility condition of limits of secants and tangent hyperplanes known as Whitney condition B (see the definition of Whitney equisingularity in the Introduction) geometrically as an equidimensionality condition of the fibers of $D_Y \rightarrow Y$. In turn, he showed that this can be controlled by the multiplicities of the polar varieties. Concretely,  Teissier \cite{Teissier} showed that the pair $(X_{\mathrm{sm}},Y)$
satisfies the Whitney conditions at $0$ if and only if $y \mapsto  m(\Gamma^{l}(X),y)$ is constant in a neighborhood of $0$ for $l=k, \ldots, \dim X$, where $k=\dim Y$. As $\Gamma^k(X) \cap Y$ is a proper analytic subset of $Y$ (for general $H$ we have $\dim(\beta^{-1}(H) \cap c^{-1}(Y)) \leq k-1$), $y \mapsto  m(\Gamma^{k}(X),y)$ is constant in a neighborhood of $0 \in Y$ if and only if $\Gamma^k(X) = \emptyset$. Thus,  Teissier's result specialized to families of curves reads as follows. 
 
\begin{theorem}[Teissier]\label{absolute} The pair $(X_{\mathrm{sm}},Y)$ satisfies the Whitney conditions at $0$ if and only if $\Gamma^k(X) = \emptyset$ and $y \mapsto m(X,y)$ is constant for all $y$ close to $0$. 
\end{theorem}
 
Adopt the setup of Sct.\ \ref{Conormal spaces}. Let $X$ be defined by the vanishing of some analytic functions $f_1, \ldots, f_p$ on a Euclidean neighborhood of $0$ in $\mathbb{C}^{n+k}$. Let $(y_1, \ldots, y_k)$ be coordinates on $(Y,0)$. For each $j=1, \ldots, k$ let $g_j$ be the column vector
\begin{equation}\label{section}
g_j :=
\begin{pmatrix}
\partial f_{1} / \partial y_{j}\\
\vdots \\
\partial f_{p} / \partial y_{j}\\
\end{pmatrix}.
\end{equation}
Denote the ideal of $Y$ in $\mathcal{O}_{X,0}$ by $m_{Y}$ and denote by $J_h(X) \subset \mathcal{O}_{X,0}^p$ the relative Jacobian module. View each $g_j$ as an element in $\mathcal{O}_{X,0}^p$. The following result \cite[Thm.\ 2.5]{Gaf1} characterizes
the Whitney conditions by the integral dependence of the $g_j$ on the module $m_{Y}J_{h}(X)$. 
 
\begin{theorem}[Gaffney]\label{Wht. int. dep.}
 The pair $(X_{\mathrm{sm}},Y)$ satisfies the Whitney conditions at $0$ if and only if $g_j$ is integrally dependent on $m_Y J_{h}(X) $ for each $j=1, \ldots, k$.
\end{theorem}
 
The two results above, due to Teissier and Gaffney, give an important connection between the absolute and relative conormal spaces in the presence of Whitney equisingularity. If $(X_{\mathrm{sm}},Y)$ satisfies the Whitney conditions at $0$ there is a finite map $C(X) \rightarrow C_{h}(X)$ because by Thm.\ \ref{Wht. int. dep.} $\mathcal{R}(J(X))$ is integral over $\mathcal{R}(J_h(X))$. In particular, by Thm.\ \ref{absolute} $\Gamma^{k}(X) = \emptyset$ implies that $\Gamma^{k}_h(X) = \emptyset$. We prove below that conversely $\Gamma^{k}_h(X) = \emptyset$ implies that $\Gamma^{k}(X) = \emptyset$.

\vspace{.2cm}
\begin{center}
{\textbf{Proof of Theorem \ref{Whitney}}}
\end{center}
\vspace{.2cm}
 
$\rm{(i)} \Leftrightarrow \rm{(ii)}$. Assume that $y \mapsto e(\mathrm{Jac}(X_y,0)) - \mathrm{cid}(X_y,0)$ is constant across $(Y,0)$. By Prop.\ \ref{Zariski-ups-cid} and the remark after Def.\ \ref{def:ZU}, for every $y$ close enough to $0$ we have $$e(\mathrm{Jac}(X_0)) - \mathrm{cid}(X_0) \geq  e(\mathrm{Jac}(X_y))-\mathrm{cid}(X_y)=\sum_{x \in (X_y)_{\mathrm{sing}}}\big(e(\mathrm{Jac}(X_y,x))-\mathrm{cid}(X_y,x)\big).$$
Thus $$\sum_{x \in (X_y)_{\mathrm{sing}}, x \neq 0}\big(e(\mathrm{Jac}(X_y,x))-\mathrm{cid}(X_y,x)\big) \leq 0.$$
But by \cite[Thm.\ 3.9]{BR24} $e(\mathrm{Jac}(X_y,x))-\mathrm{cid}(X_y,x) \geq 0$ for each $x \in X_y$  with equality if and only if  $(X_y,x)$ is smooth. Thus for each $y \in Y$ the only singular point of $X_y$ is $0$. Now assume that this is the case. By Thm.\ \ref{main mpt}  $y \mapsto e(\mathrm{Jac}(X_y,0)) - \mathrm{cid}(X_y,0)$ constant is equivalent to $\Gamma_{h}^k(X)=\emptyset$ which in turn is equivalent to asking that $c_{h}^{-1}(Y) \rightarrow Y$ has equidimensional fibers. Also, the constancy of $y \mapsto e(\mathrm{Jac}(X_y,0)) - \mathrm{cid}(X_y,0)$ is equivalent to $b_{h}^{-1}(Y) \rightarrow Y$ having finite fibers by Prop. \ref{vertical} (ii).
 
$\rm{(ii)} \Leftrightarrow \rm{(iii)}$. Assume that for each $y \in Y$ the only singular point of $X_y$ is $0$ and that $\Gamma_{h}^k(X)$ is empty. By Thm.\ \ref{equimult} $y \mapsto m_y$ is constant which is equivalent to $y \mapsto m(X,y)$ constant by Thm.\ \ref{HSch}. Thus by Thm.\ \ref{absolute} it is enough to show that $\Gamma^k(X) = \emptyset$. Adopt the setup of Scts.\ \ref{Conormal spaces} and \ref{polar}. Consider the pair of modules $J_{h}(X) \subset J(X)\subset \mathcal{O}_{X}^p$ and the corresponding integral closures $\overline{J_{h}(X)} \subset \overline{J(X)}$ of  $J_{h}(X)$ and  $J(X)$ in $\mathcal{O}_{X}^p$. Set $V:=\mathrm{Supp}_{X}(\overline{J(X)}/\overline{J_{h}(X)})$. Because the union of the singular points of $X_y$ is $Y$ we have $V \subset Y$: at a point where $h$ is smooth, $J_h(X)$ is a free direct summand of $\mathcal{O}_X^p$ of the same rank as $J(X)$, so the two modules coincide there. By a result of Hironaka \cite{Hir76} $(X-Y,Y)$ satisfies the Whitney conditions at all points $y$ in a Zariski open subset of $Y$. Thus $V$ is a proper closed subset of $Y$.
Because $\Gamma_{h}^k(X)= \emptyset$ as a germ at $0$, we have $\dim c_h^{-1}(y)=n-2$ for all $y \in Y$ close to $0$ (if $\dim c_h^{-1}(y)=n-1$, then $c_h^{-1}(y)=\{y\} \times \check{\mathbb{P}}^{n-1}$ and $y \in \Gamma_h^k(X)$ for every $H$). Thus
$$\dim c_h^{-1}(V) \leq n-2+\dim V \leq n-2 + \dim Y-1=n+k-3 =\dim C_h(X)-2.$$
Therefore, by \cite[Cor.\ 10.7]{KT-Al} $V$ is empty. So $J_h(X)$ and $J(X)$ have the same integral closure which yields a finite surjective morphism $C(X) \rightarrow C_h(X)$. Thus $\dim c^{-1}(0)=n-2$ where $c: C(X) \rightarrow X$ is the structure morphism. This implies that $\Gamma^k(X) = \emptyset$. The converse follows from the explanation provided after Thm.\ \ref{Wht. int. dep.}.
 
$\rm{(ii)} \Rightarrow \rm{(iv)}$. This is the content of Thm.\ \ref{equimult}.
 
$\rm{(iv)} \Rightarrow \rm{(i)}$. Because $\overline{X}_y=\overline{X_y}$ for each $y \in Y$ we have $y \mapsto\delta_y$ is constant by a result of Teissier, Raynaud, Chiang-Hsieh and Lipman (\cite[Sct.\ 3.3]{T77} and \cite[Thm.\ 5.6]{CL06}). The constancy of $y \mapsto r_y$ follows from the topological triviality of  $\nu^{-1}(Y) \rightarrow Y$ where $\nu: \overline{X} \rightarrow X$ is the normalization. Thus 
by  \cite[Thm.\ 3.9]{BR24} it remains to show that $y \mapsto m_y$ is constant.
Denote the irreducible components of $\nu^{-1}(Y)=\overline{X} \times_X Y$ by $Y_1, \ldots, Y_r$, with their natural structure, and for each $y \in Y$ set $\nu_{y}^{-1}(0)=\{\bar{x}_y(1), \ldots, \bar{x}_{y}(r)\}$ with $\bar{x}_y(i) \in Y_i$. As in the proof of the second part of Thm.\ \ref{equimult}, $m_y=\sum_{i=1}^{r}o_i(y)$, where $o_i(y)$ is the length of the fiber of $Y_i \rightarrow Y$ over $y$, i.e.\ the order of $\nu_{y}^{*}(\mathfrak{m}_y)$ in $\mathcal{O}_{\overline{X_y},\bar{x}_y(i)}$, with $\mathfrak{m}_y$ the maximal ideal of $\mathcal{O}_{X_y,0}$. By assumption, $(Y_i,\bar{x}_0(i)) \cong ((Y_{i})_0,\bar{x}_0(i)) \times (Y,0)$. Thus all fibers of $Y_i \rightarrow Y$ have the same length $o_i(0)$, which yields $m_y=\sum_{i=1}^{r}o_i(0)=m_0$ for all $y$ close to $0$.
 
\begin{proposition}\label{jacobian mult.}
If $y \mapsto e(\mathrm{Jac}(X_y,0))$ is constant across $(Y,0)$,  then for each $y \in Y$ the only singular point of $X_y$ is $0$ and the family $h: (X,0) \rightarrow (Y,0)$ admits strong simultaneous resolution. 
\end{proposition}
\begin{proof} Let $g$ be a general $\mathbb{C}$-linear combination of generators of the relative Jacobian ideal $\mathrm{Jac}_h(X)$, and denote by $g_y$ its image in $\mathcal{O}_{X_y}$; then $g_0$ generates a reduction of $\mathrm{Jac}(X_0,0)$ and $g_y \in \mathrm{Jac}(X_y,x)$ for every $x \in X_y$. For all $y$ close enough to $0$ we have \begin{align*}e(\mathrm{Jac}(X_0,0))=\dim_{\mathbb{C}}\mathcal{O}_{X_0,0}/(g_0) &\geq \sum_{x \in (X_y)_{\mathrm{sing}}}\dim_{\mathbb{C}}\mathcal{O}_{X_y,x}/(g_y)\\ &\geq \sum_{x \in (X_y)_{\mathrm{sing}}}e(\mathrm{Jac}(X_y,x))=:e(\mathrm{Jac}(X_y)),\end{align*} by the upper semicontinuity of the fiber length of the finite $\mathcal{O}_Y$-module $\mathcal{O}_X/(g)$ and because $\dim_{\mathbb{C}}\mathcal{O}/(g) \geq e(J)$ for $g \in J$ in a one-dimensional Cohen--Macaulay local ring $\mathcal{O}$.

Thus $e(\mathrm{Jac}(X_y,x))=0$ for each $x \in X_y$ with $x \neq 0$ showing that $X_y$ is smooth away from $0 \in X_y$. This shows that for each $y \in Y$ the only singular point of $X_y$ is $0$.
 
By Prop.\ \ref{Zariski-ups-cid} and the remark after Def.\ \ref{def:ZU} we have $\mathrm{cid}(X_0) \geq \mathrm{cid}(X_y)$ for $y$ close to $0$. Because $e(\mathrm{Jac}(X_0,0))=e(\mathrm{Jac}(X_y,0))$, by Thm.\ \ref{main mpt} we have $\Gamma^{k}_h(X) = \emptyset$. Thus by Thm.\ \ref{Whitney} we conclude that  $h: (X,0) \rightarrow (Y,0)$ admits strong simultaneous resolution.
\end{proof}
For one-parameter families of curves the fact that the constancy of $y \mapsto e(\mathrm{Jac}(X_y,0))$ implies $\Gamma_{h}^k(X)= \emptyset$ was discovered by Kleiman, Ulrich and Validashti in connection with their work on the epsilon multiplicity of a module. The converse to Prop.\ \ref{jacobian mult.} is not true in general as demonstrated by the following example taken from \cite[Sct.\ 7]{BG80}.
 
\begin{example}
\rm{Consider the one-parameter family of irreducible curves $(X,0) \rightarrow (D,0)$ given parametrically by  $\mathbb{C} \{u^4,u^7+tu^6,u^9,u^{10}\}$, where $t$ is the uniformizing parameter of $\mathcal{O}_D$. Using \textsc{Singular} we compute that the base space of the miniversal deformation of $X_0$ is irreducible and that $X_0$ is smoothable. Thus $X_t$ is smoothable for each $t$.
We compute that for each $t$ the delta invariant of $X_t$ is $5$ and so $\mu(X_t,0)=10$ (the value $\mu(X_t,0) = 12$ stated in \cite[Sct.\ 7]{BG80} should be $10$). Thus, for each $t \in D$ we have $$e(\mathrm{Jac}(X_t,0))-\mathrm{cid}(X_t,0)=\mu(X_t,0)+m(X_t,0)-1=10+4-1 = 13.$$ Since the parametrization is injective and immersive away from $u=0$, the origin is the only singular point of $X_t$, and so the family is Whitney equisingular by Thm.\ \ref{Whitney}. However, a computation with \textsc{Singular} reveals that $e(\mathrm{Jac}(X_0))=21$ and $\mathrm{cid}(X_0,0)=8$, whereas $e(\mathrm{Jac}(X_t,0))=19$ and $\mathrm{cid}(X_t,0)=6$ for $t \neq 0$. This computation shows that, outside the case of families of complete intersection curves, $e(\mathrm{Jac}(X_t,0))$ is not a Whitney equisingularity invariant in general. So the presence of the correcting term $\mathrm{cid}(X_t,0)$ to $e(\mathrm{Jac}(X_t,0))$ is necessary for obtaining numerical control of  Whitney equisingularity.}
\end{example}
\textbf{Remark.} Below we give a sketch of a topological proof of the fact that Whitney equisingularity for $(X,0) \rightarrow (Y,0)$ is equivalent to the constancy of $y \mapsto (\mu(X_y,0),m(X_y,0))$ across $(Y,0)$ using
Thm.\ \ref{top. equis.} together with results of L\^e, Teissier, Buchweitz and Greuel.
 
We say $(X,0) \rightarrow (Y,0)$ is topologically equisingular via an ambient homeomorphism if there exists a homeomorphism $g \colon (\mathbb{C}^{n+k},0) \rightarrow (\mathbb{C}^{n+k},0)$ whose restriction to $X$ induces the homeomorphism $(X,0) \cong (X_{0} \times Y,0)$. By  \cite[Scts.\ 2.2 and 2.3]{LT2}  $(X,0) \rightarrow (Y,0)$ is Whitney equisingular if it is topologically equisingular via an ambient homeomorphism and $y \mapsto m(X_y,0)$ is constant. 
 
Assume $y \mapsto(\mu(X_y,0),m(X_y,0))$ is constant across $(Y,0)$. By Thm.\ \ref{ZU of Milnor} \rm{(ii)} the total Milnor number satisfies $\mu_0 \geq \mu_y \geq \mu(X_y,0)$ for $y$ close to $0$, so the constancy of $y \mapsto \mu(X_y,0)$ forces $\mu(X_y,x)=0$ for every $x \in (X_y)_{\mathrm{sing}} \setminus \{0\}$; thus the origin is the only singular point of $X_y$ for each $y \in Y$. By Thm. \ref{top. equis.} $(X,0) \rightarrow (Y,0)$ is topologically equisingular, which by \cite[Thm.\ 5.2.2, $(6) \Rightarrow (5)$]{BG80} and by \cite[Conj.\ 1]{Gre17} and the remark that follows it (this step is conditional on that conjecture) implies that
 $(X,0) \rightarrow (Y,0)$ is topologically equisingular via an ambient homeomorphism and thus by the results of L\^e and Teissier, the family is Whitney equisingular. 
 
Suppose $(X,0) \rightarrow (Y,0)$ is Whitney equisingular. Then it is topologically equisingular and $y \mapsto m(X_y,0)$ is constant.  By Thm.\ \ref{top. equis.} $y \mapsto \mu(X_y,0)$ is constant. 

\section{Appendix}

Let $h \colon (X,x_0) \rightarrow (Y,y_0)$ be a morphism with equidimensional fibers of positive dimension $d$ between equidimensional complex analytic varieties such that $\dim X=d+k$, $X$ and the fibers of $h$ are generically reduced, and $Y$ is smooth of dimension $k$. In Thm.\ \ref{MPT} below $X$ is assumed to be Cohen--Macaulay, and is therefore reduced. Let $N\subset F$ be coherent $\mathcal{O}_{X}$-modules such that $F$ is free of rank $e$. Suppose $T:=\mathrm{Supp}_X(F/N)$ is finite over $Y$. Let $\mathcal{R}(N)$ be the Rees algebra of $N$: it is the subalgebra of $\mathrm{Sym}(F)$ generated in degree one by the generators of $N$. 
 
For each $y \in Y$ denote by $F(y)$ the restriction of $F$ to $X_y$ and by $N(y)$ the image of $N$ in $F(y)$.  As before, define the Buchsbaum--Rim multiplicity $e(N(y),F(y))$ as the normalized leading coefficient of  $\dim_{\mathbb{C}}(F^l(y)/N^{l}(y))$ which is a polynomial of degree $r:=d+e-1$ for $l$ large enough, where $F^{l}(y)$ and $N^{l}(y)$ are the $l$th graded components of $\mathrm{Sym}(F(y))$ and the Rees algebra $\mathcal{R}(N(y))$, respectively. Note that $\mathrm{Projan}(\mathcal{R}(N)) \subset X \times \mathbb{P}^{g(N)-1}$ where $g(N)$ is the cardinality of a generating set for $N$ as an $\mathcal{O}_X$-module. Denote by $\pi \colon  \mathrm{Projan}(\mathcal{R}(N)) \rightarrow X$ the structure morphism. Consider the composition of maps
 
$$\pi^{-1} (T) \hookrightarrow  X \times \mathbb{P}^{g(N)-1} \xrightarrow{pr_2} \mathbb{P}^{g(N)-1}.$$
As $X$ is reduced and $N$ is generically free of rank $e$, the algebra $\mathcal{R}(N) \subset \mathrm{Sym}(F)$ is torsion-free over $\mathcal{O}_X$, so $\mathrm{Projan}(\mathcal{R}(N))$ is of pure dimension $d+k+e-1=k+r$ and none of its irreducible components lies over $T$. Thus $\dim \pi^{-1}(T) \leq k+r-1$. By Kleiman's Transversality Theorem \cite{Kl74}, the intersection of $\pi^{-1}(T)$ with $X \times H_{r}$, where $H_r$ is a general linear subspace of $\mathbb{P}^{g(N)-1}$ of codimension $r$, is of dimension at most $\dim Y - 1$. Denote by $\Gamma^{k}(N)$ the projection of $\mathrm{Projan}(\mathcal{R}(N)) \cap (X \times H_{r})$ to $X$, considered as a germ at $x_0$ (note that $x_0 \in \Gamma^{k}(N)$ if and only if $\pi^{-1}(x_0) \cap (X \times H_r) \neq \emptyset$). This is what Gaffney \cite{GaffP} calls the $k$-{\it dimensional  polar variety} of $N$. Applying Kleiman's theorem to $\pi^{-1}(X_{y_0} \setminus T)$, which has dimension $r$, we see that $\Gamma^{k}(N) \cap X_{y_0}$ is discrete; hence $\Gamma^{k}(N)$ is finite over $(Y,y_0)$. For  $y$ in a Zariski open subset $U$ of $Y$, the fiber  of $\Gamma^{k}(N)$ over $y$ consists of the same number of points, each of them appearing with multiplicity one because $X$ is generically reduced, and because locally at each one of them $N$ is free (see the end of the proof of Thm.\ \ref{MPT} for the precise statement). Denote this number by $\mathrm{deg}_{Y}\Gamma^{k}(N)$. The following is a special case of Gaffney's Multiplicity--Polar Theorem (see \cite{GaffMPT} and \cite[Sct.\ 2]{Rangachev} for generalizations).
 
\begin{theorem}[Gaffney]\label{MPT} Suppose $X$ is Cohen--Macaulay and $T$ is finite over $Y$. Then for each $y$ in a Zariski open subset $U$ of $Y$ we have
$$
e(N(y_0),F(y_0))-
e(N(y),F(y))=\mathrm{deg}_{Y}\Gamma^{k}(N).$$
\end{theorem}
\begin{proof} The argument is the one used in the proof of Thm.\ \ref{main mpt} in Sct.\ \ref{proof mpt}. Note first that $X$ is reduced, being Cohen--Macaulay and generically reduced, and that $\dim X=d+k$; hence $h$ is flat by \cite[Thm.\ 18.16]{Eis95}, and each fiber $X_y$ is Cohen--Macaulay. Let $N'$ be the submodule of $N$ generated by the $r$ generic linear combinations of generators of $N$ that define $H_r$; by \cite[Cor.\ 16.4.7]{Huneke} we may assume that $N'(y_0)$ is a reduction of $N(y_0)$ at $x_0$. Set $\mathcal{I}:=\mathrm{Fitt}_{0}(F/N')$, the ideal of maximal minors of the $e \times r$ matrix $M'$ presenting $F/N'$, and denote by $T'$ the subspace of $X$ defined by $\mathcal{I}$. Set-theoretically $T'=\mathrm{Supp}(F/N')=T \cup T''$ with $T'':=\mathrm{Supp}(N/N')$. Off $T$ the module $N$ is free of rank $e$, so $T'' \setminus T=\Gamma^{k}(N) \setminus T$. Since $\Gamma^{k}(N) \cap X_{y_0}$ is discrete, the germ $T'$ at $x_0$ is finite over $(Y,y_0)$ with $T'_{y_0} \subset \{x_0\}$, and $\dim T' \leq k$. On the other hand every irreducible component of $T'$ has codimension at most $r-e+1=d$ in $X$ by the Eagon--Northcott bound on the height of determinantal ideals. Thus $\mathcal{I}$ has height $d$ at every point of $T'$, and $\mathcal{O}_X/\mathcal{I}$ is Cohen--Macaulay by the Eagon--Northcott theorem (see \cite[Thm.\ A2.10]{Eis95}) because $X$ is. Because $Y$ is smooth and $T' \rightarrow Y$ is finite, $T' \rightarrow Y$ is flat, and so $$\dim_{\mathbb{C}}\mathcal{O}_{X_{y_0}}/\mathcal{I}(y_0)=\dim_{\mathbb{C}}\mathcal{O}_{X_{y}}/\mathcal{I}(y)$$ for each $y \in Y$, where $\mathcal{I}(y)$ is the image of $\mathcal{I}$ in $\mathcal{O}_{X_y}$ and both sides are sums over the points of $T'_{y_0}$ and $T'_y$, respectively. By base change $\mathcal{I}(y)=\mathrm{Fitt}_0(F(y)/N'(y))$. Since $N'(y_0)$ is a reduction of $N(y_0)$ generated by $r=d+e-1$ elements and $\mathcal{O}_{X_{y_0},x_0}$ is Cohen--Macaulay, \cite[Cor.\ 16.5.7]{Huneke} and \cite[4.5, p.\ 223]{Buch} give $e(N(y_0),F(y_0))=e(N'(y_0),F(y_0))=\dim_{\mathbb{C}}\mathcal{O}_{X_{y_0}}/\mathcal{I}(y_0)$. It remains to interpret the right-hand side of the displayed identity for generic $y$.
 
Let $U$ be the complement in $Y$ of the image of $\pi^{-1}(T) \cap (X \times H_r)$. It is a nonempty Zariski open subset of $Y$ because $\dim(\pi^{-1}(T) \cap (X \times H_r)) \leq k-1$ and $T$ is finite over $Y$. For $y \in U$ we have $\Gamma^{k}(N)_y \cap T_y=\emptyset$, hence $T'_y=T_y \sqcup \Gamma^{k}(N)_y$. Moreover, for $y \in U$ and $t \in T_y$ the fiber $\pi^{-1}(t) \cap (X \times H_r)$ is empty and contains $\mathrm{Projan}(\mathcal{R}(N(y)_t)) \cap (X \times H_r)$, so $N'(y)_t$ is a reduction of $N(y)_t$ (see the proof of \cite[Lem.\ 3.8]{BR24}, or \cite[Thm.\ 16.2.3]{Huneke}). As above, $e(N(y)_t,F(y)_t)=\dim_{\mathbb{C}}\mathcal{O}_{X_y,t}/\mathcal{I}(y,t)$, where $\mathcal{I}(y,t)$ is the image of $\mathcal{I}$ in $\mathcal{O}_{X_y,t}$. Finally, after shrinking $U$, every $t \in \Gamma^k(N)_y$ satisfies $\dim_{\mathbb{C}}\mathcal{O}_{X_y,t}/\mathcal{I}(y,t)=1$. Indeed, off $\pi^{-1}(T)$ the space $\mathrm{Projan}(\mathcal{R}(N)) \cap (X \times H_r)$ is reduced of pure dimension $k$ by Kleiman's theorem, since $(X \setminus T) \times \mathbb{P}^{e-1}$ is reduced. The set of points of $T' \setminus T$ over which its fiber is positive dimensional, i.e.\ where $M'$ has rank at most $e-2$, has dimension at most $k-1$, and off this set $\mathrm{Projan}(\mathcal{R}(N)) \cap (X \times H_r)$ maps isomorphically onto the subspace $T' \setminus T$ of $X$. Hence $T' \setminus T$ is reduced over a Zariski open subset of $Y$, and the fibers of the finite map $T' \setminus T \rightarrow Y$ are reduced for $y$ in a Zariski open subset of $Y$. Therefore, $$\dim_{\mathbb{C}}\mathcal{O}_{X_{y}}/\mathcal{I}(y)=\sum_{t \in T_y}\dim_{\mathbb{C}}\mathcal{O}_{X_{y},t}/\mathcal{I}(y,t)+\#\Gamma^k(N)_y=e(N(y),F(y))+\mathrm{deg}_{Y}\Gamma^{k}(N)$$ for $y \in U$, which proves the theorem.
\end{proof}


\begin{thebibliography}{abcdef}

\bibitem[Bas77]{Bas77}
Bassein, R., {\it On smoothable curve singularities: local methods.}
Math.\ Ann.\ {\bf 230},  (1977), 273--277.

\bibitem[BLR26]{BR24}
Bengu\textcommabelow{s}-Lasnier, A.,  Rangachev, A.,
{\it The complete intersection discrepancy of a curve I: Numerical invariants.} With an appendix by Marc Chardin. Accepted in the Transactions of the American Math.\ Society
(2026),
\url{https://arxiv.org/pdf/2504.09362}.


\bibitem[BGG80]{BGG}
Brian\c{c}on, J., Galligo, A., Granger, M., {\it D\'eformations \'equisingulieres
des germes de courbes gauches r\'eduites.}
M\'emoires de la S.\ M.\ F.\ 2e s\'erie, tome 1 (1980), 1--69.


\bibitem[BR64]{Buch}
Buchsbaum, D.A., and Rim, D.S., {\it A generalized Koszul complex. II. Depth
and multiplicity.} Trans. Amer. Math. Soc., {\bf 111} (1964), 197--224.

\bibitem[BG80]{BG80} 
Buchweitz, R., Greuel, G., {\it The Milnor Number and deformations of
complex curve singularities.} Invent. Math. {\bf 58}, (1980), 241--281.



\bibitem[CL06]{CL06} Chiang-Hsieh, H-J., Lipman, J.,
{\it A numerical criterion for simultaneous normalization.}
Duke Math.\ J.\ {\bf 133}, No. 2 (2006), 347--390.

\bibitem[Ch84]{Ch84}
Chiarli, N., {\it A Hurwitz type formula for singular curves.}
C. R. Math. Rep. Acad. Sci. Canada, Vol. {\bf 6} (2), p. 67--72, (1984).

\bibitem[DGPS]{DGPS}
Decker, W., Greuel, G.-M., Pfister, G., Sch{\"o}nemann, H.: 
\newblock {\sc Singular} {4-3-0} --- {A} computer algebra system for polynomial computations.
\newblock {https://www.singular.uni-kl.de} (2022).

\bibitem[E95]{Eis95}
Eisenbud, D., {\it Commutative algebra with a view toward algebraic geometry.}
GTM, {\bf 150}. Springer--Verlag, New York, 1995. 

\bibitem[Fis76]{Fisch}
Fischer, G., {\it Complex analytic geometry},
LNM, {\bf 538}. Springer--Verlag, Heidelberg, 1976.

\bibitem[FT18]{Flores}
Flores, A.G., Teissier, B., {\it Local polar varieties in the 
geometric study of singularities.}
Annales de la Facult\'e des Sciences de Toulouse, {\bf 27}, no.\ 4 (2018), 679--775.

\bibitem[Fri67]{Frisch}
Frisch, J., {\it Points de platitude d’un morphisme
d’espaces analytiques complexes.}
Invent.\ Math.\ {\bf 4} (1967), 118--138.

\bibitem[Ful84]{Ful}
Fulton, W., {\it Intersection Theory.}
Ergebnisse der Mathematik und ihrer Grenzgebiete (3)
[Results in Mathematics and Related Areas (3)], 2. Springer--Verlag, Berlin, 1984.


\bibitem[Gaf07]{GaffMPT}
Gaffney, T., {\it The multiplicity polar theorem.}
arXiv:math/0703650v1 [math.CV]

\bibitem[Gaff04]{GaffP}
Gaffney, T., {\it Polar methods, invariants of pairs of modules,
and equisingularity.}
In Real and complex singularities, eds T. Gaffney and M.A.S. Ruas,
Contemp. Math., {\bf 354}, Amer. Math. Soc., Providence, RI, 2004, 113--135.

\bibitem[Gaf92]{Gaf1}
Gaffney, T., {\it Integral closure of modules and Whitney equisingularity.}
Invent. Math., {\bf 107} (1992), 301--322.

\bibitem[GrR84]{GrR84}
Grauert, H., Remmert, R., {\it Coherent analytic sheaves.}
Springer--Verlag Berlin Heidelberg, 1984.

\bibitem[Gre20]{Gre20}
Greuel, G.M.\, {\it
Deformation and smoothing of singularities.}
In Jos\'e Luis Cisneros Molina, L\^e D\~ung Tr\'ang, and Jos\'e Seade, editors, Handbook of Geometry and Topology of Singularities I, pg.\ 389--448. Springer, Cham, 2020.
		\url{https://doi.org/10.1007/978-3-030-53061-7_7}.
		Preprint version: \url{https://arxiv.org/abs/1903.03661}.

\bibitem[Gre17]{Gre17}
Greuel, G.M., {\it Equisingular and equinormalizable deformations of isolated
non-normal singularities.}
Methods and Applications of Analysis, vol.\ {\bf 24}, no.\ 2 (2017), 215--276.

\bibitem[Gre73]{Greuel73}
Greuel, G.M., {\it Der Gauss-Manin Zusammenhang isolierter Singularit\"aten
von vollständigen Durchschnitten.}
Dissertation. Göttingen, 1973. Math. Ann. {\bf 214} (1975) 235--266. 



\bibitem[HM83]{HM83}
Henry, J.-P., Merle, M., {\it Limites de normales, conditions de Whitney et
eclatement d'Hironaka.} In Singularities, Part 1 (Arcata, Calif., 1981),
575--584, Proc.\ Sympos.\ Pure Math.\ {\bf 40},
Amer.\ Math.\ Soc.\, Providence, RI, 1983.

\bibitem[HIO88]{HIO88}
Hermann, M., Ikeda, S., and Orbanz, U.,  "Equimultiplicity and blowing up. An algebraic study," with an appendix by B.\ Moonen, Springer-Verlag, 1988.

\bibitem[Hir76]{Hir76} 
Hironaka, H., {\it Stratification and Flatness.}
in "Real and Complex singularities, Nordic Summer School, Oslo, 1976,"
Sijthoff and Noordhoff, 1977, 199--265.




\bibitem[Kav04]{Kav04}
Kaveh, K., {\it Morse theory and Euler characteristic of sections of
spherical varieties.}
Transformation Groups {\bf 9} (2004), 47--63.

\bibitem[KT94]{KT-Al}
Kleiman, S., Thorup, A., {\it  A geometric theory of the
Buchsbaum-Rim multiplicity.}
J. Algebra {\bf 167} (1994), 168--231.

\bibitem[Kl74]{Kl74}
Kleiman, S., {\it The transversality of a general translate.}
Compositio Mathematica {\bf 28} (1974), no. 3, 287--297.

\bibitem[Kl17]{Kl17}
Kleiman, S.L., {\it Two formulas for the BR multiplicity.}
Ann. Univ. Ferrara, {\bf 63}, p. 147--158, (2017).



\bibitem[LT88]{LT2}
L\^e, D.\ T., Teissier, B., {\it Limites d'espaces tangent en
g\'eom\'etrie analytique.}
Comment.\ Math.\  Helvetici {\bf 63} (1988), 540--578. 

\bibitem[L\^e74]{Le74}
L\^e D. T., {\it Calculation of Milnor number of isolated singularity
of complete intersection.}
Functional Analysis and Its Applications, {\bf 8} (1974), 127--131.

\bibitem[Lip82]{Lip}
Lipman, J., {\it Equimultiplicity, reduction and blowing-up.}
in “Commutative algebra: analytic methods,”
Dekker Lecture Notes in Pure and Applied Math. {\bf 68} (1982), 111--48.

\bibitem[Loo84]{Loo84}
Looijenga, E.J.N.,\ {\it Isolated singular points on complete intersections.}
London Math.\ Soc.\ Lecture Note Ser., {\bf 77}, Cambridge University Press, 1984. 

\bibitem[Łoj64]{Loj}
Łojasiewicz, S., {\it Triangulation of semi--analytic sets.}
Ann.\ Scuola Norm.\ Sup.\ Pisa (3) {\bf 18} (1964), 449--474.



\bibitem[Mas03]{Massey}
Massey, D., {\it Numerical control over complex analytic singularities.}
Memoirs of the AMS, vol.\ 163, no.\ 778, 2003. 


\bibitem[Mat86]{Matsumura}
Matsumura, H., {\it Commutative ring theory.}
Cambridge University Press, 1986. 

\bibitem[MvS01]{MvS01}
Mond, D. and van Straten, D., {\it Milnor number equals Tjurina number
for functions on space curves.}
J.\ Lond.\ Math.\ Soc.\, II. Ser.\ 63, No.\ 1, (2001) 177--187.

\bibitem[MV90]{MVS}
Montaldi, J., Van Straten, D.,
{\it One forms on singular curves and the topology of real curve singularities.}
(1990), Topology, Vol. 29, no. 4, p. 501--510.

\bibitem[Mum76]{Mumford}
Mumford, D., ``Algebraic Geometry I: Complex Projective Varieties," Springer--Verlag, 1976.
\bibitem[Nob75]{Nob75}
Nobile, A., {\it Some properties of the Nash blowing-up.}
Pacific Journal of Math.\ {\bf 60}, (1975) 297--305.

\bibitem[NOT13]{NOT13}
Nu\~no-Ballesteros, J.J., Or\'efice-Okamoto, B., Tomazella, J.N.,
{\it The vanishing Euler characteristic of an isolated determinantal singularity.}
Israel Journal of Mathematics, {\bf 197} (2013), 475--495. 

\bibitem[NT08]{NT08}
Nu\~no-Ballesteros, J.J., Tomazella, J.N., {\it The Milnor number of a function
on a space curve germ.}
Bull. London Math.\ Soc.\ {\bf 40} (2008), no.\ 1, 129--138.




\bibitem[Pi78]{Pi}
Piene, R., {\it Ideals associated to a desingularization.}
In Lønsted, K. (eds), Algebraic Geometry. Lecture Notes in Mathematics,
vol.\ {\bf 732} (1978), Springer, Berlin, Heidelberg, 503--517.

\bibitem[PR26a]{PR26}
Portilla, P., Rangachev, A., {\it Topological Fibrations and Connectivity of Nearby Analytic Fibres} (in preparation).


\bibitem[PR26b]{PR25}
Portilla, P., Rangachev, A., {\it The Milnor fiber of a smoothable curve}
(in preparation).


\bibitem[Ram73]{Ramanujam}
Ramanujam, C. P., {\it On a geometric interpretation of multiplicity.}
Invent. Math. {\bf 22} (1973), 63--67.

\bibitem[RT26]{RT24}
Rangachev, A., Teissier, B., {\it A Plücker formula for curves.} (in preparation).

\bibitem[Ran22]{Rangachev}
Rangachev, A., {\it Local volumes, equisingularity and generalized smoothability.}
\\
\url{https://arxiv.org/abs/2105.08749}. 




\bibitem[Stks]{Stacks}
The {Stacks Project Authors}, {\itshape Stacks Project.} \url{http://stacks.math.columbia.edu}, 2026.

\bibitem[SH06]{Huneke}
Swanson, I., and Huneke, C., {\it Integral closure of ideals, rings, and modules.}
London Mathematical Society Lecture Note Series, vol. 336,
Cambridge University Press, Cambridge, 2006.



\bibitem[T81]{Teissier}
Teissier, B., {\it Vari\'{e}t\'{e}s polaires II: Multiplicit\'{e} polaires,
sections planes, et conditions de Whitney.}
Actes de la conf\'{e}rence de g\'{e}ometrie alg\'{e}brique \'{a} la R\'{a}bida,
Springer Lecture Notes, {\bf 961} (1981), 314--491.

\bibitem[T80]{T80}
Teissier, B., {\it Resolution simultanee I, II.}
Seminaire sur les Singularites des Surfaces, ed. Demazure et al.,
Lecture Notes Mathematics, {\bf 777}, Springer--Verlag, Berlin 1980.

\bibitem[T77]{T77}
Teissier, B.,{\it The hunting of invariants in the geometry of discriminants.}
Real and complex singularities
(Proc.\ Ninth Nordic Summer School NAVF Sympos.\ Math.\, Oslo, 1976),
Sijthoff and Noordhoff, Alphen aan den Rijn (1977), 565--678.

\bibitem[T73]{T73}
B. Teissier, {\it  Cycles \'evanescents, sections planes et conditions de Whitney.}
(French) Singularit\'es \`a Carg\'ese (Rencontre Singularit\'es
en G\'eom\'etrie Analytique, Institut d'\'etudes Scientifiques de
Carg\'ese, 1972), pp. 285--362.
Asterisque, Nos. 7 et 8, Soc. Math. France, Paris, 1973.
\end{thebibliography}
\end{document}